\documentclass[10pt]{article}
\usepackage{amsmath,amssymb}
\usepackage{mathrsfs}
\usepackage{graphicx}  
\usepackage{hyperref}
\usepackage{subfigure}
\usepackage{float}
\usepackage{color} 
\usepackage{hyperref}
\usepackage{titlesec}
\usepackage{listings}
\usepackage{algorithm}
\usepackage{algpseudocode}
\usepackage{algpseudocode}   %
\usepackage{subeqnarray}
\usepackage{cases}

\usepackage{graphicx} 
\usepackage{makecell} 
\usepackage{multirow}  

\titleformat{\section}
{\raggedright\large\bfseries}
{\thesection .\quad}
{0pt}
{}

\titleformat{\subsection}
{\raggedright\normalsize\bfseries}
{\thesection .}
{0pt}
{}
\makeatletter
\@addtoreset{equation}{section}
\@addtoreset{table}{section}
\makeatother

\usepackage{amsthm}

\newtheorem{thm}{Theorem}[section]
\newtheorem{theorem}{Theorem}[section]

\newtheorem{lemma}[thm]{Lemma}
\newtheorem{proposition}[thm]{Proposition}

\newtheorem{remark}[thm]{Remark}
\newtheorem{example}[thm]{Example}
\newtheorem{assumption}[thm]{Assumption} 

\numberwithin{equation}{section}
\numberwithin{table}{section}
\numberwithin{figure}{section}

\allowdisplaybreaks[4]
\makeatletter\@addtoreset{equation}{section}

\usepackage{array}

\usepackage{booktabs}

\title{Finite Difference Method for Stochastic Cahn-Hilliard Equation Driven by A Fractional Brownian Sheet}
 
\author{Nan Deng\thanks{Email: deng\_nan@seu.edu.cn},\ $\ $ Wanrong Cao\thanks{Corresponding author. Email address: wrcao@seu.edu.cn. Tel/Fax: +86 25 52090590}\\\\School of Mathematics, Southeast University, Nanjing 210096, P.R.China.} 
\date{}
\begin{document}
\maketitle	

\begin{abstract} 
	The stochastic Cahn–Hilliard equation driven by a fractional Brownian sheet provides a more accurate model for correlated space-time random perturbations. This study delves into two key aspects: first, it rigorously examines the regularity of the mild solution to the stochastic Cahn–Hilliard equation, shedding light on the intricate behavior of solutions under such complex perturbations. Second, it introduces a fully discrete numerical scheme designed to solve the equation effectively. This scheme integrates the finite difference method for spatial discretization with the tamed exponential Euler method for temporal discretization. The analysis demonstrates that the proposed scheme achieves a strong convergence rate of $O\big(h^{1-\epsilon}+\tau^{H_1-\frac{1}{8}-\frac{\epsilon}{2}}\big)$, where $\epsilon$ is an arbitrarily small positive constant, providing a solid foundation for the numerical treatment of such equations.
	
\end{abstract}

\section{Introduction}
\label{sec;introduction}
The Cahn–Hilliard equation has found wide application across various scientific domains, including spinodal decomposition \cite{Cahn1961}, diblock copolymer \cite{Jeong2014},
image inpainting \cite{Bertozzi2006}, multiphase fluid flows \cite{badalassi2003computation}, microstructures with elastic inhomogeneity \cite{zaeem2012effects},  tumor growth simulation \cite{wise2008three}, etc. Particularly in materials science, the equation has significantly advanced our understanding of material structures and facilitated the design of new materials \cite{ Moelans2008,Zhou2007Multimaterial}.\par

Recognizing the influence of thermal fluctuations, Cook introduced space-time white noise into the Cahn–Hilliard equation  \cite{cook1970brownian}, which has since sparked considerable interest in its stochastic variant (see, e.g., \cite{antonopoulou2016existence,cui2020absolute,cui2022wellposedness,da1996stochastic,elezovic1991stochastic,qi2022strong}). Langer further emphasized the importance of accounting for spatial correlations in the noise incorporated into the stochastic Cahn–Hilliard equation (SCHE) \cite{langer1971theory}. In practical experiments, external noise is inevitable, regardless of the precision of the equipment used. Therefore, it is crucial to consider the spatio-temporal structure of fluctuations. Stochastic models driven by fractional Brownian sheets offer a robust framework for this purpose, as they allow for the precise adjustment of fluctuation correlations through manipulation of the Hurst parameter.

In this work, we consider the following SCHE driven by a fractional Brownian sheet
\begin{subequations}\label{eq:CH_FractionalNoise}
	\begin{numcases}{}
		\frac{\partial u}{\partial t}+\Delta^2 u=\Delta f(u)+\sigma \frac{\partial^2B^H}{\partial t \partial x}, \ \text { in }[0, T] \times \mathcal{O}, \label{eq:CH_FractionalNoise a}\\ 
		\frac{\partial u}{\partial n}=\frac{\partial \Delta u}{\partial n}=0,\qquad\qquad \ \quad  \quad  \text { on }[0, T] \times \partial \mathcal{O}, \label{eq:CH_FractionalNoise b} \\  
		u(0, \cdot)=u_0, \label{eq:CH_FractionalNoise c}
	\end{numcases}
\end{subequations} 
where $T>0,\ \mathcal{O}:=(0,\pi)$ and $\left\{B^H(t,x)\right\}_{ (t,x)\in[0, T] \times \mathcal{O}} $ denotes a fractional Brownian sheet with anisotropic Hurst parameters $H=(H_1,H_2) \in[\frac{1}{2},1)\times[\frac{1}{2},1)$. The nonlinear function $f$ is a polynomial of degree three with a positive leading coefficient, i.e., $f(x)=a_0x^3+a_1x^2+a_2x+a_3$ with $a_0\in \mathbb{R}^+,\ a_1,\ a_2,\ a_3\in\mathbb{R}$.  \par 

Given the inherent challenges in obtaining explicit solutions for the stochastic Cahn–Hilliard equation (SCHE), the development of reliable numerical methods is essential. Significant progress has been made in numerical approaches for SCHEs driven by Q-Wiener processes or space-time white noise. For the additive noise case, Kov{\'a}cs \emph{et al.} \cite{kovacs2011finite} and Furihata \emph{et al.} \cite{furihata2018strong} established strong convergence for semi-discrete and fully discrete finite element methods, respectively. Building on this, Qi and Wang \cite{Qi2020} provided rigorous strong convergence rates for a fully discrete finite element scheme. Utilizing the spectral Galerkin method for spatial discretization, subsequent studies explored various time discretization techniques, including the accelerated implicit Euler method \cite{Cui2021}, the backward Euler method \cite{qi2022strong}, and the tamed exponential Euler method \cite{cai2023strong}, further refining the strong convergence rates for fully discrete schemes. Cardon-Weber introduced a finite difference scheme, establishing its convergence in probability with specific rates \cite{Cardon2000}. Hutzenthaler and Jentzen employed general perturbation theory and leveraged exponential integrability properties of both exact and numerical solutions to demonstrate a strong convergence rate for spatial spectral Galerkin approximations \cite{Hutzenthaler2014}. More recently, Cai \emph{et al.} achieved weak convergence for a fully discrete scheme \cite{cai2023weak}. For the multiplicative noise case, Hong \emph{et al.} presented strong convergence rates and density convergence results for the finite difference method \cite{Hong2022,Hong2022b}. Additionally, Feng \emph{et al.} derived strong convergence rates for an implicit, fully discrete mixed finite element method applied to the SCHE with gradient-type multiplicative noise \cite{Feng2020}. 

However, research on SCHEs driven by fractional Brownian sheets remains relatively sparse. Bo \emph{et al.} explored SCHEs with fractional noise (fractional in time and white in space), establishing existence, uniqueness, and regularity of solutions \cite{BoLijun2008}. Furthermore, Gregory \emph{et al.} demonstrated that the law of the solution for the SCHE with small perturbations satisfies the large deviation principle in the H\"{o}lder norm \cite{Gregory2023}. Most recently, Arezoomandan and Soheili examined the finite element approximation of a linearized SCHE driven by fractional Brownian motion, albeit without the nonlinear term \cite{Arezoomandan2024}.

This paper has two primary objectives. The first is to establish the regularity of the mild solution to the SCHE driven by a fractional Brownian sheet \eqref{eq:CH_FractionalNoise}. To accomplish this, we investigate the regularity of the stochastic convolution $o(t,x)$ under three different cases, categorized by the range of the Hurst parameters: $4H_1 + H_2 - 1 \leq 2$, $2 < 4H_1 + H_2 - 1 \leq 3$, and $3 < 4H_1 + H_2 - 1 < 4$. Through this analysis, we demonstrate that the mild solution to \eqref{eq:CH_FractionalNoise} exhibits the following regularity for $p \geq 1$ and $0 \leq \beta < 4H_1 - H_2 - 1$ (Theorem \ref{th: u_regularity}): 
\begin{equation*} \sup_{t \in [0, T]} \mathbb{E}\left[|u(t, \cdot)|^p_{\dot{H}^{\beta}}\right] < \infty. 
\end{equation*}
The second objective of this paper is to propose and analyze a fully discrete numerical scheme for solving the SCHE \eqref{eq:CH_FractionalNoise}. Building on the frameworks established by \cite{BoLijun2008} and \cite{Hong2022}, we employ a finite difference method for spatial semi-discretization, coupled with a tamed exponential Euler scheme for temporal discretization. In Theorem \ref{thm: ||u^{M,N}-u||}, we establish the strong convergence rate of the fully discrete scheme: 
\begin{equation*}
	\sup_{t \in[0, T]}\left\|u(t,\cdot)-u^{M,N}(t, \cdot)\right\|_{L^p(\Omega,L^\infty)}   \leq C\left(h^{1-\epsilon}+ \tau^{H_1-\frac{1}{8}-\frac{\epsilon}{2}} \right),  
\end{equation*}
where $\epsilon$ is an arbitrary small positive number, and $h$ and $\tau$ represent the spatial and temporal step size, respectively. 

We note that the noise considered in our study includes the case of white noise by setting $H_1 = H_2 = \frac{1}{2}$. In this scenario, Cai \emph{et al.} \cite{cai2023strong} and Qi \emph{et al.} \cite{qi2022strong} developed numerical schemes that achieved a convergence rate of $3/8$ for temporal discretization in the $L^2$-norm. Additionally, Cui \emph{et al.} \cite{Cui2021} presented a numerical scheme with a convergence rate of $3/4$ in the $L^2$-norm for temporal discretization. In comparison, our approach achieves a convergence rate of $3/8$ for temporal discretization in the more stringent $L^\infty$-norm.

Our analysis encounters three significant challenges.

Firstly, the technical complexity arises from demonstrating how the anisotropic Hurst parameters influence the regularity of the stochastic convolution.  Specifically, in Lemma \ref{lem: 3.1 regularity of o(t,x)}, we establish that 
\begin{equation*}
	\mathbb{E}\left[\|o(t, \cdot)-o(s, \cdot)\|^p_{\dot{H}^{\beta}}\right]\leq C(t-s)^{ \left(\frac{4H_1+H_2-1-\beta}{4}
		-\frac{\epsilon}{2} \right) p},  
\end{equation*}
where $0\leq s<t\leq T$, $p \geq 1$, and $0\leq\beta<4H_1+H_2-1$.

Secondly, unlike the Dirichlet boundary conditions considered in \cite{Hong2022, Hong2022b}, the Neumann boundary condition in the SCHE \eqref{eq:CH_FractionalNoise}  renders the discrete Laplacian matrix $A_N$ non-invertible. Consequently, we must take the inner product of the differential equation with the eigenfunctions of $A_N$ and isolate the component corresponding to the first eigenfunction for further detailed analysis.

Lastly, the presence of the unbounded operator $A_N$ preceding the nonlinear term in the numerical scheme significantly complicates the error analysis. Typically, error analysis for such problems relies on the Taylor expansion of the nonlinear term \cite{cai2023strong, Cui2021, Qi2020}. However, in the context of the SCHE \eqref{eq:CH_FractionalNoise}, this approach becomes challenging due to the necessity of $L^p$-estimates on a complex double integral related to the fractional Brownian sheet. Moreover, the first derivative of the nonlinear term in the integrand introduces additional complexities. To address this issue, Lemma \ref{lem: F local Lipschitz} establishes the local Lipschitz continuity and polynomial growth of $F_N$ under the semi-norm $|\cdot|_{k,N}$ for $k = 1, 2, 3$, which is crucial for managing the nonlinear term.

The remainder of this paper is structured as follows. Section \ref{section2} introduces the necessary notations and assumptions. In Section \ref{section3}, we present the regularity of the mild solution to the SCHE \eqref{eq:CH_FractionalNoise}. Section \ref{section4} is devoted to analyzing the strong error estimates of the spatial semi-discretization. In Section \ref{section5}, we propose a fully discrete scheme and examine its strong convergence rate. Finally, Section \ref{section6} provides a numerical example to validate our theoretical findings. 

\section{Preliminaries}\label{section2}  
Define $\mathcal{C}^k(\mathcal{O})$ as the space of $k$-times continuously differentiable functions on $\mathcal{O}$ for $k \in \mathbb{N}$. When $k=0$, $\mathcal{C}^0(\mathcal{O})=\mathcal{C}(\mathcal{O})$, which represents the space of all continuous functions on $\mathcal{O}$. For $d \geq 1$, we denote  the Euclidean norm and inner product in $\mathbb{R}^d$ by $|\cdot|$ and $\langle\cdot, \cdot\rangle$, respectively. For $1\leq p<\infty$, let $L^p(\mathcal{O}):=L^p(\mathcal{O},\mathbb{R})$ denote the space of all $\mathbb{R}$-valued, $p$-times integrable functions, with the associated inner product and norm denoted by $\langle \cdot,\cdot \rangle_{L^p}$ and $\|\cdot\|_{L^p}$.  Additionally, for any Banach space $\left(\mathcal{H},\|\cdot\|_\mathcal{H}\right)$, we denote by $L^p(\Omega,\mathcal{H})$ the space of all $\mathcal{H}$-valued, $L^p$ integrable random variables, equipped with the norm $\left\|\cdot\right\|_{L^p(\Omega,\mathcal{H})}:=\big(\mathbb{E}\big[\|\cdot \|^p_\mathcal{H}\big]\big)^\frac{1}{p}$.  Throughout this paper, $C$ denotes a generic positive constant that is independent of the discretization parameters and may change from line to line. \par 
For Hurst parameter $H=(H_1,H_2) \in [\frac{1}{2},1)\times[\frac{1}{2},1)$, let $B^H=\left\{ B^H(t, x)\right\}_{(t, x) \in[0, T] \times \mathcal{O}}$ be a fractional Brownian sheet with Hurst parameter $H$, which is a centered Gaussian process defined on $\left(\Omega ,\mathcal{F},\mathbb{P}\right)$ with the covariance function given by
$$
\mathbb{E}\left[B^H(t,x) B^H(s, y)\right]=\frac{1}{4}\left( t^{2H_1}+s^{2H_1}-|t-s|^{2H_1} \right) 
\left( x^{2H_2}+y^{2H_2}-|x-y|^{2H_2} \right), 
$$
for $s, t \in[0, T]$ and $x,y \in \mathcal{O}$. \par 
In \cite{Biagini2008}, \cite{Hu2000} and \cite{Oksendal2001},  the stochastic integral related to the fractional Brownian sheet has been defined. Due to the complexity of this definition, the details are omitted here.  However, note that the following estimates are directly derived from this definition and are crucial for our study:
for any $p\geq2$, $g \in L^{1/{H_1}}\left([0,T], L^{1/{H_2}}\left(\mathcal{O}\right)\right)$, $g_1  \in L^{1/{H_1}}\left([0,T]\right)$ and $g_2 \in L^{1/{H_2}}\left(\mathcal{O}\right)$,      it is not different to check 
\begin{align}
	&\bigg\|\int_{s}^{t} \int_{x}^y g(r,z)  B^H(\mathrm{d} r, \mathrm{d} z)\bigg\|_{L^p(\Omega,\mathbb{R})}^2\leq  C \bigg(\int_{s}^{t}\bigg(\int_{x}^y |g(r,z)|^\frac{1}{H_2} \mathrm{d} z\bigg)^\frac{H_2}{H_1}\mathrm{d} r\bigg)^{2H_1} , \label{ineq: Ito Isometry 2} \\
	&\bigg\|\int_{s}^{t} \int_{x}^y g_1  (r)g_2 (z)   B^H(\mathrm{d} r, \mathrm{d} z)\bigg\|_{L^p(\Omega,\mathbb{R})}^2\leq  C  \mathcal{I}(H_1,g_1  ,s,t)\mathcal{I}(H_2,g_2,x,y) ,  \label{ineq: Ito Isometry}
\end{align}
where $\mathcal{I}$ is defined by 
\begin{equation*}  
	\mathcal{I}(\tilde{H},\tilde{g},\tilde{x},\tilde{y})=
	\left\{
	\begin{aligned} 
		&\int_{\tilde{x}}^{\tilde{y}}|\tilde{g}(z) |^2\mathrm{d} z ,\qquad\qquad\qquad\qquad\qquad\qquad \tilde{H}=\frac{1}{2}, \\  
		& \int_{\tilde{x}}^{\tilde{y}}\int_{\tilde{x}}^{\tilde{y}}  \tilde{g}\left(z_1\right)\tilde{g}\left(z_2\right) \left|z_1-z_2\right|^{2\tilde{H}-2}\mathrm{d} z_1\mathrm{d} z_2,\ \frac{1}{2}<\tilde{H}<1,
	\end{aligned} 
	\right. 
\end{equation*} 
with $ 0\leq \tilde{x} \leq \tilde{y}$ and $\tilde{g}\in L^{1/{\tilde{H}}}\left([\tilde{x},\tilde{y}]\right).$\par 

Below, we provide the properties of $f$. Due to the polynomial assumption,
there exits  a constant $C>0$ such that 
\begin{align}
	& (y-x)( f(x)-f(y))  \leq C|x-y|^2,\quad\quad\quad \forall\ x,\ y \in \mathbb{R}, \label{es: f}\\
	&|f(x)-f(y)|\leq C(1+x^2+y^2)|x-y|,\quad \forall\ x,\ y \in \mathbb{R}.  \label{es: f 2}
\end{align}

Finally the assumption regarding the initial function  $u_0$ is specified.
\begin{assumption}\label{assu:2}
	Let the initial function satisfy $u_0\in \mathcal{C}^5(\mathcal{O}).$ 
\end{assumption}

\section{Regularity of the mild solution to the stochastic Cahn-Hilliard equation}\label{section3}
This section is devoted to the  regularity of the mild solution to the SCHE \eqref{eq:CH_FractionalNoise}.
As in \cite{da1996stochastic}, the Green function associated with the operator $\partial_t+\Delta^2$ with the homogeneous Neumann boundary conditions for the SCHE \eqref{eq:CH_FractionalNoise} is given by 
\begin{equation}  \label{eq:the series decompositionof $G$}
	G_t(x, y)=\sum_{j=0}^{\infty} e^{-\lambda_j^2 t} \phi_j(x) \phi_j(y),\quad \forall\ t \in[0, T],\ x,\ y \in \mathcal{O}, 
\end{equation}
where $\lambda_j=j^2$ and  $\phi_j(x)=\left\{\begin{array}{ll}
	\sqrt{1/ \pi}, & j=0, \\
	\sqrt{2/ \pi} \cos (j x), & j>0,
\end{array}\right.$  in which $\left\{\phi_j\right\}_{j \geq 0}$ forms an orthonormal basis of $\dot{H}(\mathcal{O}):=\left\{v\in L^2(\mathcal{O}) : \frac{\partial v}{\partial n}\big|_{x\in\partial \mathcal{O} }=0 \right\}$. For any $\beta>0$,  we define the following Banach space
$$\dot{H}^\beta(\mathcal{O}):=\bigg\{ v\in L^2(\mathcal{O}) :    \sum_{j=0}^{\infty} \lambda_j^\beta\left|\left\langle v, \phi_j\right\rangle_{L^2}\right|^2 <\infty \bigg\}$$ 
with the norm $\|\cdot\|_{\dot{H}^\beta}$, which is given by
$
\|\cdot\|_{\dot{H}^{\beta}}=\left(\sum_{j=0}^{\infty} \lambda_j^\beta\left|\left\langle\cdot\ , \phi_j\right\rangle_{L^2}\right|^2\right)^{\frac{1}{2}}.
$
The  Green function  $G_t(x, y)$ possesses the following property, as demonstrated in \cite[Lemma 1.2 \& Lemma 1.8]{Cardon-Weber2001} and \cite[Lemma 3.1]{Cardon2000}.
\begin{lemma}   
	For any $0<\alpha<1$, there exist  constants $C>0$ and $c>0$ such that 
	\begin{align} 
		&\left|\Delta G_t(x, y)\right|  \leq \frac{C}{t^{3 / 4}} \exp \left(-c \frac{|x-y|^{4 / 3}}{|t|^{1 / 3}}\right) ,\qquad\qquad\qquad\quad \ \forall\ t\in[0,T],\ x,\ y\in\mathcal{O},  \label{es: Delta G_t(x,y)} \\
		&\int_0^t  \int_\mathcal{O}|G_{t-s}(x, y)-G_{t-s}(z, y)|^2 \mathrm{d} y\mathrm{d} s \leq C|x-z|^2,\quad\quad \ \ \forall\ t\in[0,T],\ x,\ z\in\mathcal{O},  \label{es:  G_t(x,y)-G_t(z,y)}\\
		&\int_0^t  \int_\mathcal{O}|\Delta G_{t-s}(x, y)-\Delta G_{t-s}(z, y)|  \mathrm{d} y\mathrm{d} s \leq C|x-z|^\alpha,\quad\ \forall\ t\in[0,T],\ x,\ z\in\mathcal{O}.   \label{es:  DeltaG_t(x,y)-DeltaG_t(z,y)}
	\end{align} 
\end{lemma}

The well-posedness of the SCHE driven by time-space white noise and fractional noise (fractional in time and white in space) has been rigorously established, see \cite[Theorem 1.3]{Cardon-Weber2001} and \cite[Theorem 1.1]{BoLijun2008}.
Employing similar arguments and utilizing estimates \eqref{ineq: Ito Isometry 2} and \eqref{ineq: Ito Isometry}, we can give the following theorem directly.
\begin{theorem}  \label{th: 2.1}
	Suppose that $u_0 \in L^p(\mathcal{O})$ for some $p>3$, then the SCHE \eqref{eq:CH_FractionalNoise} has a 
	unique mild solution $u\in \mathcal{C}\left([0, T], L^p(\mathcal{O})\right) $, which satisfies
	\begin{equation}\label{eq: CH_FractionalNoise_solution_mild}
		\begin{aligned}
			u(t, x)= & \int_{\mathcal{O}} G_t(x, y) u_0\left(y\right) \mathrm{d} y+\int_0^t \int_{\mathcal{O}} \Delta G_{t-s}(x, y) f\left(u\left(s, y\right)\right) \mathrm{d} y \mathrm{d} s  +\sigma\int_0^t \int_{\mathcal{O}} G_{t-s}(x, y) \  B^H(\mathrm{d} s, \mathrm{d} y)\\
			= &: u_1(t,x)+u_f(t,x)+o(t,x),\ \qquad\qquad \qquad\qquad\qquad \quad\qquad\qquad \qquad\forall\ (t, x) \in[0, T] \times \mathcal{O}.
		\end{aligned}
	\end{equation}
	Furthermore, for $\rho\in[p,\infty)$, one can have
	\begin{equation}\label{eq: CH_FractionalNoise_solution_mild  boundness}
		\begin{aligned}
			\mathbb{E}\left[\sup_{t\in[0,T]}\|u(t, \cdot)\|_{L^p}^\rho\right]<\infty.
		\end{aligned}
	\end{equation}
\end{theorem}

To investigate the regularity of the mild solution $u$, we  begin with studying regularity of the stochastic convolution $o(t,x)$.
\begin{lemma}  \label{lem: 3.1 regularity of o(t,x)}
	For any $p \geq 1$ and $0\leq\beta<4H_1+H_2-1$, the stochastic convolution $o(t,x)$   enjoys
	the following regularity property
	\begin{equation}\label{ineq: regularity of o(t,x)}
		\mathbb{E}\left[\sup_{t\in[0,T]} \|o(t, \cdot)\|^p_{\dot{H}^{\beta}}\right] < \infty.
	\end{equation}
	Moreover, there exists a constant $C>0$ such that, for any $0\leq s<t\leq T$,
	\begin{equation}\label{ineq: continuous of o(t,x)}
		\mathbb{E}\left[\|o(t, \cdot)-o(s, \cdot)\|^p_{\dot{H}^{\beta}}\right]\leq C(t-s)^{ \left(\frac{4H_1+H_2-1-\beta}{4}
			-\frac{\epsilon}{2} \right) p},
	\end{equation}
	where $\epsilon$ is an arbitrary small positive number, and for any $x,\ z \in \mathcal{O}$,  
	\begin{equation}\label{ineq: futher property of O_t}
		\sup_{t\in[0,T]}\mathbb{E}\left[|o(t, x)-o(t, z)|^p\right]\leq C|x-z|^p. 
	\end{equation} 
\end{lemma}
\begin{proof}
	Utilizing  the orthonormality of $\left\{\phi_j\right\}_{j\geq 0} $ and \eqref{ineq: Ito Isometry},  we obtain, for any $p \geq 1$ and  $0\leq\beta<4H_1+H_2-1$,
 
		\begin{align}
			& \|o(t,\cdot)-o(s,\cdot)\|_{L^{2p}(\Omega,\dot{H}^\beta)} \nonumber\\
			= & 	\Big\|\sum_{j=0}^{\infty}\lambda_j^\beta \sigma^2\Big\langle\int_0^s \int_{\mathcal{O}} \Big(G_{t-r}(\cdot, y)-G_{s-r}(\cdot, y)\Big) B^H(\mathrm{d} r, \mathrm{d} y)+\int_s^t \int_{\mathcal{O}} G_{t-r}(\cdot, y)  B^H(\mathrm{d} r, \mathrm{d} y),\phi_j(\cdot) \Big\rangle^2_{L^2 } 	\Big\|_{L^p(\Omega,\mathbb{R})}^\frac{1}{2} \nonumber\\
			=&\Big\|\sum_{j=0}^{\infty}\lambda_j^\beta\sigma^2 \Big| \int_0^s \int_{\mathcal{O}}   \Big(e^{-\lambda_j^2 (t-r)}- e^{-\lambda_j^2 (s-r)}\Big)  \phi_j(y)  B^H(\mathrm{d} r, \mathrm{d} y) +\int_s^t \int_{\mathcal{O}}   e^{-\lambda_j^2 (t-r)}   \phi_j(y)  B^H(\mathrm{d} r, \mathrm{d} y) \Big|^2	\Big\|_{L^p(\Omega,\mathbb{R})}^\frac{1}{2} \nonumber \\
			\leq &C\Big(  \sum_{j=0}^{\infty}\lambda_j^\beta \Big(\big(1-e^{-\lambda_j^2(t-s)}\big)^2 \Big\|  \int_0^s \int_{\mathcal{O}}   g_{1,j}(r)  \phi_j(y)  B^H(\mathrm{d} r, \mathrm{d} y)\Big\|_{L^{2p}(\Omega,\mathbb{R})}^2  \nonumber\\
			&\qquad +	\Big\|  \int_s^t \int_{\mathcal{O}}   g_{2,j}(r)    \phi_j(y)  B^H(\mathrm{d} r, \mathrm{d} y) \Big\|_{L^{2p}(\Omega,\mathbb{R})}^2 \Big)\Big)^\frac{1}{2}   \nonumber \\
			\leq &C\Big( (t-s)^{2H_1} + \sum_{j=1}^{\infty}\lambda_j^\beta\Big( \big(1-e^{-\lambda_j^2(t-s)}\big)^2 
			\mathcal{I}(H_1,g_{1,j},0,s) \mathcal{I}(H_2,\phi_j,0,\pi)  + \mathcal{I}(H_1,g_{2,j},s,t) \mathcal{I}(H_2,\phi_j,0,\pi)\Big) 
			\Big)^\frac{1}{2} ,   \label{regularity of o(t,x)  1}
		\end{align} 
	where $g_{1,j}(r):=e^{-\lambda_j^2(s-r)}$ and $g_{2,j}(r):=e^{-\lambda_j^2(t-r)}$ for $r\in[0,T]$.
	\par  
	It is easy to check that, for any $\alpha\in[0,\infty)$ and $\hat{\alpha}\in[0,1)$, there exist  a constant $C$ such that
	\begin{equation} \label{estimate: e^{-x} 2}
		e^{-x} \leq C  x^{-\alpha}, \quad \forall\ x> 0,
	\end{equation}  
	and
	\begin{equation} \label{estimate: e^{-x} 1}
		1-	e^{-x} \leq C  x^{\hat{\alpha}}, \quad \forall\ x> 0.
	\end{equation}  
	When $H_1=\frac{1}{2}$, by \eqref{estimate: e^{-x} 1} one has
	\begin{equation}  \label{regularity of o(t,x)  1.1} 
		\begin{aligned}
			\mathcal{I}(H_1,g_{1,j},0,s)= &  \int_0^{s} e^{-2\lambda_j^2 (s-r)}\mathrm{d}r=\frac{1}{2\lambda_j^2}\big(1-e^{-2\lambda_j^2s}\big)\leq C\lambda_j^{-2+2\hat{\alpha}_1}s^{\hat{\alpha}_1},\ \forall\ \hat{\alpha}_1\in[0,1).
		\end{aligned}
	\end{equation}
	When $H_1>\frac{1}{2}$, taking $r_1=sr_4$ and $r_2=r_1r_3$ and using \eqref{estimate: e^{-x} 2}, there exist $\alpha_1,\ \alpha_2\geq 0$ such that  
	\begin{align*}
		\mathcal{I}(H_1,g_{1,j},0,s) 
		= & 2 \int_0^s\int_0^{r_1} e^{-\lambda_j^2 (s-r_1)} e^{-\lambda_j^2 (s-r_2)} (r_1-r_2)^{2H_1-2}\mathrm{d}r_2 \mathrm{d} r_1\\
		= & 2 \int_0^1\int_0^{1} e^{-\lambda_j^2 s(1-r_4)} e^{-\lambda_j^2 s(1-r_3r_4)} s^{2H_1}r_4^{2H_1-1}(1-r_3)^{2H_1-2}\mathrm{d}r_3 \mathrm{d} r_4\\
		\leq & C\lambda_j^{-2(\alpha_1+\alpha_2)}s^{2H_1-(\alpha_1+\alpha_2)} \int_0^1\int_0^{1} (1-r_4)^{-\alpha_1}(1-r_3r_4)^{-\alpha_2} r_4^{2H_1-1}(1-r_3)^{2H_1-2}\mathrm{d}r_3 \mathrm{d} r_4.
	\end{align*} 
	Note that when taking $\alpha_1\in[0,1)$ and $\alpha_2\in[0,2H_1-1)$, we have
	\begin{equation}\label{regularity of o(t,x)  1.2} 
		\begin{aligned}
			\mathcal{I}(H_1,g_{1,j},0,s) \leq & C\lambda_j^{-2(\alpha_1+\alpha_2)}s^{2H_1-(\alpha_1+\alpha_2)}.
		\end{aligned}
	\end{equation}
	Likewise, we can obtain 
	\begin{equation}\label{regularity of o(t,x)  1.3} 
		\begin{aligned}
			\mathcal{I}(H_1,g_{2,j},s,t) \leq 
			\left\{\begin{array}{ll}
				C\lambda_j^{-2+2\hat{\alpha}_2}(t-s)^{\hat{\alpha}_2},  &H_1=\frac{1}{2},\  \hat{\alpha}_2\in[0,1), \\
				C\lambda_j^{-2(\alpha_3+\alpha_4)}(t-s)^{2H_1-(\alpha_3+\alpha_4)},  &  H_1 \in (\frac{1}{2},1),\ \alpha_3\in[0,1),\ \alpha_4\in[0,2H_1-1).
			\end{array}\right. 
		\end{aligned}
	\end{equation}
	As to the term $\mathcal{I}(H_2,\phi_j,0,\pi)$, when $H_2=\frac{1}{2}$ it is easy to check
	\begin{equation}\label{regularity of o(t,x)  1.4} 
		\begin{aligned}
			\mathcal{I}(H_2,\phi_j,0,\pi)\leq C.
		\end{aligned}
	\end{equation}
	When $H_2>\frac{1}{2}$, we use the fact that
	$ \int_{\mathcal{O}} \int_{\mathcal{O}} \sin(jz_1)\sin(jz_2)|z_1-z_2|^{2H_2-2}\mathrm{d}z_2 \mathrm{d}z_1\geq0 $
	to get
	\begin{equation*} 
		\begin{aligned}
			\mathcal{I}(H_2,\phi_j,0,\pi) 
			\leq&  \int_{\mathcal{O}} \int_{\mathcal{O}} \cos(jz_1)\cos(jz_2)|z_1-z_2|^{2H_2-2}\mathrm{d}z_2 \mathrm{d}z_1+\int_{\mathcal{O}} \int_{\mathcal{O}} \sin(jz_1)\sin(jz_2)|z_1-z_2|^{2H_2-2}\mathrm{d}z_2 \mathrm{d}z_1 \\
			=& \int_{\mathcal{O}} \int_{\mathcal{O}} \cos(j|z_1-z_2|) |z_1-z_2|^{2H_2-2}\mathrm{d}z_2 \mathrm{d}z_1.
		\end{aligned}
	\end{equation*}
	By letting $z_3=\frac{z_1-z_2}{2}$, $z_4=\frac{z_1+z_2}{2}$ and $z_5=jz_3$, it follows from integration by parts and H\"{o}lder's inequality that
	\begin{equation}  \label{regularity of o(t,x)  1.5} 
		\begin{aligned}
			\mathcal{I}(H_2,\phi_j,0,\pi) 
			\leq&C\int_{0}^\pi \int_{-\frac{\pi}{2}}^{\frac{\pi}{2}} \cos(2j|z_3|) |z_3|^{2H_2-2}\mathrm{d}z_3\mathrm{d}z_4\\
			\leq&C  \int_{0}^{\frac{\pi}{2}} \cos(2jz_3) z_3^{2H_2-2}\mathrm{d}z_3\\
			\leq&\frac{C}{j}\bigg(  \sin(2jz_3)z_3^{2H_2-2}\Big|_{z_3=0}^{\frac{\pi}{2}} -(2H_2-2) \int_{0}^{\frac{\pi}{2}}  \sin(2jz_3) z_3^{2H_2-3}\mathrm{d}z_3\bigg)\\
			\leq&\frac{C}{j}  \Big(\int_{0}^{\frac{\pi}{2}} \Big|\frac{\sin(2jz_3)}{z_3}\Big|^{\frac{1}{2H_2-1-\epsilon}}  \mathrm{d}z_3\Big)^{2H_2-1-\epsilon}\Big(\int_{0}^{\frac{\pi}{2}} z_3^{-\frac{2-2H_2}{2-2H_2+\epsilon}}  \mathrm{d}z_3\Big)^{2-2H_2+\epsilon} \\
			\leq& Cj^{1-2H_2+\epsilon}\Big(  \int_{0}^{\frac{j\pi}{2}} \Big|\frac{\sin(2z_5)}{z_5}\Big|^{\frac{1}{2H_2-1-\epsilon}}  \mathrm{d}z_5\Big)^{2H_2-1-\epsilon}  \\
			\leq& C\lambda_j^{\frac{1}{2}-H_2+\frac{\epsilon}{2}},
		\end{aligned}
	\end{equation}
	where $\epsilon$ is an arbitrary small positive number.
	Taking $\alpha_1=1-\frac{\epsilon}{8},\  \alpha_2=2H_1-1-\frac{\epsilon}{8},\ \alpha_3=-\frac{4H_1+H_2-1-\beta}{2}+1+\epsilon,$ and $\alpha_4=2H_1-1-\frac{\epsilon}{4}$    in  \eqref{regularity of o(t,x)  1.2} and \eqref{regularity of o(t,x)  1.3},  
	substituting the estimates and \eqref{regularity of o(t,x)  1.5} into  \eqref{regularity of o(t,x)  1}, and applying \eqref{estimate: e^{-x} 1} with $\hat{\alpha}=\frac{4H_1+H_2-1-\beta
		-2\epsilon}{4}$, we can obtain
	\begin{equation}  \label{regularity of o(t,x)  1,6}
		\begin{aligned}
			\|o(t,\cdot)-o(s,\cdot)\|_{L^{2p}(\Omega,\dot{H}^{\beta})} 
			\leq  C(t-s)^{\frac{4H_1+H_2-1-\beta}{4}
				-\frac{\epsilon}{2}},\quad   H_1,\ H_2\in (\frac{1}{2},1).
		\end{aligned}
	\end{equation}
	Furthermore, by \eqref{regularity of o(t,x)  1.1}, \eqref{regularity of o(t,x)  1.3}  and \eqref{regularity of o(t,x)  1.4}, we know that \eqref{regularity of o(t,x)  1,6} also holds for $H_1=\frac{1}{2}$ and $H_2=\frac{1}{2}$. 
	Therefore by  H\"{o}lder's inequality one has 
	\begin{equation} \label{regularity of o(t,x)  5}
		\begin{aligned}
			\mathbb{E}\left[\|o(t, \cdot)-o(s, \cdot)\|^p_{\dot{H}^{\beta} }\right]\leq \|o(t,\cdot)-o(s,\cdot)\|_{L^{2p}(\Omega,\dot{H}^{\beta})}^p\leq C(t-s)^{\frac{(4H_1+H_2-1-\beta
					-2\epsilon)p}{4}},\quad   H_1,\ H_2\in [\frac{1}{2},1).
		\end{aligned}
	\end{equation} 
	Due to the arbitrariness of $p$ in \eqref{regularity of o(t,x)  5},  we apply  \cite[Theorem 2.1]{revuz2013continuous}  to obtain that there exists a constant $C>0$ such that
	$$
	\mathbb{E}\left[\left(\sup _{t\neq s} \frac{ \|o(t, \cdot)-o(s, \cdot)\|_{\dot{H}^{\beta}} }{|t-s|^{\frac{4H_1+H_2-1-\beta}{4}
			-\frac{\epsilon}{2}-\frac{2}{p^*}}}   \right)^{p^*} \right] \leq C,
	$$
	where $p^*\geq \max\{p, \frac{8}{4H_1+H_2-1-\beta
		-4\epsilon}\}.$ 
	This result yields
	\begin{equation*} \label{es: sup_t O^N }
		\begin{aligned}
			\bigg(\mathbb{E}\bigg[\sup_{t\in[0, T]}\|o(t, \cdot)\|_{\dot{H}^{\beta}}^p  \bigg]\bigg)^\frac{1}{p}\leq&
			\bigg(\mathbb{E}\bigg[\Big(\sup_{t\in[0, T]}\|o(t, \cdot)-o(0, \cdot)\|_{\dot{H}^{\beta}}  \Big)^{p^*} \bigg] \bigg)^\frac{1}{p^*}
			\leq	  CT^{\frac{4H_1+H_2-1-\beta}{4}
				-\frac{\epsilon}{2}-\frac{2}{p^*}}
			\leq  C.
		\end{aligned}
	\end{equation*}
	Finally, it follows from \eqref{ineq: Ito Isometry 2},  H{\"o}lder's inequality and \eqref{es:  G_t(x,y)-G_t(z,y)} that 
	\begin{equation} \label{ineq: o(t,x)-o(t,y)}
		\begin{aligned}
			\left\|o(t,x)-o(t,z)\right\|_{L^{2p}(\Omega,\mathbb{R})}  
			\leq&\bigg(\int_{0}^{t}\bigg(\int_{\mathcal{O}}  |G_{t-s}(x,y)-G_{t-s}(z,y)|^\frac{1}{H_2} \mathrm{d} y\bigg)^\frac{H_2}{H_1}\mathrm{d} s\bigg)^{H_1} \\ 
			\leq& 
			\bigg( \int_{0}^{t}\int_{\mathcal{O}} |G_{t-s}(x,y)-G_{t-s}(z,y)|^2 \mathrm{d}y\mathrm{d} s\bigg)^\frac{1}{2}\\ 
			\leq&C|x-z|,
		\end{aligned} 
	\end{equation}
	which yields \eqref{ineq: futher property of O_t}. The proof is completed.
\end{proof}

The regularity of the mild solution \eqref{eq: CH_FractionalNoise_solution_mild} is presented in the following theorem.
\begin{theorem}  \label{th:  u_regularity}
	Suppose that Assumption \ref{assu:2} holds.
	Then for any $p \geq 1$ and $0\leq\beta<4H_1+H_2-1$, 
	the unique mild solution \eqref{eq: CH_FractionalNoise_solution_mild} enjoys
	the following regularity property
	\begin{equation}\label{ineq: u_regularity}
		\sup_{t\in[0,T]}	\mathbb{E}\left[\|u(t, \cdot)\|^p_{\dot{H}^{\beta}}\right] < \infty.
	\end{equation}
\end{theorem}
\begin{proof} 
	Owing to \eqref{eq: CH_FractionalNoise_solution_mild}, we split $\|u(t,x)\|_{L^{p}(\Omega,\dot{H}^{\beta})}$ as 
	\begin{equation} \label{regularity of u(t,x)  1}
		\begin{aligned}
			\|u(t,\cdot)\|_{L^{p}(\Omega,\dot{H}^{\beta})} \leq 
			\|u_1(t,\cdot)\|_{L^{p}(\Omega,\dot{H}^{\beta})} +\|u_f(t,\cdot)\|_{L^{p}(\Omega,\dot{H}^{\beta})} +\|o(t,\cdot)\|_{L^{p}(\Omega,\dot{H}^{\beta})}. 
		\end{aligned}
	\end{equation}
	Now we estimate $\|u_1(t,\cdot)\|_{L^{p}(\Omega,\dot{H}^{\beta})}$. \par   
	Using \eqref{eq:CH_FractionalNoise b} and integration by parts, it is easy to show
	$$  \int_{\mathcal{O}} u_0(x)\phi_j(x) \mathrm{d}x= \frac{1}{\lambda_j^2}\int_{\mathcal{O}} u_0^{(4)}(x)\phi_j(x) \mathrm{d}x, \quad \forall\ j\geq 1.$$
	Then  by the orthonormality of $\left\{\phi_j\right\}_{j\geq 0} $ and Assumption \ref{assu:2},  we can obtain  
	\begin{align}
		\|u_1(t,x)\|_{L^{p}(\Omega,\dot{H}^{\beta})}=&\Big(\sum_{j=0}^{\infty} \lambda_j^\beta e^{-2\lambda_j^2t} \langle u_0,\phi_j   \rangle^2_{L^2 }  \Big)^\frac{1}{2}\nonumber\\
		\leq&\Big(\langle u_0,\phi_0   \rangle^2_{L^2} +\sum_{j=1}^{\infty} \lambda_j^\beta   \langle u_0,\phi_j   \rangle^2_{L^2 }  \Big)^\frac{1}{2}\nonumber\\
		\leq&\Big(\|u_0\|^2_{L^2 } +\sum_{j=1}^{\infty} \lambda_j^{\beta-4}  \langle u_0^{(4)},\phi_j   \rangle^2_{L^2 }  \Big)^\frac{1}{2}\nonumber\\
		\leq&\big(\|u_0\|^2_{L^2 } +\|u_0^{(4)}\|^2_{L^2 }   \big)^\frac{1}{2}\nonumber\\
		\leq& C. \label{regularity of u(t,x) u_1}
	\end{align} 
	
	Since we have already estimated $o(t,x)$ in Lemma \ref{lem: 3.1 regularity of o(t,x)}, the remainder of this proof is to  analyze $\|u_f(t,\cdot)\|_{\dot{H}^{\beta}}$.  It follows from the orthonormality of $\left\{\phi_j\right\}_{j\geq 0} $, H\"{o}lder's inequality and \eqref{estimate: e^{-x} 2} that
	\begin{equation} \label{eq: regularity of u(t,x) u_f 1}
		\begin{aligned}
			\|u_f(t,\cdot)\|_{\dot{H}^{\beta}}^2=& \sum_{j=1}^{\infty} \lambda_j^\beta \bigg|\int_0^t   \lambda_j e^{-\lambda_j^2 (t-s)}  \langle f(u(s,\cdot)),\phi_j(\cdot)  \rangle_{L^2 } \mathrm{d} s \bigg|^2\\
			\leq & \sum_{j=1}^{\infty} \lambda_j^\beta  \int_0^t   \lambda_j^2 e^{-\lambda_j^2 (t-s)}\mathrm{d} s   \int_0^te^{-\lambda_j^2 (t-s)}\langle f(u(s,\cdot)),\phi_j(\cdot)  \rangle_{L^2 }^2 \mathrm{d} s \\
			\leq & C     \int_0^t   (t-s)^{-\alpha}\sum_{j=1}^{\infty} \lambda_j^{\beta-2\alpha} \langle f(u(s,\cdot)),\phi_j(\cdot)  \rangle_{L^2 }^2 \mathrm{d} s,
		\end{aligned}
	\end{equation}
	where $\alpha$ is an arbitrarily nonnegative number.  In the following we estimate $\|u_f(t,\cdot)\|_{\dot{H}^{\beta}}$. \par  
	\noindent \textbf{ Step 1:} $4H_1+H_2-1\leq2$ \par 
	Taking $\alpha=\frac{\beta}{2}$ in \eqref{eq: regularity of u(t,x) u_f 1}, it follows from H\"{o}lder's inequality, \eqref{eq: regularity of u(t,x) u_f 1} and  \eqref{eq: CH_FractionalNoise_solution_mild  boundness} that
	\begin{align*}
		\|u_f(t,\cdot)\|_{L^{p}(\Omega,\dot{H}^{\beta})}\leq& 	\|u_f(t,\cdot)\|_{L^{2p}(\Omega,\dot{H}^{\beta})}\\
		\leq & C	\left\| \int_0^t   (t-s)^{-\frac{\beta}{2}} \| f(u(s,\cdot)) \|_{L^2 }^2 \mathrm{d} s\right\|_{L^p(\Omega,\mathbb{R})}^\frac{1}{2}\\ 
		\leq & C	\left\| \int_0^t   (t-s)^{-\frac{\beta}{2}}\left(1+ \|u(s,\cdot)\|_{L^6 }^6\right)\mathrm{d} s\right\|_{L^p(\Omega,\mathbb{R})}^\frac{1}{2}\\ 
		\leq&  C   \bigg(  \int_0^t   (t-s)^{-\frac{\beta}{2}}   \left(1+\|(u(s,\cdot))\|_{L^{6p}(\Omega,L^6)}^6 \right)\mathrm{d} s \bigg)^\frac{1}{2} \\
		\leq& C.
	\end{align*}
	Combining this with \eqref{ineq: regularity of o(t,x)}, \eqref{regularity of u(t,x)  1} and \eqref{regularity of u(t,x) u_1}  shows \eqref{ineq: u_regularity} holds for $0\leq \beta<4H_1+H_2-1\leq2$. 
	
	\noindent \textbf{ Step 2:}  $2<4H_1+H_2-1\leq3$. \par 
	Define $ L^2_0(\mathcal{O}):=\{v\in L^2(\mathcal{O})| v(0)=v(\pi)=0\}$ and $\partial_x u(t,x):=\partial 	u(t, x)/\partial x.$  Let  $\psi_j(x)= 
	\sqrt{2/ \pi} \sin (j x)$ for $j>0. $ Since $\left\{\psi_j\right\}_{j > 0}$ forms an orthonormal basis of  $L^2_0(\mathcal{O})$, we use integration by parts and  the Neumann boundary condition \eqref{eq:CH_FractionalNoise b}  to get  
	\begin{equation} \label{regularity of u(t,x) u_f 2.1}
		\begin{aligned}
			\|\partial_x u(t,\cdot)\|_{L^2}^2=  \sum_{j=0}^{\infty}  \langle\partial_x u(t,\cdot),\psi_j(\cdot)  \rangle_{L^2 }^2 
			=     \sum_{j=1}^{\infty}\lambda_j  \langle u(t,\cdot),\phi_j(\cdot)  \rangle_{L^2 }^2 
			= \|u (t,\cdot)\|_{\dot{H}^1}^2  .   
		\end{aligned}
	\end{equation}
	
	Note that the result in step 1 implies $ \|u_f(t,\cdot)\|_{L^{p}(\Omega,\dot{H}^1)}\leq C$. It follows from \eqref{regularity of u(t,x) u_f 2.1} that
	\begin{equation}\label{ineq: u_regularity   u_x}
		\sup_{t\in[0,T]}	\mathbb{E}\left[\|\partial_x u(t,\cdot)\|^p_{L^2}\right] < \infty.
	\end{equation}
	Further, we take $\alpha=\frac{\beta-1}{2}$ in \eqref{eq: regularity of u(t,x) u_f 1} and employ H\"{o}lder's inequality, integration by parts,  the Neumann boundary condition\eqref{eq:CH_FractionalNoise b}, \eqref{eq: CH_FractionalNoise_solution_mild  boundness} and \eqref{ineq: u_regularity   u_x} to show
	\begin{equation*} \label{regularity of u(t,x) u_f 3}
		\begin{aligned}
			\|u_f(t,\cdot)\|_{L^{p}(\Omega,\dot{H}^{\beta})} \leq & C	\left\| \int_0^t   (t-s)^{-\frac{\beta-1}{2}}\sum_{j=1}^{\infty} \langle f^\prime(u(s,\cdot))\partial_xu(s,\cdot),\psi_j(\cdot)  \rangle_{L^2 }^2 \mathrm{d} s\right\|_{L^p(\Omega,\mathbb{R})}^\frac{1}{2}\\ 
			\leq & C	\left\| \int_0^t   (t-s)^{-\frac{\beta-1}{2}}\|f^\prime(u(s,\cdot))\partial_xu(s,\cdot)\|_{L^2 }^2 \mathrm{d} s\right\|_{L^p(\Omega,\mathbb{R})}^\frac{1}{2}\\  
			\leq & C	\left\| \int_0^t   (t-s)^{-\frac{\beta-1}{2}}\Big(1+\|u(s,\cdot)\|_{L^8 }^4\Big)\|\partial_xu(s,\cdot)\|_{L^4 }^2 \mathrm{d} s\right\|_{L^p(\Omega,\mathbb{R})}^\frac{1}{2}\\ 
			\leq&  C   \bigg(  \int_0^t   (t-s)^{-\frac{\beta-1}{2}} \Big(1+\|(u(s,\cdot))\|_{L^{8p}(\Omega,L^{8})}^4\Big) \|\partial_xu(s,\cdot)\|_{L^{4p}(\Omega,L^4)}^2 \mathrm{d} s \bigg)^\frac{1}{2} \\
			\leq& C,
		\end{aligned}    
	\end{equation*} 
	combined with  \eqref{ineq: regularity of o(t,x)}, \eqref{regularity of u(t,x)  1},\eqref{regularity of u(t,x) u_1} and the result in step 1  leads to  \eqref{ineq: u_regularity} holds for  $\beta<4H_1+H_2-1\leq 3.$ 
	
	\noindent \textbf{ Step 3:}  $3<4H_1+H_2-1<4$. \par 
	Define  $\partial_{xx} u(t,x):=\partial^2 	u(t, x)/\partial x^2.$ 
	Owing to the Neumann boundary condition\eqref{eq:CH_FractionalNoise b}, we obtain
	\begin{equation} \label{regularity of u(t,x) u_f appendix}
		\begin{aligned}
			\|\partial_{xx} u(t,\cdot)\|_{L^2}^2= \sum_{j=0}^{\infty}  \langle\partial_{xx} u(t,\cdot),\phi_j(\cdot)  \rangle_{L^2 }^2 
			=   \sum_{j=1}^{\infty}\lambda_j^2  \langle  u(t,\cdot),\phi_j(\cdot)  \rangle_{L^2 }^2 
			=   \|u (t,x)\|_{\dot{H}^2}^2.
		\end{aligned}
	\end{equation}
	Letting $\alpha=\frac{\beta-2}{2}$ in \eqref{eq: regularity of u(t,x) u_f 1} and utilizing H\"{o}lder's inequality, integration by parts, \eqref{eq: CH_FractionalNoise_solution_mild  boundness} and \eqref{ineq: u_regularity   u_x} to show 
	\begin{align*}
		\|u_f(t,\cdot)\|_{L^{p}(\Omega,\dot{H}^{\beta})}
		\leq & C	\Big\| \int_0^t  (t-s)^{-\frac{\beta-2}{2}}\sum_{j=1}^{\infty} \langle f^{\prime\prime}(u(s,\cdot))|\partial_xu(s,\cdot)|^2+f^\prime(u(s,\cdot))\partial_{xx}u(s,\cdot),\phi_j(\cdot)  \rangle_{L^2 }^2  \mathrm{d} s	\Big\|_{L^p(\Omega,\mathbb{R})}^\frac{1}{2}\\ 
		\leq & C		\Big\| \int_0^t  (t-s)^{-\frac{\beta-2}{2}}  \|f^{\prime\prime}(u(s,\cdot))|\partial_xu(s,\cdot)|^2+f^\prime(u(s,\cdot))\partial_{xx}u(s,\cdot)\|_{L^2 }^2  \mathrm{d} s	\Big\|_{L^p(\Omega,\mathbb{R})}^\frac{1}{2} \\
		\leq & C		\Big\| \int_0^t (t-s)^{-\frac{\beta-2}{2}}  \big(\big(1+\|u(s,\cdot)\|_{L^4 }^2\big)\|\partial_xu(s,\cdot)\|_{L^8 }^4+
		\big(1+\|u(s,\cdot)\|_{L^8}^4 \big)\|\partial_{xx}u(s,\cdot)\|_{L^4 }^2 \big) \mathrm{d} s\Big\|_{L^p(\Omega,\mathbb{R})}^\frac{1}{2}\\  
		\leq&  C   \Big(  \int_0^t   (t-s)^{-\frac{\beta-2}{2}} \Big(\big(1+\|(u(s,\cdot))\|_{L^{4p}(\Omega,L^{4})}^2\big) \|\partial_xu(s,\cdot)\|_{L^{8p}(\Omega,L^8)}^4\\
		&\qquad +\big(1+\|(u(s,\cdot))\|_{L^{8p}(\Omega,L^{8})}^4\big) \|\partial_{xx}u(s,\cdot)\|_{L^{4p}(\Omega,L^4)}^2\Big) \mathrm{d} s \Big)^\frac{1}{2}  
		\leq C.
	\end{align*}   
	Using the above estimate, \eqref{ineq: regularity of o(t,x)}, \eqref{regularity of u(t,x)  1}, \eqref{regularity of u(t,x) u_1} and  the results in step 1 and step 2, one deduces that   \eqref{ineq: u_regularity} holds for  $\beta<4H_1-H_2-1<4.$  
	The proof is completed.
\end{proof}

\begin{remark}
	By the Sobolev embedding theorem, it follows that $\mathcal{C}(\mathcal{O})$  can be embedded into $W^{1,2}(\mathcal{O})$. Therefore we can utilize \eqref{eq: CH_FractionalNoise_solution_mild  boundness} and \eqref{ineq: u_regularity   u_x} to obtain 
	\begin{equation} \label{ineq: u_sup_boundness}
		\sup_{t\in[0,T]}\mathbb{E}\left[\sup_{x \in\mathcal{O}}|u(t, x)|^p\right] <\infty .
	\end{equation} 
\end{remark}

\section{Spatial finite difference discretization}\label{section4}
In this section, we introduce the finite difference method (FDM) for spatial discretization of \eqref{eq:CH_FractionalNoise} and investigate its strong convergence rate.  Let $N$ be a positive integer, and define $h=\pi/N$, $x_i=(i-\frac{1}{2})h$ for $i=1,2,\ldots,N$.  
Then the spatial grid points are denoted as $\mathcal{O}_h=\{x_i|1\leq i \leq N\}$. Define $\mathcal{W}_h:=\left\{w^h\mid w^h=(w^h_1,\ w^h_1,\cdots,w^h_N) \right\}$
as the space of grid functions defined on $\mathcal{O}_h$. 
For any grid function $w^h \in \mathcal{W}_h$, in accordance with the homogeneous Neumann boundary condition, we define second and fourth order difference operators with reference to  \cite[Chapter 9]{Sun2023}  as follows:
$$
\delta_h^2 w_i^h:=   
\left\{\begin{array}{ll}
	(-w^h_1+w^h_2)/ h^2, & i=1, \\
	(w^h_{i-1}-2 w^h_i+w^h_{i+1})/ h^2, & i=2,3,\cdots,N-1,\\
	(w^h_{N-1}-w^h_N)/h^2, & i=N,
\end{array}\right.  \  \text{and}\quad \delta_h^4 w_i^h:=\delta_h^2\left(\delta_h^2 w_i^h\right),\ i=1,2,\cdots,N.
$$   

For $t \in (0, T]$ and $i = 1, 2, \ldots, N$,  we consider Eq. \eqref{eq:CH_FractionalNoise} at $(t,x_i)$. By replacing
$\Delta$, $\Delta^2$ and $\frac{\partial B^H(t,x_i)}{ \partial x}$ with $\delta_h^2$, $\delta_h^4$ and $\frac{\ N }{\pi}\left(B^H(t,x_{i}+\frac{1}{2}h)-B^H(t, x_{i}-\frac{1}{2}h)\right) $ respectively, we establish the following difference scheme 
\begin{equation}\label{eq:discrete_CH_FractionalNoise_SDE}
	\begin{aligned}
		& \mathrm{d} u^N(t,x_i)+\delta_h^4 u^N(t, x_i) \mathrm{d} t= \delta_h^2 f\big(u^N(t, x_i)\big) \mathrm{d} t+ \frac{\sigma N }{\pi}  \mathrm{d}\big(B^H(t,x_{i}+\frac{1}{2}h)-B^H(t, x_{i}-\frac{1}{2}h)\big).
	\end{aligned}
\end{equation}
Further, we  use the polygonal interpolation
$$ u^N(t,x)=\left\{\begin{array}{lll}
	u^N(t,x_1), & x\in[0,x_1],\\
	u^N(t,x_i)+ \frac{x-x_i }{h}\left(u^N(t,x_{i+1})-u^N(t,x_i)\right),& x \in[x_{i},x_{i+1}],\ i=1,2,\cdots, N-1,\\
	u^N(t,x_N), & x\in[x_N,\pi].
\end{array}  \right. $$
To solve \eqref{eq:discrete_CH_FractionalNoise_SDE}, we introduce 
$$U^N_t=(u^N(t,x_1),\ u^N(t,x_2),\cdots,\ u^N(t,x_N) )^\top\quad \text{and} \quad  \beta^H_t=( \Delta B^H_1(t),\  \Delta B^H_2(t),\cdots,\ \Delta B^H_N(t) )^\top,$$
where $\Delta B^H_i(t)=  \sqrt{N/\pi} \big(B^H(t,,x_{i}+\frac{1}{2}h)-B^H(t,,x_{i}-\frac{1}{2}h)\big)$ for $i=1,2,\cdots,N$. 
Combined with the initial value conditions, we can obtain the following SCHE after spatial discretization
\begin{subequations}\label{eq:eq:discrete_CH_FractionalNoise_SDE_matrix}
	\begin{numcases}{}
		\mathrm{d} U_t^N+A_N^2 U_t^N \mathrm{d} t=A_N F_N(U_t^N) \mathrm{d} t+ \sqrt{\sigma^2N/\pi}     \mathrm{d} \beta^H_t,\ t\in (0,T],\label{eq:eq:discrete_CH_FractionalNoise_SDE_matrix a} \ \\
		U_0^N=U_0, \label{eq:eq:discrete_CH_FractionalNoise_SDE_matrix b}  \
	\end{numcases}
\end{subequations} 
where
$$A_N:={\footnotesize  \frac{N^2}{\pi^2}\left( \begin{array}{ccccc}
		-1 & 1 & 0 & \cdots &  0 \\
		1 & -2 & 1 & \ddots &  \vdots \\ 
		0 & \ddots &  \ddots & \ddots &  0\\ 
		\vdots & \ddots & 1 & -2 & 1 \\
		0 & \cdots  & 0 & 1 & -1
	\end{array}\right) } \ , \ 
U_0=   
\left( \begin{array}{c}
	u_{0,1}\\
	u_{0,2}\\
	\vdots \\ 
	u_{0,N}
\end{array}\right):=
\left( \begin{array}{c}
	u_0(x_1)\\
	u_0(x_2)\\
	\vdots \\ 
	u_0(x_N)
\end{array}\right) ,  
$$
and  $F_N:\ \mathbb{R}^N\to\mathbb{R}^N$ is a deterministic mapping given by 
$$F_N(w):=(f(w_1),f(w_2),\cdots,f(w_N))^\top,\ \ \forall\ w=(w_1,w_2,\cdots,w_N)^\top\in\mathbb{R}^N.$$  
The matrix $-A_N$ is positive semi-definite with the eigensystem $\left\{\lambda_{N,j}, e_j\right\}_{j=0}^{N-1}$, where $\lambda_{N,j}= \frac{4N^2 \sin^2(\frac{j\pi}{2N})}{\pi^2}$, $e_j=(e_j(1),e_j(2),\cdots,e_j(N))^\top$. 
and $e_j(k)=\sqrt{\pi /N} \phi_j(x_k)$ for $k=1,2,\cdots,N$.
Note that  $\left\{e_j\right\}_{j=0}^{N-1}$ constitutes an orthonormal basis in $\mathbb{R}^{N}$. \par 
In the sequel, we shall define the discrete inner product and norms. 
For any vectors  $\mu=\left(\mu_1, \ldots, \mu_N\right)^{\top}$ and $\nu=\left(\nu_1, \ldots, \nu_N\right)^{\top}$,
we can define
\begin{align*}
	&\langle \mu, \nu\rangle_{l_N^2}=h \sum_{i=1}^{N} \mu_i \nu_i\quad\text{and}\quad \| \mu \|_{l_N^2}=\bigg(h \sum_{i=1}^{N} \mu_i^2\bigg)^\frac{1}{2}.
\end{align*}
The discrete $L^2$-norm  has the following relation with the Euclidean norm on $\mathbb{R}^N$ :
\begin{equation} \label{eq: l2 to Euclidean norm}
	\| \mu \|_{l_N^2}^2=h|\mu|^2. 
\end{equation}
Furthermore, for any $p\geq1$ and $k\in \mathbb{N} $ we can define the discrete $L^p$-norm and $W^{k,2}$-norm  as
$$
\|\mu\|_{l_N^p}= \begin{cases}\left(h \sum_{i=1}^{N}\left|\mu_i\right|^p\right)^{\frac{1}{p}}, & 1 \leq p<\infty, \\ \max _{1 \leq i \leq N}\left|\mu_i\right|, & p=\infty,\end{cases} \quad \text{and} \quad  \|\mu\|_{k,N}=
\begin{cases} \| \mu \|_{l_N^2}, & k=0, \\ \left( \| \mu \|_{l_N^2}^2+ \sum_{i=1}^{k}|\mu|_{i,N}^2 \right)^\frac{1}{2}, & k\geq 1,\end{cases} 
$$
where $|\mu|_{i,N}=\left\|\left(-A_N\right)^{\frac{i}{2}} \mu\right\|_{l_N^2}$ and  $\left(-A_N\right)^\gamma:=\sum_{j=0}^{N-1}\lambda_{N,j}^\gamma e_je_j^\top$ for any $i\in\mathbb{N}^+$ and $\gamma>0$.  

It is crucial to emphasize that, due to the definition of $A_N$, we have
\begin{equation} \label{eq: (-A_N)w to w_i}
	|\mu|_{1,N}^2=\left\langle-A_N \mu, \mu\right\rangle_{l_N^2}=\frac{1}{h} \sum_{j=1}^{N-1}\left|\mu_{j+1}-\mu_j\right|^2.
\end{equation} 

Below, we state some key properties of $A_N$.
\begin{lemma}  
	Let $2 \leq p \leq \infty,\ t,\ s>0$ and $N \geq 2$. Then for any $\mu=\left(\mu_1, \mu_2, \ldots, \mu_N\right)^\top,$ 
	$\nu= (\nu_1,$ $ \nu_2, \ldots, \nu_N )^\top$, $w=\left(w_1, w_2, \ldots, w_N\right)^\top:=
	\left(\mu_1\nu_1, \mu_2\nu_2, \ldots, \mu_N\nu_N\right)^\top,\ \gamma,\ \gamma_1,\ \gamma_2\in[0,\infty),\ \gamma_3\in[0,1)$ and $k=1,2,3$ , one can have
	\begin{align}
		&\|\mu\|_{l_N^{\infty}}   \leq C\|\mu\|_{1,N}  , \label{ineq: l_infty to l2} \\
		&| w|_{k,N}   \leq C\|\mu\|_{k,N}\|\nu\|_{k,N}, \label{ineq: uv to u v} \\  
		&\|\mu\|_{l_N^{6}}   \leq
		C\Big(\left\|A_N \mu\right\|_{l_N^2}^{\frac{1}{6}}\|\mu\|_{l_N^2}^{\frac{5}{6}}+\|\mu\|_{l_N^2}\Big) , \label{ineq: l_6 to l2} \\
		&\left\|(-A_N)^\gamma e^{-A_N^2t} \mu\right\|_{l_N^{2}}  \leq Ct^{-\frac{\gamma}{2}}\left\|\mu\right\|_{l_N^2}, \  
		\label{ineq: e^{-A_N^2t}(-A_N)^r to l2} \\ 
		&\left\|(-A_N)^{\gamma_1} e^{-A_N^2t}\left( I- e^{-A_N^2s}\right)\mu\right\|_{l_N^{2}}  \leq C t^{-\frac{\gamma_1-\gamma_2+2\gamma_3}{2}}s^{\gamma_3}\left\|(-A_N)^{\gamma_2}\mu\right\|_{l_N^2},		\label{ineq: (-A_N)^re^{-A_N^2t}(I- e^{-A_N^2t}) to l2}  
	\end{align}
	where $C>0$ is a constant independent of $\mu,\ \nu$ and $N$.
\end{lemma}
\begin{proof}
	See Appendix A.
\end{proof}

\subsection{Key estimates for two auxiliary equations}
To obtain the  moment boundedness and strong convergence rate of the numerical solution,  we need develop delicate estimates for  two auxiliary equations. More precisely, we consider two following auxiliary equations 
\begin{equation}\label{eq: perturbed problem 1}
	\left\{  
	\begin{aligned}
		&\mathrm{d} Y_t+A_N^2 Y_t \mathrm{d} t=A_N F_N(Y_t+Z_t) \mathrm{d} t, \quad t\in (0,T],\\
		&Y_0=y,
	\end{aligned}
	\right.
\end{equation}
and 
\begin{equation}\label{eq: err problem 1}
	\left\{  
	\begin{aligned}
		&\mathrm{d} \hat{Y}_t+A_N^2 \hat{Y}_t \mathrm{d} t= A_N \left[F_N(\hat{X}_t+\hat{Y}_t+\hat{Z}_t)-F_N(\hat{X}_t)\right] \mathrm{d} t, \quad t\in (0,T],\\
		&\hat{Y}_0=0,
	\end{aligned}
	\right.
\end{equation}
where $y,\ Z_t,\ \hat{X}_t,\ \hat{Z}_t \in l^2_N$.
Eq. \eqref{eq: perturbed problem 1} is often called the perturbation equation, which plays a significant role in  proving the  moment boundedness of the numerical solution. The next proposition gives an estimate for 	$\|Y_t\|_{1,N}$.
\begin{proposition}\label{propo: perturbed problem}  
	Let $Y_t$ be the solution of the perturbed problem \eqref{eq: perturbed problem 1}. Then
	for any  $p \geq 1$, there exists a constant $C>0$, independent of $N$, such that 
	\begin{equation}\label{es: solution_perturbed problem}
		\|Y_t\|_{1,N}  \leq C \left(1+\|y\|_{1,N}^{36} + \sup_{s\in[0, t]}\left\|Z_s\right\|_{l^\infty_N}^{27}\right), \quad t \in[0, T].
	\end{equation} 
\end{proposition}
\begin{proof} 
	
	Due to the definition of $\|\cdot\|_{1,N}$, we need estimate $\|Y_t\|_{l^2_N}$ and  $\left\|(-A_N)^{\frac{1}{2}}Y_t\right\|_{l^2_N}$. Hence 	we divide our proof into two steps.\par

	\textbf{Step 1: } We first verify 
	\begin{equation} \label{es: Y }
		\begin{aligned}
			\left\|Y_t\right\|^{2}_{l^2_N} +
			\int_{0}^{t}\left\|A_NY_s\right\|^2_{l^2_N}\mathrm{d} s \leq C \left(1+\left\|  y\right\|^8_{1,N}+ \sup_{s\in[0, t]}\left\|Z_s\right\|_{l^\infty_N}^6\right).  
		\end{aligned}
	\end{equation}
	
	Taking the inner product on \eqref{eq: perturbed problem 1} with $ Y_t $ and integrating with respect to time, we obtain 
	\begin{equation} \label{es: V  1}
		\begin{aligned}
			&\frac{1}{2}\left\| Y_t\right\|^2_{l^2_N}+\int_{0}^{t}\left\|A_N Y_s\right\|^2_{l^2_N}\mathrm{d} s  \\
			= &\frac{1}{2}\left\|y\right\|^2_{l^2_N}
			+\int_{0}^{t}  \big\langle F_N(Y_s+Z_s)-F_N(Y_s), A_NY_s \big\rangle _{l^2_N}\mathrm{d} s +\int_{0}^{t}\big \langle F_N(Y_s), A_N Y_s  \big\rangle _{l^2_N}\mathrm{d} s .
		\end{aligned}
	\end{equation}
	Since $$
	|f(x+y)-f(x)|\leq C(1+|x|^2+|y|^2)|y|,\quad \forall\ x,\ y\in\mathbb{R},
	$$
	it follows from  H\"{o}lder's inequality and Young's inequality that
	\begin{equation} \label{es: < F_N(U_s^N)-F_N(Y_s), A_N Y_s>}
		\begin{aligned}
			&\int_{0}^{t}  \big\langle F_N(Y_s+Z_s)-F_N(Y_s), A_N Y_s \big\rangle _{l^2_N}\mathrm{d} s\\
			\leq&C \int_{0}^{t}  \left\|Z_s \right\|_{l^\infty_N}(1+ \left\| Y_s \right\|^2_{l^4_N}+ \left\| Z_s \right\|^2_{l^4_N})  \left\|A_N Y_s \right\|_{l^2_N} \mathrm{d} s\\ 
			\leq& \frac{1}{4}\int_{0}^{t} \left\|A_N Y_s \right\|^2_{l^2_N} \mathrm{d} s
			+C\sup_{s\in[0, t]}\left\|  Z_s \right\|^2_{l^\infty_N}\left(1+ \int_{0}^{t}\left\|  Y_s+Z_s \right\|^4_{l^4_N} \mathrm{d} s+\int_{0}^{t} \left\|  Z_s \right\|^4_{l^\infty_N} \mathrm{d} s   \right).
		\end{aligned}
	\end{equation}
	Due to the fact $ \big\langle Y_t^3, A_N Y_t \big\rangle _{l^2_N}\leq 0 $, one can have 
	\begin{align}
		\int_{0}^{t} \big\langle F_N(Y_s), A_N Y_s \big\rangle _{l^2_N}\mathrm{d} s 
		\leq &\frac{1}{4}\int_{0}^{t} \left\|A_N Y_s \right\|^2_{l^2_N} \mathrm{d} s+
		C\left(1+\int_{0}^{t} \left\| Y_s\right\|^4_{l^4_N} \mathrm{d} s\right)\nonumber\\
		\leq &\frac{1}{4}\int_{0}^{t} \left\|A_N Y_s \right\|^2_{l^2_N} \mathrm{d} s+
		C\left(1+ \int_{0}^{t} \left\| Y_s+Z_s \right\|^4_{l^4_N} \mathrm{d} s +\int_{0}^{t}\left\|  Z_s \right\|^4_{l^\infty_N} \mathrm{d} s\right).
		\label{es: <  F_N(V^N(s)), A_N V^N(s)>}
	\end{align} 
	Substituting \eqref{es: < F_N(U_s^N)-F_N(Y_s), A_N Y_s>} and \eqref{es: <  F_N(V^N(s)), A_N V^N(s)>}   into \eqref{es: V  1} yields
	\begin{equation} \label{es: Y   1.1}
		\begin{aligned}
			\left\|Y_t\right\|^{2}_{l^2_N} +
			\int_{0}^{t}\left\|A_NY_s\right\|^2_{l^2_N}\mathrm{d} s \leq \left\|y\right\|^2_{l^2_N}+C\Big(1+ \sup_{s\in[0, t]}\left\|  Z_s \right\|^2_{l^\infty_N}\Big)\left(1+ \int_{0}^{t}\left\|  Y_s+Z_s \right\|^4_{l^4_N} \mathrm{d} s+\int_{0}^{t} \left\|  Z_s \right\|^4_{l^\infty_N} \mathrm{d} s   \right).
		\end{aligned}
	\end{equation}
	In what follows we estimate the term $\int_{0}^{t}\left\|  Y_s+Z_s \right\|^4_{l^4_N} \mathrm{d} s$.
	According to \eqref{eq: perturbed problem 1}, for $t\in (0,T]$ one can have $\left\langle Y_t,e_0\right\rangle =\left\langle y,e_0\right\rangle $ and
	
	\begin{equation*} 
		\mathrm{d} \left\langle Y_t-y,e_j\right\rangle +\lambda_{N,j}^2 \left\langle Y_t,e_j\right\rangle \mathrm{d} t=-\lambda_{N,j} \left\langle F_N(Y_t+Z_t),e_j\right\rangle \mathrm{d} t, \ j=1,2,\cdots,N-1.
	\end{equation*}
	Multiplying the above equation by $\lambda_{N,j}^{-1}h\left\langle Y_t-y,e_j\right\rangle$ and integrating with respect to time, for $j=1,2,\cdots,N-1$, we obtain 
	\begin{equation*} 
		\begin{aligned}
			&\frac{h \left\langle Y_t-y,e_j \right\rangle^2 }{2\lambda_{N,j}}+ \int_{0}^{t}h \big\langle(-A_N)^{\frac{1}{2}}Y_s,e_j \big\rangle^2 \mathrm{d} s \\
			= &\int_{0}^{t}h \big\langle(-A_N)^{\frac{1}{2}}Y_s,e_j \big\rangle\big\langle(-A_N)^{\frac{1}{2}}y,e_j \big\rangle \mathrm{d} s -\int_{0}^{t} h\left\langle F_N(Y_s+Z_s),e_j\right\rangle \left\langle Y_s-y,e_j\right\rangle  \mathrm{d} s.
		\end{aligned}
	\end{equation*} 
	Notice $\lambda_{N,0}=0$ and $\left\langle Y_t-y,e_0\right\rangle =0$. By summing up for $j$ from $1$ to $N-1$, we obtain
	\begin{equation}\label{eq: (-A_N)^{1/2}Y(t)  1.1} 
		\begin{aligned}
			&\sum_{j=1}^{N-1}\frac{\left\langle Y_t-y,e_j\right\rangle^2_{l^2_N}}{2\lambda_{N,j}}+ \int_{0}^{t}\left\|(-A_N)^{\frac{1}{2}}Y_s\right\|^2_{l^2_N}\mathrm{d} s \\
			= & \int_{0}^{t}\big\langle(-A_N)^{\frac{1}{2}}Y_s,(-A_N)^{\frac{1}{2}}y \big\rangle_{l^2_N}\mathrm{d} s -\int_{0}^{t} \langle F_N(Y_s+Z_s), Y_s-y \rangle_{l^2_N}\mathrm{d} s,
		\end{aligned}
	\end{equation}
	which leads to
	\begin{equation}\label{eq: (-A_N)^{1/2}Y(t)  1.2} 
		\begin{aligned}
			&\frac{1}{2}	\int_{0}^{t}\left\|(-A_N)^{\frac{1}{2}}Y_s\right\|^2_{l^2_N}\mathrm{d} s +\int_{0}^{t}\langle F_N(Y_s+Z_s),  Y_s+Z_s \rangle _{l^2_N}\mathrm{d} s  \\
			\leq  &\frac{t}{2}\left\|(-A_N)^{\frac{1}{2}}y\right\|^2_{l^2_N}+\int_{0}^{t}\langle F_N(Y_s+Z_s), Z_s+y \rangle _{l^2_N}\mathrm{d} s.
		\end{aligned}
	\end{equation}
	We need to estimate the second term on the left-hand side and  the second term on the right-hand side of Eq.  \eqref{eq: (-A_N)^{1/2}Y(t)  1.2} . Due to the definition of $F_N$, H\"{o}lder's inequality and Young's inequality, one can have 
	\begin{equation}\label{eq: < F_N(Y+Z), Y+Z >} 
		\begin{aligned}
			& \int_{0}^{t} \big\langle F_N(Y_s+Z_s),Y_s+Z_s \big\rangle_{l^2_N}\mathrm{d} s  \geq \frac{7}{8}a_0\int_{0}^{t}  \left\|Y_s+Z_s\right\|_{l^4_N}^4\mathrm{d}s-C,
		\end{aligned}
	\end{equation}
	and
	\begin{equation}\label{eq: < F_N(Y+Z), Z >} 
		\begin{aligned}
			\int_{0}^{t} \big\langle F_N(Y_s+Z_s), Z_s+y \big\rangle _{l^2_N}\mathrm{d} s   
			\leq& \frac{5}{4}a_0 \int_{0}^{t}   \left\|Y_s+Z_s\right\|_{l^4_N}^3 \left\|Z_s+y\right\|_{l^4_N} \mathrm{d} s 
			+C \int_{0}^{t}  \left\|Z_s+y\right\|_{l^1_N} \mathrm{d} s  \\
			\leq& \frac{5}{8}a_0 \int_{0}^{t}   \left\|Y_s+Z_s\right\|_{l^4_N}^4  \mathrm{d} s 
			+C \int_{0}^{t}  \left\|Z_s\right\|_{l^4_N}^4 \mathrm{d} s+C\left(\|y\|_{l^\infty_N}^4+1\right) . \\
		\end{aligned}
	\end{equation}
	By substituting  \eqref{eq: < F_N(Y+Z), Y+Z >}  and \eqref{eq: < F_N(Y+Z), Z >}    into \eqref{eq: (-A_N)^{1/2}Y(t)  1.2}   and applying \eqref{ineq: l_infty to l2}, we derive\begin{equation} \label{es: (A)^{-1/2}Y }
		\begin{aligned}
			\int_{0}^{t}\left\|(-A_N)^{\frac{1}{2}}Y_s\right\|^2_{l^2_N}\mathrm{d} s   +
			\int_{0}^{t}\left\|Y_s+Z_s\right\|^4_{l^4_N}\mathrm{d} s 
			\leq  C \left(1+\left\|y\right\|^4_{1,N}+\int_{0}^{t}  \left\|Z_s\right\|_{l^4_N}^4  \mathrm{d} s\right).
		\end{aligned}
	\end{equation} 
	This combined with \eqref{es: Y   1.1}  leads to  \eqref{es: Y }.
	
	\textbf{Step 2: }		Next we show
	\begin{equation}  \label{es: (-A_n)^1/2 Y}
		\begin{aligned}
			\left\|(-A_N)^\frac{1}{2}Y(t)\right\|_{l^2_N} 
			\leq&   	C \left(1+\|y\|_{1,N}^{36} + \sup_{s\in[0, t]}\left\|Z_s\right\|_{l^\infty_N}^{27}\right).  
		\end{aligned}
	\end{equation}  
	Note that the mild solution to \eqref{eq: perturbed problem 1} can be written as $ Y_t=e^{-A_N^2t}y+\int_0^t A_Ne^{-A_N^2(t-s)}F_N(Y_s+Z_s)\ \mathrm{d} s$. 
	It follows from   \eqref{ineq: e^{-A_N^2t}(-A_N)^r to l2} that
	\begin{equation} \label{es: (-A_n)^1/2 Y 1.1}
		\begin{aligned}
			\left\|(-A_N)^\frac{1}{2}Y_t\right\|_{l^2_N}\leq&   \left\|(-A_N)^\frac{1}{2}e^{-A_N^2t}y\right\|_{l^2_N}+
			\int_{0}^{t}  \left\|(-A_N)^\frac{3}{2}e^{-A_N^2(t-s)}  F_N(Y_s+Z_s)\right\|_{l^2_N}\mathrm{d} s\\
			\leq&  C\left(\left\|(-A_N)^\frac{1}{2}y\right\|_{l^2_N}+ \int_{0}^{t}(t-s)^{-\frac{3}{4}} \left\| F_N(Y_s+Z_s )\right\|_{l^2_N}\mathrm{d} s\right)\\
			\leq&  C\left(1+\left\|(-A_N)^\frac{1}{2}y\right\|_{l^2_N}+\sup_{s\in[0, t]}\left\|Z_s\right\|_{l^\infty_N}^{3}+ \int_{0}^{t}(t-s)^{-\frac{3}{4}}\left\|Y_s\right\|_{l^6_N}^3\mathrm{d} s  \right) 
		\end{aligned}
	\end{equation} 
	Due to \eqref{ineq: l_infty to l2}, \eqref{ineq: e^{-A_N^2t}(-A_N)^r to l2}, \eqref{estimate: e^{-x} 2} with $\alpha=\frac{13}{16}$ and H\"{o}lder's inequality, one can have 
	\begin{align}
		\left\|Y_t\right\|_{l^6_N}\leq&   \left\|e^{-A_N^2t}y\right\|_{l^6_N}+
		\int_{0}^{t}  \left\| A_Ne^{- A_N^2(t-s)}F_N(Y_s+Z_s)\right\|_{l^6_N}\mathrm{d} s\nonumber\\
		\leq&  \left\|e^{-A_N^2t}y\right\|_{l^\infty_N}+
		\int_0^t\Big\|\sum_{j=1}^{N-1} - \lambda_{N,j}e^{-\lambda_{N,j}^2(t-s)}\left\langle F_N(Y_s+Z_s), e_j\right\rangle  e_j\Big\|_{l^6_N} \mathrm{d} s\nonumber \\
		\leq& 
		C  \left\|e^{-A_N^2t}y\right\|_{1,N}
		+  \sum_{j=1}^{N-1} \int_0^t\lambda_{N,j}e^{-\lambda_{N,j}^2(t-s)}\left|\left\langle F_N(Y_s+Z_s), e_j\right\rangle\right|\left\|e_j\right\|_{l^6_N}\mathrm{d} s	\nonumber		
		\\
		\leq&  C\bigg(\|y\|_{1,N}+   \sum_{j=1}^{N-1} \int_0^t \lambda_{N,j}\sqrt{h}e^{-\lambda_{N,j}^2(t-s)} \left|\left\langle F_N(Y_s+Z_s), e_j\right\rangle\right|  \mathrm{d} s	 \bigg)\nonumber\\ 
		\leq&  C\bigg(\|y\|_{1,N} + \int_0^t   \Big(\sum_{j=1}^{N-1}\lambda_{N,j}e^{-\lambda_{N,j}^2(t-s)}\Big)^\frac{1}{2}
		\Big(h\sum_{j=1}^{N-1}\lambda_{N,j}e^{-\lambda_{N,j}^2(t-s)} \left|\left\langle F_N(Y_s+Z_s), e_j\right\rangle\right|^2 \Big)^\frac{1}{2}  \mathrm{d} s\bigg)   \nonumber	 \\ 
		\leq&  C\bigg(\|y\|_{1,N}+\int_0^t (t-s)^{-\frac{13}{32}} \left\| (-A_N)^\frac{1}{2}e^{- \frac{1}{2}A_N^2(t-s)}F_N(Y_s+Z_s)\right\|_{l^2_N}  \mathrm{d} s	 \bigg)  \qquad\qquad\qquad \nonumber\\
		\leq&  C\bigg(\|y\|_{1,N}
		+\int_0^t (t-s)^{-\frac{21}{32}} \left\|  F_N(Y_s+Z_s)\right\|_{l^2_N}  \mathrm{d} s	  \bigg).\nonumber
	\end{align}
	Concerning the above estimate, we utilize  the definition of $f$, \eqref{ineq: l_6 to l2} and H\"{o}lder's inequality  to get
	\begin{align}
		\left\|Y_t\right\|_{l^6_N}	\leq&  C\bigg(1+\|y\|_{1,N}
		+\int_0^t (t-s)^{-\frac{21}{32}} \left(\left\| Y_s\right\|_{l^6_N}^3+\left\| Z_s\right\|_{l^6_N}^3         \right) \mathrm{d} s	  \bigg)\nonumber
		\\
		\leq&  C\bigg(1+\|y\|_{1,N}
		+\int_0^t (t-s)^{-\frac{21}{32}} \left\|A_N Y_s\right\|_{l_N^2}^{\frac{1}{2}}\|Y_s\|_{l_N^2}^{\frac{5}{2}}   \mathrm{d} s \bigg)\nonumber\\
		&+  C\bigg(\int_0^t (t-s)^{-\frac{21}{32}} \left\| Y_s\right\|_{l^2_N}^3    \mathrm{d} s+\int_0^t (t-s)^{-\frac{21}{32}} \left\| Z_s\right\|_{l^6_N}^3    \mathrm{d} s \bigg)\nonumber\\   
		\leq&  C\bigg(1+\|y\|_{1,N}
		+\int_0^t (t-s)^{-\frac{21}{32}} \left\| Y_s\right\|_{l^2_N}^3    \mathrm{d} s  +\int_0^t (t-s)^{-\frac{21}{32}} \left\| Z_s\right\|_{l^6_N}^3    \mathrm{d} s \bigg)\nonumber\\
		&+C \left( \int_0^t   \left\|A_N Y_s\right\|_{l_N^2}^2   \mathrm{d} s\right)^\frac{1}{4}
		\left( \int_0^t  (t-s)^{-\frac{7}{8}} \left\| Y_s\right\|_{l_N^2}^\frac{10}{3}   \mathrm{d} s\right)^\frac{3}{4}.
		\label{es: Y_l6 1}
	\end{align}
	Combining \eqref{es: Y_l6 1} with \eqref{es: Y } and \eqref{es: (-A_n)^1/2 Y 1.1}, we can obtain \eqref{es: (-A_n)^1/2 Y}. \par 
	Finally, gathering \eqref{es: Y } and \eqref{es: (-A_n)^1/2 Y} together  enables us to obtain \eqref{es: solution_perturbed problem},  which completes the proof. 
\end{proof}

As to \eqref{eq: err problem 1}, we can call it the error equation, which is necessary to  establish the strong convergence rate of the numerical solution.  Now we give an estimate for $\hat{Y}_t$.
\begin{proposition}\label{propo: err problem 1}
	Let $\hat{Y}_t$ be the solution of the error equation \eqref{eq: err problem 1}. Then
	for any $p \geq 1$, there exists a constant $C>0$ , independent of $N$, such that  
	\begin{equation}\label{es: err problem 1}
		\begin{aligned} 
			\mathbb{E} \left[ \|\hat{Y}_t \|_{l_N^2}^p\right]   \leq&  CK_2 \int_0^t(t-s)^{-\frac{1}{2}} 
			\left(\mathbb{E}\big[\big\|\hat{Z}_s\big\|_{l_N^2}^{2p}\big]\right)^\frac{1}{2} \mathrm{d} s +	C \sqrt{K_1K_2}  \bigg(  \int_{0}^{T} \left(\mathbb{E}\left[\|\hat{Z}_t\|_{l_N^{4}}^{4p}\right]\right)^\frac{1}{2}\mathrm{d}s\bigg)^\frac{1}{2},\ t\in [0,T],
		\end{aligned}  
	\end{equation} 
	where $$K_1=\max\bigg\{ 1, \sup_{s\in[0, T]}\left( \mathbb{E}\left[\|\hat{X}_s+\hat{Y}_s\|_{l_N^{ 8}}^{8p}\right]
	\right)^\frac{1}{2},\sup_{s\in[0, T]}\left( \mathbb{E}\left[\| \hat{Z}_s\|_{l_N^{ 8}}^{8p}\right]
	\right)^\frac{1}{2} \bigg\}$$ and  $$K_2= \max\bigg\{ 1, \sup_{s\in[0, T]}\left( \mathbb{E}\left[\|\hat{X}_s\|_{1,N}^{4p}\right]
	\right)^\frac{1}{2},\sup_{s\in[0, T]}\left( \mathbb{E}\left[\|\hat{X}_s+\hat{Y}_s\|_{1,N}^{4p}\right]
	\right)^\frac{1}{2},\sup_{s\in[0, T]}\left( \mathbb{E}\left[\|\hat{X}_s+\hat{Y}_s+\hat{Z}_s\|_{1,N}^{4p}\right]
	\right)^\frac{1}{2} \bigg\}. $$ 
\end{proposition}

In order to prove this proposition, we introduce the following lemma, which presents the local Lipschitz continuity and polynomial growth of $F_N$ under the semi-norm $|\cdot |_{k,N}$ for $k=1,2,3$. 
\begin{lemma}\label{lem: F local Lipschitz}  
	For any $\mu,\ \nu \in \mathbb{R}^{N}$, there exists a constant $C>0$, independent of $\mu,\ \nu$ and $N$, such that 
	\begin{align} & | \left(F_N(\mu)-F_N(\nu)\right) |_{k,N} 
		\leq C\Big(1+	\|\mu\|_{k,N}^2
		+ \|\nu \|_{k,N}^2 \Big)\|\mu-\nu\|_{k,N}, \qquad k=1,2,3, \label{ineq:  F local Lipschitz}\\
		&|F_N(\mu)\ |_{k,N}  \leq C\left(1+	\|\mu\|_{k,N}^3\right) , \qquad\qquad\qquad\qquad \qquad\qquad\qquad\ \  k=1,2,3.  \label{ineq:  F local Lipschitz   2}
	\end{align} 
\end{lemma}
\begin{proof} 
	See Appendix B. 
\end{proof} 

Now we are ready to prove Proposition \ref{propo: err problem 1}

\begin{proof}[Proof of  Proposition \ref{propo: err problem 1}]
	Similar to \eqref{eq: (-A_N)^{1/2}Y(t)  1.1}, one can have
	$$		\begin{aligned}
		& \sum_{j=1}^{N-1}\frac{\big\langle \hat{Y}_t,e_j\big\rangle^2_{l^2_N}}{2\lambda_{N,j}}+ \int_{0}^{t}\left\|(-A_N)^{\frac{1}{2}}\hat{Y}_s\right\|^2_{l^2_N}\mathrm{d} s
		=-\int_{0}^{t} \left\langle F_N(\hat{X}_s+\hat{Y}_s+\hat{Z}_s)-F_N(\hat{X}_s),\hat{Y}_s\right\rangle_{l^2_N}  \mathrm{d} s, 
	\end{aligned}$$ 
	which yields   
	\begin{align}
		& \sum_{j=1}^{N-1}\frac{\big\langle \hat{Y}_t,e_j\big\rangle^2_{l^2_N}}{2\lambda_{N,j}}+ \int_{0}^{t}\left\|(-A_N)^{\frac{1}{2}}\hat{Y}_s\right\|^2_{l^2_N}\mathrm{d} s \nonumber \\ 
		\leq&  \int_{0}^{t} 
		\Big\langle F_N(\hat{X}_s )-F_N(\hat{X}_s+\hat{Y}_s),  \hat{Y}_s \Big\rangle_{l^2_N}\mathrm{d}s
		+
		\int_{0}^{t} \Big\langle F_N(\hat{X}_s+\hat{Y}_s)-F_N(\hat{X}_s+\hat{Y}_s+\hat{Z}_s),  \hat{Y}_s\Big\rangle_{l^2_N} 
		\mathrm{d}s\nonumber \\
		\leq&C\left(\int_{0}^{t} \left\|\hat{Y}_s\right\|_{l^2_N}	^2		\mathrm{d}s
		+ \int_{0}^{t} \left\| F_N(\hat{X}_s+\hat{Y}_s)-F_N(\hat{X}_s+\hat{Y}_s+\hat{Z}_s)\right\|_{l^2_N}^2			\mathrm{d}s\right).\label{es: (-A_N)^{-1/2}(U_t-tilde{U}(t))    1}
	\end{align}
	From \eqref{eq: err problem 1}, it is evident that $\big\langle \hat{Y}_t, e_0\big\rangle_{l^2_N}^2 = 0$ for any $t \in [0, T]$. 
	Hence, one can deduce 
	\begin{equation} \label{es: (-A_N)^{-1/2}(U_t-tilde{U}(t))    2.1}
		\begin{aligned}
			\int_{0}^{t} \|\hat{Y}_s \|_{l^2_N}	^2		\mathrm{d}s= \int_{0}^{t} \sum_{j=1}^{N-1}\big\langle \hat{Y}_s,e_j\big\rangle^2_{l^2_N}	\mathrm{d}s 
			\leq  \int_{0}^{t} \bigg( \sum_{j=1}^{N-1}\frac{\big\langle \hat{Y}_s,e_j\big\rangle^2_{l^2_N}}{\lambda_{N,j}}\bigg)^\frac{1}{2}
			\left\|(-A_N)^{\frac{1}{2}}\hat{Y}_s\right\|_{l^2_N} \mathrm{d}s.
		\end{aligned}
	\end{equation}
	It follows from \eqref{es: f 2} that
	\begin{equation}\label{es: (-A_N)^{-1/2}(U_t-tilde{U}(t))    2.2}
		\begin{aligned}
			\left\|F_N(\hat{X}_s+\hat{Y}_s)-F_N(\hat{X}_s+\hat{Y}_s+\hat{Z}_s)\right\|^2_{l^2_N}
			\leq C \left(1+\|\hat{X}_s+\hat{Y}_s\|_{l_N^{ 8}}^4+\|\hat{Z}_s\|_{l_N^8}^4\right)\|\hat{Z}_s\|_{l_N^{4}}^2.
		\end{aligned}
	\end{equation} 
	Upon substituting \eqref{es: (-A_N)^{-1/2}(U_t-tilde{U}(t))    2.1} and \eqref{es: (-A_N)^{-1/2}(U_t-tilde{U}(t))    2.2} into \eqref{es: (-A_N)^{-1/2}(U_t-tilde{U}(t))    1} and employing Young's inequality, we can deduce that 
	\begin{equation}  \label{es: int ||(-A)^{-1/2} E^N||    1 modify}
		\begin{aligned}
			& \sum_{j=1}^{N-1}\frac{\big\langle \hat{Y}_t,e_j\big\rangle^2_{l^2_N}}{\lambda_{N,j}}+
			\int_{0}^{t}\left\|(-A_N)^{\frac{1}{2}}\hat{Y}_s\right\|_{l^2_N}^2\mathrm{d}s \\
			\leq  & C \bigg( \int_{0}^{t}   \sum_{j=1}^{N-1}\frac{\big\langle \hat{Y}_s,e_j\big\rangle^2_{l^2_N}}{\lambda_{N,j}} \mathrm{d}s+\int_{0}^{t}   \left(1+\|\hat{X}_s+\hat{Y}_s\|_{l_N^{ 8}}^4+\|\hat{Z}_s\|_{l_N^8}^4\right)\|\hat{Z}_s\|_{l_N^{4}}^2\mathrm{d}s\bigg).
		\end{aligned}
	\end{equation}
	Furthermore, it follows from the Gronwall lemma that 
	$$\sum_{j=1}^{N-1}\frac{\big\langle \hat{Y}_t,e_j\big\rangle^2_{l^2_N}}{\lambda_{N,j}} 
	\leq  C\int_{0}^{t}   \left(1+\|\hat{X}_s+\hat{Y}_s\|_{l_N^{ 8}}^4+\|\hat{Z}_s\|_{l_N^8}^4\right)\|\hat{Z}_s\|_{l_N^{4}}^2\mathrm{d}s,  $$
	which combined with \eqref{es: int ||(-A)^{-1/2} E^N||    1 modify} yields
	\begin{equation*} \label{es: int ||(-A)^{-1/2} E^N||    1}
		\begin{aligned}
			\int_{0}^{t}\bigg(\sum_{j=1}^{N-1}\frac{\big\langle \hat{Y}_s,e_j\big\rangle^2_{l^2_N}}{\lambda_{N,j}}+ \left\|(-A_N)^{\frac{1}{2}}\hat{Y}_s\right\|_{l^2_N}^2\bigg)\mathrm{d}s 
			\leq  C \int_{0}^{T}   \left(1+\|\hat{X}_s+\hat{Y}_s\|_{l_N^{ 8}}^4+\|\hat{Z}_s\|_{l_N^8}^4\right)\|\hat{Z}_s\|_{l_N^{4}}^2\mathrm{d}s .
		\end{aligned}
	\end{equation*}
	By the above estimate and  H{\"o}lder's inequality, one deduces 
	
	\begin{equation*}
		\begin{aligned}
			\mathbb{E}\bigg[ \bigg|\int_{0}^{t}\bigg(\sum_{j=1}^{N-1}\frac{\big\langle \hat{Y}_s,e_j\big\rangle^2_{l^2_N}}{\lambda_{N,j}} +
			\left\|(-A_N)^{\frac{1}{2}}\hat{Y}_s\right\|_{l^2_N}^2\bigg)\mathrm{d}s\bigg|^p  \bigg]  
			\leq& C \mathbb{E}\left[ \left| \int_{0}^{T}\left(1+\|\hat{X}_s+\hat{Y}_s\|_{l_N^{ 8}}^4+\|\hat{Z}_s\|_{l_N^8}^4\right)\|\hat{Z}_s\|_{l_N^{4}}^2  \mathrm{d}s\right|^p  \right] \\
			\leq& C K_1   \int_{0}^{T}  \left(\mathbb{E}\left[\|\hat{Z}_s\|_{l_N^{4}}^{4p}\right]\right)^\frac{1}{2}\mathrm{d}s ,
		\end{aligned}
	\end{equation*}
	where $$K_1=\max\bigg\{ 1, \sup_{s\in[0, T]}\left( \mathbb{E}\left[\|\hat{X}_s+\hat{Y}_s\|_{l_N^{ 8}}^{8p}\right]
	\right)^\frac{1}{2},\sup_{s\in[0, T]}\left( \mathbb{E}\left[\| \hat{Z}_s\|_{l_N^{ 8}}^{8p}\right]
	\right)^\frac{1}{2} \bigg\}.$$
	Combining this with \eqref{es: (-A_N)^{-1/2}(U_t-tilde{U}(t))    2.1} shows
	
	\begin{equation}  \label{ineq: int  Y 1,N}
		\begin{aligned}
			\mathbb{E}\bigg[
			\bigg|\int_0^t  \| \hat{Y}_s \|_{1,N}^2\mathrm{d} s\bigg|^p   \bigg]&\leq 
			C\mathbb{E}\bigg[ \bigg|\int_{0}^{t}\sum_{j=1}^{N}\frac{\big\langle \hat{Y}_s,e_j\big\rangle^2_{l^2_N}}{\lambda_{N,j}} +
			\left\|(-A_N)^{\frac{1}{2}}\hat{Y}_s\right\|_{l^2_N}^2\mathrm{d}s\bigg|^p  \bigg]  
			\leq CK_1   \int_{0}^{T} \left(\mathbb{E}\left[\|\hat{Z}_s\|_{l_N^{4}}^{4p}\right]\right)^\frac{1}{2}\mathrm{d}s .
		\end{aligned}
	\end{equation}   
	Since the mild solution to \eqref{eq: err problem 1} can be written as $$ \hat{Y}_t=\int_0^t A_Ne^{-A_N^2(t-s)}\left(F_N(\hat{X}_t+\hat{Y}_t+\hat{Z}_t)-F_N(\hat{X}_t)\right) \mathrm{d} s,$$  it follows from \eqref{ineq: e^{-A_N^2t}(-A_N)^r to l2} and  Lemma \ref{lem: F local Lipschitz} that 
	\begin{equation} \label{ineq: ||Y_t||_l^2}
		\begin{aligned}
			\|\hat{Y}_t \|_{l_N^2} \leq& \int_0^t\left\|A_N  e^{-A_N^2(t-s)}\left(F_N(\hat{X}_t+\hat{Y}_t+\hat{Z}_t)-F_N(\hat{X}_t)\right)\right\|_{l_N^2} \mathrm{d} s \\
			\leq& C \int_0^t(t-s)^{-\frac{1}{2}}  \left\|F_N(\hat{X}_s+\hat{Y}_s+\hat{Z}_s)-F_N(\hat{X}_s+\hat{Y}_s) \right\|_{l_N^2}\mathrm{d} s\\
			&+C \int_0^t(t-s)^{-\frac{1}{4}}  \left\|\left(-A_N\right)^{\frac{1}{2}}\left(F_N(\hat{X}_s+\hat{Y}_s)-F_N(\hat{X}_s)\right) \right\|_{l_N^2}\mathrm{d} s  \\
			\leq& C \int_0^t(t-s)^{-\frac{1}{2}}\left(1+\big\|\hat{X}_s+\hat{Y}_s+\hat{Z}_s\big\|_{l_N^{\infty}}^2+\big\|\hat{X}_s+\hat{Y}_s \big\|_{l_N^{\infty}}^2\right)\big\|\hat{Z}_s\big\|_{l_N^2}\mathrm{d} s \\
			&+ C \int_0^t(t-s)^{-\frac{1}{4}}\left(1 + 	\|\hat{X}_s+\hat{Y}_s\|_{1,N}^2+	\|\hat{X}_s\|_{1,N}^2\right)\|\hat{Y}_s\|_{1,N}\mathrm{d} s.
		\end{aligned}
	\end{equation} 
	Using  H{\"o}lder's inequality, \eqref{ineq: ||Y_t||_l^2}, \eqref{ineq: int  Y 1,N} and \eqref{ineq: l_infty to l2}, we can obtain 
	\begin{align*}
		\mathbb{E} \left[ \|\hat{Y}_t \|_{l_N^2}^p\right] 
		\leq 	&C \int_0^t(t-s)^{-\frac{1}{2}}\mathbb{E}\bigg[\left(1+	\big\|\hat{X}_s+\hat{Y}_s+\hat{Z}_s\big\|_{l_N^{\infty}}^2+\big\|\hat{X}_s+\hat{Y}_s\big\|_{l_N^{\infty}}^2\right)^p\big\|\hat{Z}_s\big\|_{l_N^2}^p\bigg]\mathrm{d} s \\
		&+C	\mathbb{E}\bigg[ \bigg| \int_0^t(t-s)^{-\frac{1}{2}}\left(1+\|\hat{X}_s+\hat{Y}_s \|_{1,N}^4+  \|  \hat{X}_s  \|_{1,N}^4\right)\mathrm{d} s    
		\int_0^t  \|\hat{Y}_s \|_{1,N}^2\mathrm{d} s\bigg|^\frac{p}{2} \bigg]\\
		\leq&  C \int_0^t(t-s)^{-\frac{1}{2}}\left(1+	\left(\mathbb{E}\Big[\big\|\hat{X}_s+\hat{Y}_s+\hat{Z}_s\big\|_{l_N^{\infty}}^{4p}\Big]\right)^\frac{1}{2}+\left(\mathbb{E}\Big[\big\|\hat{X}_s+\hat{Y}_s \big\|_{l_N^{\infty}}^{4p}\Big]\right)^\frac{1}{2}\right)
		\left(\mathbb{E}\Big[\big\|\hat{Z}_s\big\|_{l_N^2}^{2p}\Big]\right)^\frac{1}{2} \mathrm{d} s \\
		&+	C \bigg(\mathbb{E}\bigg[
		\bigg|\int_0^t(t-s)^{-\frac{1}{2}}\left(1+\|\hat{X}_s+\hat{Y}_s \|_{1,N}^4+\|\hat{X}_s  \|_{1,N}^4\right)\mathrm{d} s\bigg|^p   \bigg]\bigg)^\frac{1}{2} \bigg(\mathbb{E}\bigg[
		\bigg|\int_0^t \| \hat{Y}_s \|_{1,N}^2\mathrm{d} s\bigg|^p   \bigg]\bigg)^\frac{1}{2} \\ 
		\leq&  CK_2 \int_0^t(t-s)^{-\frac{1}{2}} 
		\left(\mathbb{E}\Big[\big\|\hat{Z}_s\big\|_{l_N^2}^{2p}\Big]\right)^\frac{1}{2} \mathrm{d} s +	C \sqrt{K_1K_2}  \bigg(  \int_{0}^{T}\Big(\mathbb{E}\Big[\|\hat{Z}_s\|_{l_N^{4}}^{4p}\Big]\Big)^\frac{1}{2}\mathrm{d}s\bigg)^\frac{1}{2},
	\end{align*} 
	where $$K_2= \max\bigg\{ 1, \sup_{d\in[0, T]}\left( \mathbb{E}\left[\|\hat{X}_s\|_{1,N}^{4p}\right]
	\right)^\frac{1}{2},\sup_{s\in[0, T]}\left( \mathbb{E}\left[\|\hat{X}_s+\hat{Y}_s\|_{1,N}^{4p}\right]
	\right)^\frac{1}{2},\sup_{s\in[0, T]}\left( \mathbb{E}\left[\|\hat{X}_s+\hat{Y}_s+\hat{Z}_s\|_{1,N}^{4p}\right]
	\right)^\frac{1}{2} \bigg\}. $$
	The proof is completed.
\end{proof}   
\subsection{Well-posedness of the semi-discrete numerical solution}
Our first target is to show the well-posedness of Eq. \eqref{eq:eq:discrete_CH_FractionalNoise_SDE_matrix}.
\begin{theorem} 
	Suppose that $u_0 \in \mathcal{C}(\mathcal{O})$, then the problem \eqref{eq:eq:discrete_CH_FractionalNoise_SDE_matrix} has a
	unique mild solution $U^N_t\in \mathcal{C}\left([0, T] , \mathbb{R}^{N}\right) $, which satisfies
	\begin{equation} \label{eq: mild_solution_space_discrete}
		\begin{aligned}
			U_t^N= & e^{-A_N^2 t}  U_0+\int_0^t A_N e^{-A_N^2(t-s)}  F_N(U_s^N) \mathrm{d} s  +\sqrt{\sigma^2N/\pi}  \int_0^t e^{-A_N^2(t-s)} \mathrm{d} \beta^H_s, \quad t \in[0, T].
		\end{aligned}
	\end{equation}
\end{theorem}

\begin{proof}
	Utilizing  \eqref{es: f 2}, we can obtain that for $\mu, \nu\in \mathbb{R}^{N}$, $|\mu|, |\nu| \leq R$ and $R \in (0, \infty)$, there exists a  positive  constant $L_N^R$ such that 
	$$
	\begin{aligned}
		\left|\left(A_N^2\mu+ A_NF_N(\mu)\right)-\left(A_N^2\nu+ A_NF_N(\nu)\right)\right| \leq L_N^R |\mu-\nu|.
	\end{aligned}
	$$
	By applying \cite[Remark 3.1.5]{Mishura2008}, 
	we deduce that  there is a unique strong solution $\{U^N_t\}_{t \in[0, T]}$ to the problem \eqref{eq:eq:discrete_CH_FractionalNoise_SDE_matrix} , and it is continuous and $\left\{\mathscr{F}_t\right\}$-adapted. Furthermore, it follows from the variations of constants formula that $U^N_t$ satisfies \eqref{eq: mild_solution_space_discrete}. The proof is completed. 
\end{proof}

Next, For any $w\in \mathcal{C}(\mathcal{O})$ we define the operators 
$$\mathrm{P}_N w(x):=\left\{\begin{array}{lll}
	w(x_1), & x\in[0,x_1],\\
	w(x_i)+ \frac{x-x_i }{h}\left(w(x_{i+1})-w(x_i)\right),& x \in[x_{i},x_{i+1}],\ i=1,2,\cdots, N-1,\\
	w(x_N), & x\in[x_N,\pi].
\end{array}  \right.$$
and
$$\Delta_N w(x):=\left\{\begin{array}{lll}
	\left(-w(x_1)+w(x_2)\right) / h^2, & x \in [0,x_1+\frac{h}{2}), \\ 
	\left(w(x_{i-1})-2w(x_i)+w(x_{i+1})\right) / h^2, & x \in [x_i-\frac{h}{2},x_i+\frac{h}{2}),\ i=2,3,\cdots, N-1,\\
	\left(w(x_{N-1})-w(x_N)\right) / h^2, &x \in [x_N-\frac{h}{2},\pi] .
\end{array}  \right.$$ 
Further, we introduce a discrete Green function $$
G_t^N(x, y)=\sum_{j=0}^{N-1} e^{-\lambda_{N,j}^2 t} \phi_{N,j}(x) \phi_j\left(\kappa_N(y)\right),
$$
where $
\kappa_N(x):=\left\{\begin{array}{ll}
	x_i, & x\in[x_i-\frac{h}{2},x_i+\frac{h}{2}),\ i=1,2,\cdots, N-1,\\
	x_N,& x\in[x_N-\frac{h}{2},x_N+\frac{h}{2}],
\end{array}  \right. 
$
and 
$
\phi_{N,j}(x):= \mathrm{P}_N  \phi_j(x).
$
Similar to   \cite[section 2]{Gyongy1998}, by \eqref{eq: mild_solution_space_discrete} and $u^N(t, x_i)=\sum_{j=1}^{N-1}\left\langle U^N_t, e_j\right\rangle e_j(k)$ it is easy to check 
\begin{equation} \label{eq:eq:discrete_CH_FractionalNoise_SDE_spatial} 
	\begin{aligned}
		u^N(t, x):= & \int_{\mathcal{O}} G_t^N(x, y) u_0\left(\kappa_N(y)\right) \mathrm{d} y+\int_0^t \int_{\mathcal{O}} \Delta_N G_{t-s}^N(x, y) f(u^N\left(s, \kappa_N(y)\right)) \mathrm{d} y \mathrm{d} s \\
		& +\sigma\int_0^t \int_{\mathcal{O}} G_{t-s}^N(x, y)  B^H(\mathrm{d} s, \mathrm{d} y), \qquad\qquad\qquad\quad \forall \ (t, x) \in[0, T] \times \mathcal{O} .
	\end{aligned}
\end{equation} 
Since $\Delta_N \phi_j\left(\kappa_N(y)\right)= -\lambda_{N,j} \phi_j\left(\kappa_N(y)\right)$, then $ \Delta_N G_{t-s}^N(x, y) =-\sum_{j=0}^{N-1}\lambda_{N,j} e^{-\lambda_{N,j}^2 t} \phi_{j, N}(x) \phi_j\left(\kappa_N(y)\right).$  \par 
The relationship between $G$ and $G^N$ are revealed in the following lemma, which plays an essential role in the error analysis. 
\begin{lemma}   \cite[Lemma 3.2]{Cardon2000}
	Given $\alpha_1 \in (0,1)$ and $\alpha_2 \in (0,2)$, there exists a constant $C>0$ such that for every $t \in(0, T]$
	\begin{align}
		& \sup_{x \in\mathcal{O}}\int_{\mathcal{O}}\left|\Delta_N G_t^N(x, y)-\Delta G_t(x, y)\right| \mathrm{d} y   \leq C  t^{-\frac{2\alpha_1+5}{8}} N^{-\alpha_1} ,\label{es: Delta_N G_s^N-Delta G_s(x, y) } \\
		& \sup_{x \in\mathcal{O}} \int_{\mathcal{O}}\left|G_t^N(x, y)-G_t(x, y)\right|^2 \mathrm{d} y  \leq C t^{-\frac{\alpha_2+1}{4}} N^{-\alpha_2} . \label{es: G_s^N-G_s(x, y) }
	\end{align}
	
\end{lemma}

\subsection{Moment boundedness  of  the semi-discrete numerical solution} 
In order to establish the moment boundedness of $U^N_t$ in \eqref{eq: mild_solution_space_discrete}, we introduce the intermediate solution
\begin{equation*}  
	\begin{aligned}
		V^N_t:=U^N_t-\mathscr{O}^N_t= & e^{-A_N^2 t}  U_0+\int_0^t A_N e^{-A_N^2(t-s)}  F_N(U_s^N) \mathrm{d} s, \quad t \in[0, T].
	\end{aligned}
\end{equation*}
where  $$\mathscr{O}^N_t:=(o^N(t,x_1),\cdots,o^N(t,x_{N}))^\top\quad \text{and} \quad o^N(t,x)=\sigma\int_{0}^{t}\int_{\mathcal{O}} G^N_{t-s}(x,y)   \  B^H(\mathrm{d} s, \mathrm{d} y).$$
Then we can obtain the corresponding perturbed differential equation 
\begin{equation} \label{eq: perturbed differential equation V^N }
	\left\{  
	\begin{aligned}
		&\mathrm{d} V^N_t+A_N^2 V^N_t \mathrm{d} t=A_N F_N(V^N_t+\mathscr{O}^N_t) \mathrm{d} t, \quad t\in (0,T],\\
		&V^N_0=U_0.
	\end{aligned}
	\right.
\end{equation}

Drawing upon Proposition \ref{propo: perturbed problem}, the regularity of $\mathscr{O}^N_t$ becomes essential to establish the moment boundedness of $U^N_t$. The following lemma provides the regularity of $\mathscr{O}^N_t$, which is a discrete version of Lemma \ref{lem: 3.1 regularity of o(t,x)}. The proof is similar and omitted.
\begin{lemma}  \label{lem: 4.4 regularity of O^N_t} 
	For any $p \geq 1$ and $0\leq\beta<4H_1+H_2-1$, the  semi-discrete stochastic convolution $\mathscr{O}^N_t$   enjoys 
	the following regularity property
	\begin{equation}\label{ineq: regularity of o^N(t,x)}
		\mathbb{E}\left[\sup_{t\in[0,T]} \left\|(-A_N)^{\frac{\beta}{2}}\mathscr{O}^N_t \right\|^p_{l^2_N}\right] < \infty.
	\end{equation}
	Moreover,  there exists a constant $C>0$, independent of $N$, such that for any $0\leq s<t\leq T$,
	\begin{equation}\label{ineq: continuous of o^N(t,x)}
		\mathbb{E}\left[\left\|(-A_N)^{\frac{\beta}{2}}\left(\mathscr{O}^N_t-\mathscr{O}^N_s\right) \right\|^p_{l^2_N}\right]\leq C(t-s)^{\left(\frac{4H_1+H_2-1-\beta}{4}
			-\frac{\epsilon}{2}\right)p},
	\end{equation}
	where $\epsilon$ is an arbitrary small positive number, and for any $x,\ z \in \mathcal{O}$,  
	\begin{equation}\label{ineq: futher property of O^N_t}
		\sup_{t\in[0,T]}\mathbb{E}\left[|o^N(t, x)-o^N(t, z)|^p\right]\leq C|x-z|^p. 
	\end{equation} 
\end{lemma}

\begin{remark}
	Notice that $4H_1+H_2-1\geq\frac{3}{2}$. It follows from Lemma \ref{lem: 4.4 regularity of O^N_t} and  \eqref{ineq: l_infty to l2}  that 
	\begin{equation} \label{ineq: O^N_sup_boundness}
		\mathbb{E}\left[\sup_{t\in[0,T]} \|\mathscr{O}^N_t \|^p_{l^\infty_N}\right]\leq C\mathbb{E}\left[\sup_{t\in[0,T]} \|\mathscr{O}^N_t \|^p_{1,N}\right] < \infty.
	\end{equation} 
\end{remark}
Now, we are ready to derive the moment boundedness of $U^N_t$.
\begin{theorem}\label{th: U^N regularity}  
	Suppose that Assumption \ref{assu:2} holds, and let $k$ be the largest integer smaller than $4H_1+H_2-1$.
	Then for $p \geq 1$ and $1\leq i\leq k $, there exists a constant $C>0$, independent of $N$,  such that
	\begin{equation}\label{ineq: U_t_l^infty}
		\mathbb{E}\left[	\left\|  U^N_t \right\|_{i,N}^p\right]\leq C, \quad \forall\ t \in[0, T].
	\end{equation} 
\end{theorem}
\begin{proof}	
	According to \eqref{eq: perturbed differential equation V^N }, it follows from Proposition \ref{propo: perturbed problem}   that 
	\begin{equation}\label{es:  V to O} 
		\|V^N_t\|_{1,N} \leq C\left(1+\|U_0\|_{1,N}^{36}+ \sup_{s\in[0, t]}\left\|\mathscr{O}^N_s\right\|_{l^\infty_N}^{27}\right), \quad t \in[0, T].
	\end{equation}  
	By Assumption \ref{assu:2}, the Neumann boundary condition \eqref{eq:CH_FractionalNoise b} and the Taylor formular, it is straightforward to show that 
	$$\left|\delta_h^4 u_{0,i}-u_{0}^{(4)}(x_i)\right|\leq Ch,\quad i=1,2,\cdots,N,   $$
	which yields
	\begin{equation} \label{es:   U_0 4,N}
		\left\| (-A_N)^2U_0\right\|_{l^2_N}^2=h\sum_{i=1}^{N}\left|\delta_h^4 u_{0,i}\right|^2\leq C h\sum_{i=1}^{N}\left( \left|u_{0}^{(4)}(x_i)\right|^2+h^2\right)\leq C.
	\end{equation}
	Then
	\begin{equation} \label{es:   U_0 k,N}
		\begin{aligned}
			\left\|U_0\right\|_{1,N}^2\leq \left\|U_0\right\|_{l_N^2}^2+\left\|(-A_N)^\frac{1}{2}U_0\right\|_{l^2_N}^2 
			\leq   h \sum_{j=1}^{N}u_0(x_{j})^2+ \left\|(-A_N)^2U_0\right\|_{l^2_N}^2
			\leq  C.
		\end{aligned}
	\end{equation} 
	It follows from \eqref{es:  V to O},  \eqref{ineq: O^N_sup_boundness} and \eqref{es:   U_0 k,N} that
	\begin{equation}    \label{es:   U_0 1,N}
		\begin{aligned}
			\mathbb{E}\left[\left\|  U^N_t \right\|^p_{1,N}\right]\leq&
			C \left(\mathbb{E}\left[\left\|  V^N_t \right\|^p_{1,N}\right]+\mathbb{E}\left[\left\|  \mathscr{O}^N_t \right\|^p_{1,N}\right]    \right)  
			\leq  C \bigg(1+\|U_0\|_{1,N}^{36p}+ \mathbb{E} \Big[\sup_{s\in[0, t]}\left\|\mathscr{O}^N_s\right\|_{l^\infty_N}^{27p} \Big] \bigg)     
			\leq C.
		\end{aligned}
	\end{equation} 
	When $ 4H_1+H_2-1>2$,  by  \eqref{ineq: e^{-A_N^2t}(-A_N)^r to l2}, Lemma \ref{lem: 4.4 regularity of O^N_t}, \eqref{es:   U_0 4,N}, Lemma \ref{lem: F local Lipschitz} and \eqref{es:   U_0 1,N} we have
	\begin{equation}    \label{es:   U_0 2,N  1}
		\begin{aligned}
			&\left\| (-A_N)  U^N_t \right\|_{L^p(\Omega,{l^2_N})}\\
			\leq&   
			\left\|	(-A_N) e^{-A_N^2 t}  U_0\right\|_{l^2_N}+\int_0^t  \left\|(-A_N)^2 e^{-A_N^2(t-s)}  F_N(U_s^N) \right\|_{L^p(\Omega,{l^2_N})}   \mathrm{d} s + \left\|(-A_N)  \mathscr{O}^N_t \right\|_{L^p(\Omega,{l^2_N})}   \\
			\leq&C\left(   
			\left\|	(-A_N) U_0\right\|_{l^2_N}+ \int_0^t  (t-s)^{-\frac{3}{4}}\left\|(-A_N)^\frac{1}{2}    F_N(U_s^N) \right\|_{L^p(\Omega,{l^2_N})}   \mathrm{d} s +1  \right)   \\ 
			\leq&C\left(  1+ \int_0^t  (t-s)^{-\frac{3}{4}}\left(1+ 	\left(\mathbb{E}\left[\left\|  U^N_s \right\|^{3p}_{1,N}\right]\right)^\frac{1}{p} \right)  \mathrm{d} s   \right)   \\
			\leq& C.
		\end{aligned}
	\end{equation}  
	Using above estimate and \eqref{es:   U_0 1,N}, one deduces
	\begin{equation}    \label{es:   U_0 2,N}
		\begin{aligned}
			\mathbb{E}\left[\left\|  U^N_t \right\|^p_{2,N}\right]\leq  C.
		\end{aligned}
	\end{equation} 
	When  $ 4H_1+H_2-1>3$,  by the similar derivation of \eqref{es:   U_0 2,N  1}  we have    
	\begin{align*}
		&\left\| (-A_N)^\frac{3}{2}  U^N_t \right\|_{L^p(\Omega,{l^2_N})}\\
		\leq&   
		\left\|	(-A_N)^\frac{3}{2} e^{-A_N^2 t}  U_0\right\|_{l^2_N}+\int_0^t  \left\|(-A_N)^\frac{5}{2} e^{-A_N^2(t-s)}  F_N(U_s^N) \right\|_{L^p(\Omega,{l^2_N})}   \mathrm{d} s + \left\|(-A_N)^\frac{3}{2}  \mathscr{O}^N_t \right\|_{L^p(\Omega,{l^2_N})}   \\
		\leq&C\left(   
		\left\|	(-A_N)^\frac{3}{2} U_0\right\|_{l^2_N}+ \int_0^t  (t-s)^{-\frac{3}{4}}\left\|(-A_N)     F_N(U_s^N) \right\|_{L^p(\Omega,{l^2_N})}   \mathrm{d} s +1  \right)   \\
		\leq&C\left(   
		\left\|	(-A_N)^2 U_0\right\|_{l^2_N}+ \int_0^t  (t-s)^{-\frac{3}{4}}\left(1+ 	\left(\mathbb{E}\left[\left\|  U^N_s \right\|^{3p}_{2,N}\right]\right)^\frac{1}{p} \right)  \mathrm{d} s +1  \right)   
		\leq C.
	\end{align*} 
	Combining this with \eqref{es:   U_0 2,N} shows
	\begin{equation*}    
		\begin{aligned}
			\mathbb{E}\left[\left\|  U^N_t \right\|^p_{3,N}\right]\leq  C.
		\end{aligned}
	\end{equation*}
	This completes the proof.
\end{proof}

\subsection{Convergence of the semi-discrete numerical solution} 
To establish the strong convergence rate of $U^N_t$ in \eqref{eq: mild_solution_space_discrete}, we introduce the auxiliary process $  \{\tilde{u}^N(t, x) \}_{(t, x) \in [0, T] \times \mathcal{O}}$ by
\begin{equation}\label{eq:auxiliary_CH_FractionalNoise_solution_mild}
	\begin{aligned}
		\tilde{u}^N(t, x)= & \int_{\mathcal{O}} G_t^N(x, y) u_0\left(\kappa_N(y)\right) \mathrm{~d} y+\int_0^t \int_{\mathcal{O}} \Delta_N G_{t-s}^N(x, y) f\left(u\left(s, \kappa_N(y)\right)\right) \mathrm{~d} y \mathrm{d} s \\
		& +\sigma\int_0^t \int_{\mathcal{O}} G_{t-s}^N(x, y) \  B^H(\mathrm{d} s, \mathrm{d} y), \qquad\qquad\qquad\quad\forall\ (t, x) \in[0, T] \times \mathcal{O}.
	\end{aligned}
\end{equation}
Let $\tilde{U}^N_t=\big(\tilde{U}^N_{1,t} ,\cdots,\tilde{U}^N_{N,t}\big)^\top:=\big(\tilde{u}^N(t, x_1) ,\cdots,\tilde{u}^N(t, x_{N})\big)^\top $, which satisfies
\begin{equation}\label{eq:eq:discrete_CH_FractionalNoise_SDE_matrix auxiliary process}
	\left\{  
	\begin{aligned}
		&\mathrm{d}\tilde{U}^N_t+A_N^2 \tilde{U}^N_t \mathrm{d} t=A_N F_N(U_t) \mathrm{d} t+\sqrt{\sigma^2N/\pi}    \mathrm{d} \beta^H_t, \quad t\in (0,T],\\
		&\tilde{U}_0=U_0.
	\end{aligned}
	\right.
\end{equation}
Here $U_t:=\big(u(t, x_1),u(t, x_2),\cdots,u(t, x_{N})\big)^\top .$  \par  
The main result of this subsection is presented in the following theorem.
\begin{theorem}\label{thm: ||u^n-u||}  
	Suppose that Assumption \ref{assu:2} holds.
	Then for $p \geq 1$, there exists a constant $C>0$, independent of $N$,  such that  
	\begin{equation}\label{ineq: ||u^n-u||}
		\sup_{t \in[0, T]}\left\|u^N(t, \cdot)-u(t,\cdot)\right\|_{L^p(\Omega,L^\infty)} \leq Ch^{1-\epsilon},
	\end{equation}  
	where $\epsilon$ is an arbitrary small positive number.
\end{theorem}
In order to estimate the error of  the spatial FDM for Eq. \eqref{eq:CH_FractionalNoise}, we separate 
$\left\|u(t, \cdot)-u^N(t, \cdot)\right\|_{L^p(\Omega,L^\infty)}$
into two parts
$$\left\|u^N(t, \cdot)-u(t, \cdot)\right\|_{L^p(\Omega,L^\infty)}\leq\left\|u^N(t,\cdot)-\tilde{u}^N(t, \cdot)\right\|_{L^p(\Omega,L^\infty)}+\left\|\tilde{u}^N(t,\cdot)-u(t, \cdot)\right\|_{L^p(\Omega,L^\infty)}.$$
Then Theorem \ref{thm: ||u^n-u||} is a straightforward consequence of two auxiliary results stated below.

\begin{proposition}\label{lem: ||tilde{u^n}-u ||}  
	Suppose that Assumption \ref{assu:2} holds.
	Then for $p \geq 1$, there exists a constant $C>0$, independent of $N$,  such that  
	\begin{equation}\label{ineq: ||tilde{u^n}-u ||}
		\sup_{t \in[0, T]}\left\|\tilde{u}^N(t,\cdot)-u(t, \cdot)\right\|_{L^p(\Omega,L^\infty)} \leq Ch^{1-\epsilon},
	\end{equation}  
	where $\epsilon$ is an arbitrary small positive number.
\end{proposition}
\begin{proposition}\label{lem: ||u^n-tilde{u^n} ||}  
	Suppose that Assumption \ref{assu:2} holds.
	Then for $p \geq 1$, there exists a constant $C>0$, independent of $N$,  such that  
	\begin{equation}\label{ineq: ||u^n-tilde{u^n} ||}
		\sup_{t \in[0, T]}\left\|u^N(t, \cdot)-\tilde{u}^N(t,\cdot)\right\|_{L^p(\Omega,L^\infty)} \leq Ch^{1-\epsilon},
	\end{equation}  
	where $\epsilon$ is an arbitrary small positive number.
\end{proposition}

It remains to prove Propositions \ref{lem: ||tilde{u^n}-u ||}  and \ref{lem: ||u^n-tilde{u^n} ||}   

\begin{proof}[Proof of Proposition \ref{lem: ||tilde{u^n}-u ||}]
	By using \eqref{eq:CH_FractionalNoise b}, the orthonormality of $\left\{\phi_j\right\}_{j\geq 0} $ and Assumption \ref{assu:2}, we have 
	\begin{equation*} 
		\begin{aligned}
			u_0(x)-\int_0^t \int_{\mathcal{O}} \Delta G_{t-s}(x, y) u_0^{\prime\prime}(y) \mathrm{d} y\mathrm{d}s  
			=& \sum_{j=0}^{\infty}\bigg(   \langle  u_0,\phi_j\rangle_{L^2}\phi_j(x)-\ \lambda_j^2\int_0^t  e^{-\lambda_j^2(t-s)}\mathrm{d} s \langle  u_0,\phi_j\rangle_{L^2}\phi_j(x)\bigg)		\\
			=&\int_{\mathcal{O}} G_t(x, y) u_0\left(y\right) \mathrm{d}y.
		\end{aligned}
	\end{equation*} 
	Likewise, we can obtain 
	\begin{equation*} 
		\begin{aligned}
			\tilde{u}^N_0(x)-\int_{0}^{t}\int_{\mathcal{O}}\Delta_NG_{t-s}^N(x, y)\Delta_Nu_0((\kappa_N(y))\mathrm{d}y\mathrm{d}s=	\int_{\mathcal{O}} G_t^N(x, y) u_0\left(\kappa_N(y)\right) \mathrm{d}y,
		\end{aligned}
	\end{equation*} 
	where
	$  \tilde{u}^N_0(x):= \mathrm{P}_Nu_0(x).$
	
	Using the above equations	and subtracting \eqref{eq: CH_FractionalNoise_solution_mild} from \eqref{eq:auxiliary_CH_FractionalNoise_solution_mild}, one has  
	\begin{equation*} 
		\begin{aligned}
			&\tilde{u}^N(t, x)-u(t, x)\\
			= &\big(	\tilde{u}^N_0(x)-u_0(x)\big)+\int_{0}^{t}\int_{\mathcal{O}}
			\left(\Delta_NG_{t-s}^N(x, y)-\Delta G_{t-s}(x, y)\right)\left(f\left(u\left(s, \kappa_N(y)\right)\right)-\Delta_N u_0((\kappa_N(y)) \right)\mathrm{d}y\mathrm{d}s \\ 
			&+\int_0^t \int_{\mathcal{O}}
			\Delta G_{t-s}(x, y)\Big(\left( f\left(u\left(s, \kappa_N(y)\right)\right)-f(u(s,y)) \right) -\left( \Delta_Nu_0(\kappa_N(y))-u^{\prime\prime}_0(y)\right)\Big) \mathrm{d} y \mathrm{d} s \\
			& +\sigma\int_0^t \int_{\mathcal{O}} G_{t-s}^N(x, y)-G_{t-s}(x, y) \  B^H(\mathrm{d} s, \mathrm{d} y), \\
			=&: \mathcal{I}_1(x)+\mathcal{I}_2(t,x)+\mathcal{I}_3(t,x)+\mathcal{I}_4(t,x).  
		\end{aligned}
	\end{equation*}
	For any $0\leq t\leq T$,  it follows from Assumption \ref{assu:2},  \eqref{ineq: u_sup_boundness} and  \eqref{es: Delta_N G_s^N-Delta G_s(x, y) } with $\alpha_1=1-\epsilon$ that
	\begin{equation} \label{ineq: ||tilde{u^n}-u || 1}
		\begin{aligned}
			&\left\|\mathcal{I}_1(t,\cdot)\right\|_{L^p(\Omega,L^\infty)}+\left\|\mathcal{I}_2(t,\cdot)\right\|_{L^p(\Omega,L^\infty)}\\
			\leq&	\sup_{x\in\mathcal{O}}\big|\tilde{u}^N_0(x)-u_0(x)\big|+\int_{0}^{t}\sup_{x \in\mathcal{O}}\int_{\mathcal{O}}
			\left|\Delta_NG_{t-s}^N(x, y)-\Delta G_{t-s}(x, y)\right|\left(\|f\left(u\left(s, \kappa_N(y)\right)\right)\|_{L^p(\Omega,\mathbb{R})}+|\Delta_Nu_0(y) | \right)\mathrm{d}y\mathrm{d}s  \\  
			\leq&Ch^2 +C\int_{0}^{t}\sup_{x \in\mathcal{O}}\int_{\mathcal{O}}
			\left|\Delta_NG_{t-s}^N(x, y)-\Delta G_{t-s}(x, y)\right| \mathrm{d}y\mathrm{d}s \\
			\leq& Ch^{1-\epsilon}.
		\end{aligned}
	\end{equation}
	
	To estimate of  $\mathcal{I}_3(t,x)$, we need to investigate $u\left(t, \kappa_N(x)\right)-u(t,x)$. Taking an arbitrary small positive number $\epsilon$, we can choose a sufficiently large number $q$ such that $q>\max\left\{ p+2,3/\epsilon \right\}$. For any $0\leq t\leq T$ and $x\in \mathcal{O}$, utilizing Assumption \ref{assu:2}, \eqref{ineq: futher property of O_t}, H{\"o}lder's inequality,  \eqref{ineq: u_sup_boundness} and \eqref{es:  DeltaG_t(x,y)-DeltaG_t(z,y)} with $\alpha=1-\epsilon$, we obtain 
	\begin{align*}
		& \left\|u\left(t, \kappa_N(x)\right)-u(t,x)\right\|_{L^{2q}(\Omega,\mathbb{R})}\\
		\leq&|u_0(\kappa_N(x))-u_0(x)|+ \left\|o(t,\kappa_N(x))-o(t,x)\right\|_{L^{2q}(\Omega,\mathbb{R})}\\
		&+\int_0^t \int_{\mathcal{O}}   \left|\Delta G_{t-s}(\kappa_N(x), y)- \Delta G_{t-s}(x, y)\right| \left(|u^{\prime\prime}_0\left(y\right)|+\left\|f\left(u\left(s, y\right)\right)\right\|_{L^{2q}(\Omega,\mathbb{R})}\right)  \mathrm{d} y \mathrm{d} s  \\
		\leq&Ch+   C\int_0^t \int_{\mathcal{O}}   \left|\Delta G_{t-s}(\kappa_N(x), y)- \Delta G_{t-s}(x, y)\right| \left(1+\left\|u\left(s, y\right)\right\|_{L^{6q}(\Omega,\mathbb{R})}^3\right)  \mathrm{d} y \mathrm{d} s\\
		\leq&Ch+ C\Big(1+\sup_{(t,x) \in[0, T] \times \mathcal{O}}\left\|u\left(s, x\right)\right\|_{L^{6q}(\Omega,\mathbb{R})}^3\Big)   \int_0^t  \int_{\mathcal{O}}   \left|\Delta G_{t-s}(\kappa_N(x), y)- \Delta G_{t-s}(x, y)\right|   \mathrm{d} y \mathrm{d} s \\ 
		\leq&Ch^{1-\epsilon}.
	\end{align*} 
	
	It follows from the aforementioned estimate, \eqref{es: Delta G_t(x,y)} and \eqref{ineq: u_sup_boundness} that
	\begin{equation} \label{ineq: ||tilde{u^n}-u || 1.1}
		\begin{aligned}
			&\left\|\mathcal{I}_3(t,x)\right\|_{L^{q}(\Omega,\mathbb{R})}\\
			\leq& \int_0^t \int_{\mathcal{O}} \left|\Delta G_s(x, y)\right| 
			\left( 	|\Delta_Nu_0(y)-u^{\prime\prime}_0(y)|+	\|f\left(u\left(s, \kappa_N(y)\right)\right)-f(u(s,y)) \|_{L^{q}(\Omega,\mathbb{R})} \right)   \mathrm{d}y\mathrm{d}s   \\
			\leq&  C\int_0^t \int_{\mathcal{O}} 
			s^{-\frac{3}{4}}e^{-c\frac{|x-y|^{4/3}}{s^{1/3}}}
			\Big(h^2+	\Big(1+\sup_{x\in\mathcal{O}}\|u(s,x) \|_{L^{4q}(\Omega,\mathbb{R})}^2\Big)\|u\left(s, \kappa_N(y)\right)-u(s,y) \|_{L^{2q}(\Omega,\mathbb{R})}\Big)\mathrm{d}y\mathrm{d}s\\
			\leq& Ch^{1-\epsilon} .
		\end{aligned}
	\end{equation} 
	Moreover, for any $x,\ z\in \mathcal{O}$, by \eqref{es:  DeltaG_t(x,y)-DeltaG_t(z,y)} with $\alpha=\frac{3}{4}$ we obtain
	\begin{equation*} 
		\begin{aligned}
			\left\|\mathcal{I}_3(t,x)-\mathcal{I}_3(t,z)\right\|_{L^{q}(\Omega,\mathbb{R})}
			\leq  Ch^{1-\epsilon}\int_0^t \int_{\mathcal{O}} \left|\Delta G_{t-s}(x, y)-\Delta G_{t-s}(z, y)\right|   \mathrm{d}y\mathrm{d}s   
			\leq Ch^{1-\epsilon}|x-z|^\frac{3}{4}.  
		\end{aligned}
	\end{equation*} 
	Thus we employ \cite[Theorem 2.1]{revuz2013continuous}  to obtain 
	\begin{equation*}   
		\begin{aligned}
			\mathbb{E}\left[\sup_{x\neq z} \left( 
			\frac{\frac{1}{h^{1-\epsilon}}|\mathcal{I}_3(t,x)-\mathcal{I}_3(t,z)|}{|x-z|^{\frac{3}{4}-\frac{2}{q}}} \right)^q\right] \leq C,
		\end{aligned}
	\end{equation*}  
	which combined with \eqref{ineq: ||tilde{u^n}-u || 1.1} yields
	\begin{equation}  \label{ineq: ||tilde{u^n}-u || 2}
		\begin{aligned}
			\left\|\mathcal{I}_3(t,\cdot)\right\|_{L^p(\Omega,L^\infty)}\leq	\left\|\mathcal{I}_3(t,\cdot)\right\|_{L^q(\Omega,L^\infty)}
			\leq	\left(	\mathbb{E}\left[\sup_{x\neq 0}
			|\mathcal{I}_3(t,x)-\mathcal{I}_3(t,0)|^q\right]\right)^\frac{1}{q}+ 	\left\|\mathcal{I}_3(t,0)\right\|_{L^{q}(\Omega,\mathbb{R})}\leq Ch^{1-\epsilon}.
		\end{aligned}
	\end{equation} 
	As to  $\mathcal{I}_4(t,x)$, we employ \eqref{ineq: Ito Isometry 2}, H{\"o}lder's inequality and  \eqref{es: G_s^N-G_s(x, y) } with $\alpha_2=2-\epsilon$ to get
	\begin{equation} \label{ineq: ||tilde{u^n}-u || 2.1}
		\begin{aligned}
			\left\|\mathcal{I}_4(t,x)\right\|_{L^{q}(\Omega,\mathbb{R})}
			\leq& C   \bigg(\int_{0}^{t}\bigg(\int_{\mathcal{O}}\left|G_{t-s}^N(x, y)-G_{t-s}(x, y)\right|^\frac{1}{H_2}\mathrm{d}y\bigg)^\frac{H_2}{H_1} \mathrm{d} s\bigg)^{H_1}   \\
			\leq& C \left( \int_{0}^{t} \int_{\mathcal{O}}\left|G_{t-s}^N(x, y)-G_{t-s}(x, y)\right|^2\mathrm{d} y  \mathrm{d} s   \right)^{\frac{1}{2}}\\
			\leq& Ch^{1-\frac{\epsilon}{2}}.
		\end{aligned}
	\end{equation}
	According to \eqref{ineq: futher property of O_t} and \eqref{ineq: futher property of O^N_t}, one has 
	\begin{equation*} 
		\begin{aligned}
			\left\|\mathcal{I}_4(t,x)-\mathcal{I}_4(t,z)\right\|_{L^{q}(\Omega,\mathbb{R})} 
			\leq C|x-z|.
		\end{aligned}
	\end{equation*} 
	Combining this with  \cite[Theorem 2.1]{revuz2013continuous}  leads to
	
	\begin{equation}   \label{ineq: ||tilde{u^n}-u || 2.2}
		\begin{aligned}
			\mathbb{E}\left[\sup_{x\neq z} \left(\frac{|\mathcal{I}_4(t,x)-\mathcal{I}_4(t,z)|}{|x-z|^{1-\frac{2}{q}}} \right)^q\right] \leq C.
		\end{aligned}
	\end{equation}  
	Then it can follow from \eqref{ineq: ||tilde{u^n}-u || 2.1} and \eqref{ineq: ||tilde{u^n}-u || 2.2} that
	\begin{equation}    \label{ineq: ||tilde{u^n}-u || 3}
		\begin{aligned}
			\left\|\mathcal{I}_4(t,\cdot)\right\|_{L^p(\Omega,L^\infty)} 
			\leq&\left\|\mathcal{I}_4(t,\cdot)\right\|_{L^q(\Omega,L^\infty)}\\ 
			\leq &
			C\left(\mathbb{E}\left[\sup_{1\leq i\leq N}\left| \mathcal{I}_4(t,x_i)\right|^q\right]\right)^\frac{1}{q} +C\left(\mathbb{E}\left[\sup_{1\leq i\leq N}\sup_{x \in[(i-1)h, ih]}\left|\mathcal{I}_4(t,x_i)-\mathcal{I}_4(t,x) \right|^q\right]\right)^\frac{1}{q}\\
			\leq &
			C\left(\sum_{i=1}^{N}\mathbb{E}\left[\left| \mathcal{I}_4(t,x_i)\right|^q\right]\right)^\frac{1}{q} +C\left(\sum_{i=1}^{N}\mathbb{E}\left[ \sup_{x \in[(i-1)h, ih]}\left|\mathcal{I}_4(t,x_i)-\mathcal{I}_4(t,x)\right|^q\right]\right)^\frac{1}{q}\\
			\leq &
			Ch^{1-\epsilon}.
		\end{aligned}
	\end{equation} 
	
	Evidently, \eqref{ineq: ||tilde{u^n}-u || 1} together with \eqref{ineq: ||tilde{u^n}-u || 2} and  \eqref{ineq: ||tilde{u^n}-u || 3} suffices to obtain \eqref{ineq: ||tilde{u^n}-u ||}. This completes the proof. 
\end{proof} 

Before proofing the Proposition \ref{lem: ||u^n-tilde{u^n} ||}, we have to give the moment boundedness of $\tilde{U}^N_t$.
\begin{lemma}\label{lem: ||tilder U || 1,N}  
	Suppose that Assumption \ref{assu:2} holds, and let $k$ be the largest integer smaller than $4H_1+H_2-1$.
	Then for $p \geq 1$ and $1\leq i\leq k $, there exists a constant $C>0$, independent of $N$,  such that
	\begin{equation}\label{ineq: ||tilder U || 1,N}
		\mathbb{E}\left[\big\| \tilde{U}^N_t 	\big\|_{i,N}^p\right]\leq C, \quad \forall\ t \in[0, T].
	\end{equation}  
\end{lemma} 
\begin{proof}
	See Appendix C. 
\end{proof}

\vspace{0.2cm}
Equipped with the previously derived lemma, we can prove Proposition \ref{lem: ||u^n-tilde{u^n} ||}.

\begin{proof}[Proof of Proposition \ref{lem: ||u^n-tilde{u^n} ||}]
	We denote $E^N_t:= \tilde{U}^N_t-U^N_t$. This allows us to derive the error equation 
	\begin{equation} \label{eq: U-tilder(U)^N_matrix}
		\left\{\begin{aligned} 
			&\mathrm{d}  E_t^N+A_N^2 E_t^N \mathrm{d}  t=A_N\left[F_N\left(U_t^N+E_t^N+(U_t-\tilde{U}^N_t)\right)-F_N\left(U_t^N\right)\right]\mathrm{d}  t,\ t \in(0, T], \\
			&E_0^N=0.
		\end{aligned}\right.
	\end{equation}
	For any $t\in[0,T]$, by utilizing  Theorem \ref{th:  u_regularity}, Theorem \ref{th: U^N regularity}, Lemma \ref{lem: ||tilder U || 1,N}, Proposition \ref{propo: err problem 1}  and Proposition \ref{lem: ||tilde{u^n}-u ||}, one can obtain
	\begin{equation*} \label{ineq: ||u^n-u|| 1} 
		\begin{aligned}
			\mathbb{E} \left[\left\|E_t^N\right\|_{l_N^2}^q\right]  
			\leq&  Ch^{q(1-\epsilon)}, \quad \forall \ q\geq1.
		\end{aligned}
	\end{equation*}
	Combining this with  \eqref{ineq: l_infty to l2},  \eqref{ineq: e^{-A_N^2t}(-A_N)^r to l2},
	\eqref{ineq: ||tilde{u^n}-u ||}, \eqref{ineq: u_sup_boundness} and \eqref{ineq: U_t_l^infty} shows 
	
	\begin{align*}
		\left\|u^N(t, \cdot)-\tilde{u}^N(t,\cdot)\right\|_{L^p(\Omega,L^\infty)}
		=& \left\|E^N_t\right\|_{L^p(\Omega,l^\infty_N)}
		\leq C \left(	\mathbb{E}\left[	\left\| E^N_t \right\|_{1,N}^p\right]
		\right)^\frac{1}{p}\\
		\leq&C\left\| \int_0^t\left\| (-A_N)^\frac{3}{2}  e^{-A_N^2(t-s)}\left[F_N\left(U_s\right)-F_N\left(U_s^N\right)\right] 
		\right\|_{l^2_N} \mathrm{d} s  \right\|_{L^{p}(\Omega,\mathbb{R})}+Ch^{1-\epsilon}\\ 
		\leq&C  \int_0^t(t-s)^{-\frac{3}{4}}  \left\|1+\|U_s\|_{l_N^{\infty}}^2+\|U_s^N\|_{l_N^{\infty}}^2\right\|_{L^{2p}\left(\Omega, \mathbb{R}\right)} \|U_s-U_s^N\|_{L^{2p}\left(\Omega, l_N^2\right)}  \mathrm{d} s +Ch^{1-\epsilon} \\ 
		\leq&C  \int_0^t(t-s)^{-\frac{3}{4}} \left(\|E^N_s\|_{L^{2p}\left(\Omega, l_N^2\right)} +\|U_t-\tilde{U}^N_t\|_{L^{2p}\left(\Omega, l_N^2\right)}\right)  \mathrm{d} s +Ch^{1-\epsilon} \\ 
		\leq&Ch^{1-\epsilon}.
	\end{align*} 
	This proof is completed.
\end{proof}

\section{Fully discrete scheme}\label{section5}
In this section, we present a fully discrete scheme and show its strong convergence rate.
Take a positive integer $M$. We define $\tau = T/M$ and $t_i = i\tau$ for $i = 0, 1, \ldots, M$. Let $\lfloor t \rfloor$ denote the largest time grid point no greater than $t$, such that $\lfloor t \rfloor=t_i$ for $t \in [t_i, t_{i+1}),\ i = 0, 1, \ldots, M-1$.  
We apply the tamed exponential Euler scheme to \eqref{eq:eq:discrete_CH_FractionalNoise_SDE_matrix} and obtain the full discretization
\begin{equation} \label{eq: mild_solution_full_discrete}
	\begin{aligned}
		U^{M,N}_t= & e^{-A_N^2 t}U_0+\int_0^t A_N e^{-A_N^2(t-\lfloor s\rfloor)} \tilde{F}_N(	U^{M,N}_{\lfloor s\rfloor})  \mathrm{d} s +\mathscr{O}_t^{M,N}, \quad \forall\ t \in[0, T].
	\end{aligned}
\end{equation}  
where $$\mathscr{O}_t^{M,N}:=\sqrt{\sigma^2 N/\pi}\int_0^t e^{-A_N^2(t-\lfloor s\rfloor)} \mathrm{d} \beta^H_s\quad \text{and} \quad\tilde{F}_N(w):=F_N(w)/(1+\tau\left\|F_N(w)\right\|_{l^2_N})  $$
for any $w\in l^2_N$.
Through the interpolation of  $U^{M,N}_t$, we can get
\begin{equation} \label{eq: full discretization}
	\begin{aligned}
		u^{M, N}(t, x):= & \int_{\mathcal{O}} G_t^N(x, y) u_0\left(\kappa_N(y)\right) \mathrm{d} y
		+\int_0^t \int_{\mathcal{O}}  \frac{\Delta_N G_{t-\lfloor s\rfloor}^N(x, y)f\left(u^{M, N}\left(\lfloor s\rfloor, \kappa_N(y)\right)\right)}{ 1+\tau \left\|f\left(u^{M, N}\left(\lfloor s\rfloor, \eta_N(\cdot)\right)\right) \right\|_{L^2}} \mathrm{~d} y \mathrm{d} s \\
		& +\sigma\int_0^t \int_{\mathcal{O}} G_{t-\lfloor s\rfloor}^N(x, y)  B^H(\mathrm{d} s, \mathrm{d} y) , \qquad\qquad\qquad\qquad\quad\ \forall \ (t, x) \in[0, T] \times \mathcal{O},
	\end{aligned}
\end{equation} 
where $o^{M,N}(t,x):=\sigma\int_{0}^{t}\int_{\mathcal{O}} G^N_{t-\lfloor s\rfloor}(x,y)   \  B^H(\mathrm{d} s, \mathrm{d} y)$.
\subsection{Moment boundedness of the fully discrete numerical solution}
An argument similar to the one used in Lemma \ref{lem: 3.1 regularity of o(t,x)} shows the regularity and continuity of stochastic convolution $\mathscr{O}^{M,N}_t$.
\begin{lemma}\label{lem: O^{M,N}_continuous}  
	For any $p \geq 1$ and $0\leq\beta<4H_1+H_2-1$, the fully discrete stochastic convolution $\mathscr{O}^{M,N}_t$   enjoys 
	the following regularity property
	\begin{equation}\label{ineq: regularity of O^{M,N}(t,x)}
		\mathbb{E}\left[\sup_{t\in[0,T]} \left\|(-A_N)^{\frac{\beta}{2}}\mathscr{O}^{M,N}_t \right\|^p_{l^2_N}\right] < \infty.
	\end{equation}
	Moreover,  there exists a constant $C>0$, independent of $M$ and $N$, such that for any $0\leq s<t\leq T$,
	\begin{equation}\label{ineq: continuous of O^{M,N}(t,x)}
		\mathbb{E}\left[\left\|(-A_N)^{\frac{\beta}{2}}\left(\mathscr{O}^{M,N}_t-\mathscr{O}^{M,N}_s\right) \right\|^p_{l^2_N}\right]\leq C(t-s)^{ \left(\frac{4H_1+H_2-1-\beta}{4}
			-\frac{\epsilon}{2}\right) p},
	\end{equation} 
	where $\epsilon$ is an arbitrary small positive number, and for any $x,\ z \in \mathcal{O}$,  
	\begin{equation*} 
		\sup_{t\in[0,T]}\mathbb{E}\left[|o^{M,N}(t, x)-o^{M,N}(t, z)|^p\right]\leq C|x-z|^p. 
	\end{equation*} 
\end{lemma}

Following Lemma \ref{lem: O^{M,N}_continuous}, we obtain the moment boundedness of $U^{M,N}_t$.
\begin{theorem} \label{thm: the moment boundedness of the numerical solution U^{M,N}}
	Suppose that Assumption \ref{assu:2} holds.
	Then for any $p \geq 2$, there exists a constant $C>0$ such that
	\begin{equation}\label{es: U^{M,N} }
		\sup _{M, N \in \mathbb{N}} \sup _{0 \leq i \leq M} \mathbb{E}\left[\left\|U_{t_i}^{M, N}\right\|_{1,N}^p\right]\leq C.
	\end{equation}
\end{theorem}

\begin{proof}
	We introduce a sequence of decreasing events
	$$
	\Omega_{R, t_i}:=\left\{\omega \in \Omega: \sup _{0 \leq j \leq i}\left\|U_{t_j}^{M, N}\right\|_{1,N} \leq R\right\}, \quad\forall \ R \in(0, \infty),\ 0 \leq i \leq M.
	$$ 
	$\textbf{B}^C$ and $\mathbb{I}_\textbf{B}$ are employed to represent the complement and indicator functions of a specified set $\textbf{B}$, respectively.\par 
	The forthcoming proof can be segmented into two steps.\par 
	\textbf{Step 1: }
	For any $q\in[2,\infty),$ we verify 
	\begin{equation} \label{es: U^{M,N} I_Omega 1}
		\begin{aligned}
			\sup _{M, N \in \mathbb{N}} \sup _{0 \leq i \leq M} \mathbb{E}\left[ \mathbb{I}_{\Omega_{R, t_i}}\left\|U_{t_i}^{M, N}\right\|_{1,N}^q\right]\leq C.
		\end{aligned}
	\end{equation} 
	The proof of \eqref{es: U^{M,N} I_Omega 1} heavily relies on the utilization of Proposition \ref{propo: perturbed problem}. We introduce the process  $\{V^{M,N}_t\}_{t\in[0,T]}$ by
	\begin{equation} \label{eq: V_t^{M, N} defination}
		\begin{aligned}
			V_t^{M, N}:= & e^{-A_N^2t} U_0 +\int_0^t A_N e^{-A_N^2(t-\lfloor s\rfloor)}\tilde{F}_N\big(U_{\lfloor s\rfloor}^{M, N}\big)\mathrm{d} s -\int_0^t A_N e^{-A_N^2(t-s)}F_N\left(U_s^{M, N}\right)\mathrm{d} s+\mathscr{O}_t^{M,N}.
		\end{aligned}
	\end{equation}
	Let $\bar{U}^{M,N}_t:=U^{M,N}_t-V^{M,N}_t.$ Then one can have
	\begin{equation*} 
		\left\{  
		\begin{aligned}
			&\mathrm{d} \bar{U}^{M,N}_t+A_N^2 \bar{U}^{M,N}_t \mathrm{d} t=A_N F_N(\bar{U}^{M,N}_t+V^{M,N}_t) \mathrm{d} t, \quad t\in (0,T],\\
			&\bar{U}^{M,N}_0=0.
		\end{aligned}
		\right.
	\end{equation*}
	It follows from Proposition \ref{propo: perturbed problem} that 
	
	\begin{equation}\label{es:  bar{U} to V}
		\left\|\bar{U}^{M,N}_t\right\|_{1,N}\leq C\left( 1+ \sup_{s\in[0, t]}\left\|V^{M,N}_s\right\|_{l^\infty_N}^{27}\right), \quad \forall\ t \in[0, T].  
	\end{equation} 
	For $t \in\left[0, t_i\right], 0 \leq i \leq M$, using \eqref{eq: V_t^{M, N} defination} we have
	\begin{equation}  \label{es: Omega V^M,N    1}
		\begin{aligned}
			\mathbb{I}_{\Omega_{R, t_{i-1}}}\left\|V_t^{M, N}\right\|_{1,N}\leq & \left\|e^{-A_N^2 t} U_0\right\|_{1,N}+\mathbb{I}_{\Omega_{R, t_{i-1}}}\left\|\int_0^t A_N e^{-A_N^2(t-\lfloor s\rfloor)}\left[\tilde{F}_N\left(U_{\lfloor s\rfloor}^{M, N}\right)-F_N\left(U_{\lfloor s\rfloor}^{M, N}\right)\right] \mathrm{d} s\right\|_{1,N} \\
			& +\mathbb{I}_{\Omega_{R, t_{i-1}}}\left\|\int_0^t A_N e^{-A_N^2(t-\lfloor s\rfloor)}\left[F_N\left(U_{\lfloor s\rfloor}^{M, N}\right)-F_N\left(U_s^{M, N}\right)\right] \mathrm{d} s\right\|_{1,N} \\
			& +\mathbb{I}_{\Omega_{R, t_{i-1}}}\left\|\int_0^t A_N e^{-A_N^2(t-s)}\left(I-e^{-A_N^2(s-\lfloor s\rfloor)}\right) F_N\left(U_s^{M, N}\right) \mathrm{d} s\right\|_{1,N}+\left\|\mathscr{O}_t^{M, N}\right\|_{1,N}\\
			=&: \mathcal{J}_1+\mathcal{J}_2+\mathcal{J}_3+\mathcal{J}_4+\left\|\mathscr{O}_t^{M, N}\right\|_{1,N} .
		\end{aligned}
	\end{equation} 
	Using  \eqref{ineq: e^{-A_N^2t}(-A_N)^r to l2} and \eqref{es:   U_0 k,N}, one can have
	\begin{equation} \label{es: I_1   U^MN  1} \mathcal{J}_1\leq C \left\| U_0\right\|_{1,N}\leq C.\end{equation}
	From \eqref{ineq: e^{-A_N^2t}(-A_N)^r to l2}, we derive
	\begin{align}
		\mathcal{J}_2\leq& \mathbb{I}_{\Omega_{R, t_i}} \int_0^t   \tau\left\| F\big(U_{\lfloor s\rfloor}^{M, N}\big)\right\|_{l^2_N}   \left\| A_Ne^{-A_N^2(t-\lfloor s\rfloor)} F\big(U_{\lfloor s\rfloor}^{M, N}\big)\right\|_{1,N}\mathrm{d} s\nonumber  \\
		\leq& \mathbb{I}_{\Omega_{R, t_i}}C \tau(1+R^3)\int_0^t
		\left(t-\lfloor s\rfloor\right)^{-\frac{3}{4}} 
		\left\|  F\big(U_{\lfloor s\rfloor}^{M, N}\big)\right\|_{l^2_N}   \mathrm{d} s  
		\leq C\tau \left(1+R^6\right). \label{es: I_2   U^MN}
	\end{align}
	
	To estimate $	\mathcal{J}_3$, we rewrite \eqref{eq: mild_solution_full_discrete} as
	$$
	U_t^{M, N}=e^{-A_N^2(t-\lfloor t\rfloor)}U_{\lfloor t\rfloor}^{M, N}
	+(t-\lfloor t\rfloor) A_Ne^{-A_N^2(t-\lfloor t\rfloor)}\tilde{F}_N\big(U_{\lfloor t\rfloor}^{M, N}\big)  
	+\mathscr{O}_t^{M,N}-e^{-A_N^2(t-\lfloor t\rfloor)} \mathscr{O}_{\lfloor t\rfloor}^{M,N},\ \forall t \in [0,t_i].
	$$
	For any $s<t$,	combining this with \eqref{ineq: l_infty to l2} and \eqref{ineq: e^{-A_N^2t}(-A_N)^r to l2}, we obtain 
	\begin{equation} \label{es: I_{Omega_{R, t_i}} U_t^{M, N}}
		\begin{aligned}
			\mathbb{I}_{\Omega_{R, t_{i-1}}}\left\|U_s^{M, N}\right\|_{1,N} 
			\leq &\mathbb{I}_{\Omega_{R, t_{i-1}}}\bigg(\left\|e^{-A_N^2(s-\lfloor s\rfloor)} U_{\lfloor s\rfloor}^{M, N} \right\|_{1,N}
			+ (s-\lfloor s\rfloor)\left\|A_Ne^{-A_N^2(s-\lfloor s\rfloor)}\tilde{F}_N\big(U_{\lfloor s\rfloor}^{M, N}\big)\right\|_{1,N} \\
			&\qquad\quad +\left\| \mathscr{O}_s^{M,N}\right\|_{1,N}
			+\left\|e^{-A_N^2(s-\lfloor s\rfloor)}  \mathscr{O}_{\lfloor s\rfloor}^{M,N}\right\|_{1,N}\bigg) \\
			\leq &  \mathbb{I}_{\Omega_{R, t_{i-1}}}C \bigg(  \left\|U_{\lfloor s\rfloor}^{M, N} \right\|_{1,N}
			+(s-\lfloor s\rfloor)^\frac{1}{4}\left\|\tilde{F}_N\big(U_{\lfloor s\rfloor}^{M, N}\big)\right\|_{l^2_N}  +\left\| \mathscr{O}_s^{M,N}\right\|_{1,N}
			+\left\| \mathscr{O}_{\lfloor s\rfloor}^{M,N}\right\|_{1,N}\bigg) \\ 
			\leq & C\left(R+\tau^{\frac{1}{4}} (1+R^3)+\mathscr{R}_s  \right),
		\end{aligned}
	\end{equation}
	where $\mathscr{R}_s:=\left\| \mathscr{O}_s^{M,N}\right\|_{1,N}+\left\| \mathscr{O}_{\lfloor s\rfloor}^{M,N}\right\|_{1,N}$.
	Utilizing \eqref{eq: mild_solution_full_discrete}, we can get 
	\begin{align}
		& \left\|U_s^{M, N}-U_{\lfloor s\rfloor}^{M, N}\right\|_{l_N^2} \nonumber\\
		\leq&\left\|\left(e^{-A_N^2s}- e^{-A_N^2\lfloor s\rfloor}\right)U_0\right\|_{l^2_N} 
		+\int_{0}^{\lfloor s\rfloor}\left\|A_Ne^{-A_N^2(\lfloor s\rfloor-\lfloor r\rfloor)}\left(I-e^{-A_N^2(s-\lfloor s\rfloor)}\right) \tilde{F}_N\big(U_{\lfloor r\rfloor}^{M, N}\big) \right\|_{l^2_N}  \mathrm{d}r\nonumber\\
		& +(s-\lfloor s\rfloor)\left\|A_Ne^{-A_N^2(s-\lfloor s\rfloor)}\tilde{F}_N\big(U_{\lfloor s\rfloor}^{M, N}\big)\right\|_{l^2_N}  
		+\left\|\mathscr{O}_s^{M,N}-\mathscr{O}_{\lfloor s\rfloor }^{M,N}\right\|_{l_N^2}\nonumber \\
		=&: \mathscr{K}_1+\mathscr{K}_2+\mathscr{K}_3+\left\|\mathscr{O}_s^{M,N}-\mathscr{O}_{\lfloor s\rfloor }^{M,N}\right\|_{l_N^2} . \label{eq: U_t^{M, N} continuous   1}
	\end{align}
	
	By combining \eqref{ineq: (-A_N)^re^{-A_N^2t}(I- e^{-A_N^2t}) to l2} with \eqref{es:   U_0 4,N}, we have
	\begin{equation} \label{es: eq: U_t^{M, N} continuous  I1}
		\begin{aligned}
			\mathscr{K}_1 = \left\|e^{ -A_N^2 \lfloor s\rfloor}\left(e^{-A_N^2( s-\lfloor s\rfloor)}-I\right)U_0\right\|_{l^2_N}    
			\leq C( s-\lfloor s\rfloor) \left\| (-A_N)^2U_0\right\|_{l^2_N} 
			\leq C\tau.
		\end{aligned}
	\end{equation}
	It follows from \eqref{ineq: (-A_N)^re^{-A_N^2t}(I- e^{-A_N^2t}) to l2} and \eqref{ineq:  F local Lipschitz   2} that
	\begin{equation} \label{es: eq: U_t^{M, N} continuous  I2}
		\begin{aligned}
			\mathscr{K}_2    	\leq&  C\int_0^{ \lfloor s\rfloor }(\lfloor s\rfloor-\lfloor r\rfloor)^{-1+\epsilon}(s-\lfloor s\rfloor)^{\frac{3}{4}-\epsilon}
			\left\|   (-A_N)^{\frac{1}{2}}\tilde{F}_N\big(U_{\lfloor r\rfloor}^{M, N}\big) \right\|_{l^2_N}  \mathrm{d}r \\
			\leq&  C\tau^{\frac{3}{4}-\epsilon} \int_0^{ \lfloor s\rfloor }(\lfloor s\rfloor-r)^{-1+\epsilon} \left(1+\left\|  U_{\lfloor r\rfloor}^{M, N}\right\|_{1,N}^3 \right)  \mathrm{d} r ,
		\end{aligned}
	\end{equation}
	where $\epsilon$ is an arbitrary small positive number. 
	Utilizing \eqref{ineq: e^{-A_N^2t}(-A_N)^r to l2}  and Lemma \ref{lem: F local Lipschitz}, one can have
	\begin{equation}\label{es: eq: U_t^{M, N} continuous  I3}
		\begin{aligned}
			\mathscr{K}_3
			\leq&  
			C (s-\lfloor s\rfloor)^\frac{3}{4}\left\|   (-A_N)^{\frac{1}{2}}\tilde{F}_N\big(U_{\lfloor s\rfloor}^{M, N}\big) \right\|_{l^2_N}      
			\leq   C\tau^{\frac{3}{4}}\left(1+\left\| U_{\lfloor s\rfloor}^{M, N}\right\|_{1,N}^3 \right). 
		\end{aligned}
	\end{equation}
	By substituting \eqref{es: eq: U_t^{M, N} continuous  I1}-\eqref{es: eq: U_t^{M, N} continuous  I3} into \eqref{eq: U_t^{M, N} continuous   1}, we obtain  
	\begin{equation}
		\begin{aligned} \label{es: I_{Omega_{R, t_i}} U_t^{M, N}-U_t^{M, N}}
			\mathbb{I}_{\Omega_{R, t_{i-1}}}\left\|	U_s^{M, N}-U_{\lfloor s\rfloor}^{M, N}\right\|_{l^2_N}  
			\leq&C\tau^{\frac{3}{4}-\epsilon}(1+R^3)+
			\left\|\mathscr{O}_s^{M,N}- \mathscr{O}_{\lfloor s\rfloor}^{M,N}\right\|_{l^2_N},
		\end{aligned}
	\end{equation}
	which combined with \eqref{ineq: l_infty to l2} and \eqref{es: I_{Omega_{R, t_i}} U_t^{M, N}}  yields
	\begin{equation} \label{es: I_3   U^MN}
		\begin{aligned}  
			\mathcal{J}_3\leq&  \mathbb{I}_{\Omega_{R, t_{i-1}}}C\int_0^t (t-s)^{-\frac{3}{4}}\left\|F\big(U_{\lfloor s\rfloor}^{M, N}\big)-F\left(U_s^{M, N}\right)\right\|_{l^2_N} \mathrm{d} s\\
			\leq&  \mathbb{I}_{\Omega_{R, t_{i-1}}}C\int_0^t (t-s)^{-\frac{3}{4}} \bigg(1+	\left\|U_{\lfloor s\rfloor}^{M, N}\right\|^2_{l^\infty_N}+	\left\|U_{s}^{M, N}\right\|^2_{l^\infty_N}\bigg) 
			\left\|U_{\lfloor s\rfloor}^{M, N}-U_s^{M, N}\right\|_{l^2_N}\mathrm{d} s\\
			\leq&  \mathbb{I}_{\Omega_{R,t_{i-1}}}C\int_0^t (t-s)^{-\frac{3}{4}}
			\left(1+R^2+\tau^{\frac{1}{2}}(1+ R^6)+\mathscr{R}_s^2\right)
			\left\|U_{\lfloor s\rfloor}^{M, N}-U_s^{M, N}\right\|_{l^2_N} \mathrm{d} s\\
			\leq&C\tau^{\frac{3}{4}-\epsilon}(1+R^3)\int_0^t (t-s)^{-\frac{3}{4}}
			\left(1+R^2+\tau^{\frac{1}{2}}(1+ R^6)+\mathscr{R}_s^2\right) \mathrm{d} s\\
			&+C\int_0^t (t-s)^{-\frac{3}{4}}	\left(1+R^2+\tau^{\frac{1}{2}}(1+ R^6)+\mathscr{R}_s^2\right)\left\|\mathscr{O}_s^{M,N}- \mathscr{O}_{\lfloor s\rfloor}^{M,N}\right\|_{l^2_N}\mathrm{d} s.
		\end{aligned}
	\end{equation}  
	It follows from   \eqref{ineq: (-A_N)^re^{-A_N^2t}(I- e^{-A_N^2t}) to l2}, \eqref{eq: U_t^{M, N} continuous   1}  and Lemma \ref{lem: F local Lipschitz} that
	\begin{equation}   \label{es: I_4   U^MN}
		\begin{aligned}
			\mathcal{J}_4 \leq &
			\mathbb{I}_{\Omega_{R, t_{i-1}}}\int_0^t  \left\|A_N e^{-A_N^2(t- s)}\left(I-e^{-A_N^2(s-\lfloor s\rfloor)}\right)F\left(U_s^{M, N}\right)\right\|_{1,N}   \mathrm{d} s  \\
			\leq & C\mathbb{I}_{\Omega_{R, t_{i-1}}}\int_0^t  (t-s)^{-\frac{3}{4}}
			(s-\lfloor s\rfloor)^{\frac{1}{4}}
			\left\|(-A_N)^\frac{1}{2}  F\left(U_s^{M, N}\right)\right\|_{l^2_N}   \mathrm{d} s  \\
			\leq & C\tau^{\frac{1}{4}}\mathbb{I}_{\Omega_{R, t_{i-1}}}\int_0^t  (t-s)^{-\frac{3}{4}}
			\left(1+\left\| U_s^{M, N} \right\|_{1,N}^3\right)   \mathrm{d} s   \\
			\leq & C\tau^{\frac{1}{4}}(1+ R^3). \\
		\end{aligned}
	\end{equation}
	By setting $R=\tau^{-\frac{1}{12}}$, it can be inferred from Lemma \ref{lem: O^{M,N}_continuous} , \eqref{es: Omega V^M,N    1}--\eqref{es: I_2   U^MN}, \eqref{es: I_3   U^MN} and  \eqref{es: I_4   U^MN} that 
	
	$$\mathbb{E} \left[\sup_{s\in[0, t]}\mathbb{I}_{\Omega_{R, t_{i-1}}}\left\|V^{M,N}_s\right\|_{1,N}^{27q}\right]\leq C.$$ 
	Consequently, it follows from this inequality and \eqref{es:  bar{U} to V} that
	\begin{equation} \label{es: U^{M,N} I_Omega i-1}
		\begin{aligned}
			\sup _{M, N \in \mathbb{N}} \sup _{0 \leq i \leq M} \mathbb{E}\left[ \mathbb{I}_{\Omega_{R, t_{i-1}}}\left\|U_{t_i}^{M, N}\right\|_{1,N}^q\right]\leq C,
		\end{aligned}
	\end{equation} 
	which yields \eqref{es: U^{M,N} I_Omega 1}.\par 
	\textbf{Step 2: } We show 
	\begin{equation} \label{es: U^{M,N} I_Omega }
		\begin{aligned}
			\sup _{M, N \in \mathbb{N}} \sup _{0 \leq i \leq M} \mathbb{E}\left[ \mathbb{I}_{\Omega^C_{R, t_i}} \left\|U_{t_i}^{M, N}\right\|_{1,N}^p \right]\leq C.
		\end{aligned}
	\end{equation} 
	It is clear that 
	\begin{equation*}
		\begin{aligned}
			\mathbb{I}_{\Omega_{R, t_m}^c}=&\mathbb{I}_{\Omega_{R, t_{m-1}}^c}+\mathbb{I}_{\Omega_{R, t_{m-1}}} \cdot \mathbb{I}_{\big\{\left\|U_{t_m}^{M, N}\right\|_{1,N}>R\big\}}
			=\sum_{i=1}^m \mathbb{I}_{\Omega_{R, t_{i-1}}} \cdot \mathbb{I}_{\big\{\left\|U_{t_i}^{M, N}\right\|_{1,N} >R\big\}}.
		\end{aligned}
	\end{equation*}
	By \eqref{eq: mild_solution_full_discrete}, \eqref{ineq: e^{-A_N^2t}(-A_N)^r to l2}, \eqref{es:   U_0 k,N} and Lemma \ref{lem: O^{M,N}_continuous}, it is simple to show that 
	
	$$ \mathbb{E}\Big[\left\|U_{t_m}^{M, N}\right\|_{1,N}^{2p}\Big] \leq C\left(1+\tau^{-2p}\right). $$
	Additionally, setting $R=\tau^{-\frac{1}{12}}$ and utilizing the Chebyshev inequality and \eqref{es: U^{M,N} I_Omega i-1}, we can obtain 
	\begin{align*}
		\mathbb{E}\left[\mathbb{I}_{\Omega_{R, t_m}^c}\left\|U_{t_m}^{M, N}\right\|_{1,N}^p\right] =&
		\sum_{i=1}^m \mathbb{E}\bigg[\left\|U_{t_m}^{M, N}\right\|_{1,N}^p\cdot \mathbb{I}_{\Omega_{R, t_{i-1}}} \mathbb{I}_{\big\{\left\|U_{t_i}^{M, N}\right\|_{1,N}>R\big\}}\bigg] \\
		\leq&  \sum_{i=1}^m\left(\mathbb{E}\Big[\left\|U_{t_m}^{M, N}\right\|_{1,N}^{2p}\Big]\right)^{\frac{1}{2}}\bigg(\mathbb{E}\Big[\mathbb{I}_{\Omega_{R, t_{i-1}}} \mathbb{I}_{\big\{\left\|U_{t_i}^{M, N}\right\|_{1,N}>R\big\}}\Big]\bigg)^{\frac{1}{2}} \\
		\leq& \sum_{i=1}^m C\left(1+\tau^{-p}\right)\left(\mathbb{P}\left(\omega \in \Omega_{R, t_{i-1}}:\ \left\|U_{t_i}^{M, N}\right\|_{1,N}>R\right)\right)^{\frac{1}{2}} \\ 
		\leq& C\left(1+\tau^{-p}\right) \sum_{i=1}^m\Bigg(\frac{\mathbb{E}\Big[\mathbb{I}_{\Omega_{R, t_{i-1}}}\left\|U_{t_i}^{M, N}\right\|_{1,N}^{24(p+1)} \Big]}{ R^{24(p+1)} }   \Bigg)^{\frac{1}{2}} \\
		\leq& C.
	\end{align*} 
	The proof is completed.
\end{proof}
\begin{remark}
	Based on Theorem \ref{thm: the moment boundedness of the numerical solution U^{M,N}} and Lemma \ref{lem: O^{M,N}_continuous}, one can repeat the analysis of \eqref{eq: U_t^{M, N} continuous   1} directly to show that  
	\begin{equation}\label{ineq: continuous of U^{M,N}(t,x)}
		\| U_t^{M, N}-U_{\lfloor t\rfloor}^{M, N}  \|_{L^p(\Omega,l^2_N)} \leq C\tau^{(H_1+\frac{H_2}{4}-\frac{1}{4}  
			-\frac{\epsilon}{2})\wedge(\frac{3}{4}-\epsilon) }, \qquad\forall\ t\in[0,T],
	\end{equation}
	where $\epsilon$ is an arbitrary small positive number.
\end{remark}
Further, we can improve the moment boundedness of $U^{M,N}_t$.
\begin{theorem}\label{thm: U^{M,N}_regularity}  
	Suppose that Assumption \ref{assu:2} holds, and let $k$ be the largest integer smaller than $4H_1+H_2-1$.	Then for $p \geq 1$ and $1\leq i\leq k $, there exists a constant $C>0$, independent of $M$ and $N$, such that 
	\begin{equation}\label{ineq: regularity of U^{M,N}(t,x)}
		\mathbb{E}\left[	\left\| U_t^{M, N} \right\|_{i,N}^p\right]\leq C, \quad \forall\ t \in[0, T]. 
	\end{equation}
\end{theorem}
\begin{proof}
	For  $t \in\left[0, T\right]$ and $p \geq 1$, combining \eqref{es: U^{M,N} } with \eqref{ineq: continuous of U^{M,N}(t,x)} yields
	\begin{equation}\label{ineq: regularity of U^{M,N}(t,x)  1}
		\mathbb{E}\left[	\left\| U_t^{M, N} \right\|_{l^2_N}^p\right]\leq C	\bigg(\mathbb{E}\left[	\left\| U_t^{M, N}-U_{\lfloor t\rfloor}^{M, N} \right\|_{l^2_N}^p\right]+	\mathbb{E}\left[	\left\| U_{\lfloor t\rfloor}^{M, N} \right\|_{l^2_N}^p\right]\bigg) \leq C.
	\end{equation}
	By the similar argument in the proof  of Lemma \ref{lem: ||tilder U || 1,N}, one can improve the moment boundedness of $ U_t^{M, N}$ to get \eqref{ineq: regularity of U^{M,N}(t,x)}.
\end{proof}

\begin{remark}
	Due to Theorem \ref{thm: U^{M,N}_regularity}, \eqref{ineq: continuous of U^{M,N}(t,x)} can be improved into 
	\begin{equation}\label{ineq: continuous of U^{M,N}(t,x) improve}
		\| U_t^{M, N}-U_{\lfloor t\rfloor}^{M, N}  \|_{L^p(\Omega,l^2_N)} \leq C\tau^{ H_1+\frac{H_2}{4}-\frac{1}{4}  
			-\frac{\epsilon}{2} }, \quad\forall\ t\in[0,T],
	\end{equation}
	where $\epsilon$ is an arbitrary small positive number.
\end{remark}
\subsection{Convergence of the fully discrete numerical solution}
To analyze the error of the fully discrete scheme, we introduce an auxiliary process
\begin{equation} \label{eq: solution_auxiliary_full_discrete}
	\begin{aligned}
		\tilde{U}^{M,N}_t= & e^{-A_N^2 t}U_0+\int_0^t A_N e^{-A_N^2(t-s)} F_N(	U^{M,N}_{s}) \mathrm{d} s+\mathscr{O}^N_t, \quad t \in[0, T].
	\end{aligned}
\end{equation}  
Base on Theorem \ref{thm: U^{M,N}_regularity}, we can obtain the moment boundedness of the auxiliary process $\tilde{U}^{M,N}_t$  in the following lemma. The proof of this result is quite similar to that given in Lemma \ref{lem: ||tilder U || 1,N} and so is omitted.
\begin{lemma}\label{lem: ||tilder U^{M,N} || 1,N}  
	Suppose that Assumption \ref{assu:2} holds, and let $k$ be the largest integer smaller than $4H_1+H_2-1$.
	Then for $p \geq 1$ and $1\leq i\leq k $, there exists a constant $C>0$, independent of $N$,  such that
	\begin{equation}\label{ineq: ||tilder U^{M,N} || 1,N}
		\mathbb{E}\left[\big\| \tilde{U}^{M,N}_t 	\big\|_{i,N}^p\right]\leq C,  \quad \forall\ t \in[0, T].
	\end{equation}  
\end{lemma}

The main result of this subsection is given in the following theorem.
\begin{theorem}\label{thm: ||u^{M,N}-u||}  
	Suppose that Assumption \ref{assu:2} holds and $h^{-1}\tau^9\leq1$.
	Then for $p \geq 1$, there exists a constant $C>0$, independent of $M$ and $N$,  such that  
	\begin{equation}\label{ineq: ||u^{M,N}-u||}
		\sup_{t \in[0, T]}\left\|u(t,\cdot)-u^{M,N}(t, \cdot)\right\|_{L^p(\Omega,L^\infty)}   \leq C\left(h^{1-\epsilon}+ \tau^{H_1-\frac{1}{8}-\frac{\epsilon}{2}} \right),
	\end{equation} 
	where $\epsilon$ is an arbitrary small positive number.
\end{theorem}
In order to estimate the error of the full discrete scheme, we decompose $\left\|u^{M,N}(t,\cdot)-u^{N}(t, \cdot)\right\|_{L^p(\Omega,L^\infty)}$ into two components:
\begin{equation*}
	\begin{aligned}
		\left\|u^{M,N}(t,\cdot)-u^{N}(t, \cdot)\right\|_{L^p(\Omega,L^\infty)}&\leq \left\|U^{M,N}_t-{U}^{N}_t\right\|_{L^p(\Omega,l^\infty_N)}\leq \left\|U^{M,N}_t-\tilde{U}^{M,N}_t\right\|_{L^p(\Omega,l^\infty_N)}+\left\|\tilde{U}^{M,N}_t-U^N_t\right\|_{L^p(\Omega,l^\infty_N)}.
	\end{aligned}
\end{equation*}
Therefore, by virture of Theorem \ref{thm: ||u^n-u||},  Theorem \ref{thm: ||u^{M,N}-u||}   follows directly from the two auxiliary results stated below.
\begin{proposition}\label{prop: ||tilde{u^{M,N}}-u^{M,N} ||}  
	Suppose that Assumption \ref{assu:2} holds and $h^{-1}\tau^9\leq1$.
	Then for $p \geq 1$, there exists a constant $C>0$, independent of $M$ and $N$, such that  such that  
	\begin{equation}\label{ineq: ||tilde{u^{M,N}}-u^{M,N} ||}
		\left\|U^{M,N}_t-\tilde{U}^{M,N}_t\right\|_{L^p(\Omega,l^\infty_N)} \leq C\left(h^{1-\epsilon}+ \tau^{H_1-\frac{1}{8}-\frac{\epsilon}{2}} \right),\quad \forall\ t \in[0, T],
	\end{equation} 
	where $\epsilon$ is an arbitrary small positive number.
\end{proposition}
\begin{proposition}\label{prop: ||u^N-tilde{u^{M,N}} ||}  
	Suppose that Assumption \ref{assu:2} holds and $h^{-1}\tau^9\leq1$.
	Then for $p \geq 1$, there exists a constant $C>0$, independent of $M$ and $N$, such that  such that  
	\begin{equation}\label{ineq: ||u^N-tilde{u^{M,N}} ||}
		\left\|\tilde{U}^{M,N}_t-U^{N}_t\right\|_{L^p(\Omega,l^\infty_N)} \leq C\left(h^{1-\epsilon}+ \tau^{H_1-\frac{1}{8}-\frac{\epsilon}{2}} \right),\quad \forall\ t \in[0, T],
	\end{equation} 
	where $\epsilon$ is an arbitrary small positive number.
\end{proposition}

Now it remains to prove Propositions \ref{prop: ||tilde{u^{M,N}}-u^{M,N} ||} and \ref{prop: ||u^N-tilde{u^{M,N}} ||}.

\begin{proof}[Proof of Propositions \ref{prop: ||tilde{u^{M,N}}-u^{M,N} ||}]
	Subtracting \eqref{eq: solution_auxiliary_full_discrete} from \eqref{eq: mild_solution_full_discrete}, one has 
	\begin{equation*} 
		\begin{aligned}
			U^{M,N}_t-\tilde{U}^{M,N}_t=\mathcal{L}_1+\mathcal{L}_2+\mathcal{L}_3+\mathcal{L}_4,
		\end{aligned}
	\end{equation*}
	where
	\begin{equation*} 
		\begin{aligned}
			&\mathcal{L}_1=-\int_0^t A_N e^{-A_N^2(t- s)}\left(I-e^{-A_N^2(s-\lfloor s\rfloor)}\right)\tilde{F}_N\left(U_{\lfloor s\rfloor}^{M, N}\right) \mathrm{d} s,\ \mathcal{L}_2=\int_0^t  A_N e^{-A_N^2(t-s)}\left[\tilde{F}_N\big(U_{\lfloor s\rfloor}^{M, N}\big)- F_N\big(U_{\lfloor s\rfloor}^{M, N}\big) \right] \mathrm{d} s,\\
			&\mathcal{L}_3=\int_0^t A_N e^{-A_N^2(t-s)}\left[F_N\big(U_{\lfloor s\rfloor}^{M, N}\big)-F_N\left(U_s^{M, N}\right)\right] \mathrm{d} s,\ \mathcal{L}_4=\mathscr{O}^{M,N}_t-\mathscr{O}^N_t.
		\end{aligned}
	\end{equation*}
	Let $k$ be the largest integer smaller than $4H_1+H_2-1$, then $4H_1+H_2-1-k\leq1$.
	It follows from  \eqref{ineq: l_infty to l2}, \eqref{ineq: (-A_N)^re^{-A_N^2t}(I- e^{-A_N^2t}) to l2} and \eqref{ineq:  F local Lipschitz   2}  that
	\begin{equation*} 
		\begin{aligned}
			\left\|\mathcal{L}_1\right\|_{l^\infty_N} \leq &
			C\int_0^t  \left\|A_N e^{-A_N^2(t- s)}\left(I-e^{-A_N^2(s-\lfloor s\rfloor)}\right)\tilde{F}_N\left(U_{\lfloor s\rfloor}^{M, N}\right)\right\|_{1,N}   \mathrm{d} s  \\
			\leq & C\int_0^t  (t-s)^{-\frac{3}{4}-\frac{1}{4}(4H_1+H_2-1-2\epsilon-k)}
			(s-\lfloor s\rfloor)^{\frac{4H_1+H_2-1-2\epsilon}{4}}
			\left\|(-A_N)^\frac{k}{2}  F\left(U_{\lfloor s\rfloor}^{M, N}\right)\right\|_{l^2_N}   \mathrm{d} s  \\
			\leq & C\tau^{H_1+\frac{H_2}{4}-\frac{1}{4}-\frac{\epsilon}{2}}\int_0^t  (t-s)^{-\frac{3}{4}-\frac{1}{4}(4H_1+H_2-1-2\epsilon-k)} 
			\left(1+\left\| U_{\lfloor s\rfloor}^{M, N} \right\|_{k,N}^3\right)   \mathrm{d} s.  \\
		\end{aligned}
	\end{equation*}
	Utilizing \eqref{ineq: l_infty to l2}, \eqref{ineq: e^{-A_N^2t}(-A_N)^r to l2} and \eqref{ineq:  F local Lipschitz}, we obtain 
	\begin{equation*} 
		\begin{aligned}
			\left\|\mathcal{L}_2\right\|_{l^\infty_N}	\leq&  C\int_0^t  \left\| A_N e^{-A_N^2(t-s)}\left[\tilde{F}_N\big(U_{\lfloor s\rfloor}^{M, N}\big)- F\big(U_{\lfloor s\rfloor}^{M, N}\big) \right] \right\|_{1,N}   \mathrm{d} s  \\
			\leq&C\int_0^t (t-s)^{-\frac{3}{4}} \left\|  \tilde{F}_N\big(U_{\lfloor s\rfloor}^{M, N}\big)- F\big(U_{\lfloor s\rfloor}^{M, N}\big)  \right\|_{l^2_N}  \mathrm{d} s\\
			\leq&C\tau\int_0^t (t-s)^{-\frac{3}{4}} \left\|   F\big(U_{\lfloor s\rfloor}^{M, N} \big)  \right\|_{l^2_N}^2 \mathrm{d} s \\
			\leq&C\tau\int_0^t (t-s)^{-\frac{3}{4}}\left( 1+\left\| U_{\lfloor s\rfloor}^{M, N}   \right\|_{1,N}^6\right) \mathrm{d} s, 
		\end{aligned}
	\end{equation*}
	and 
	
	\begin{align*}
		\left\|\mathcal{L}_3\right\|_{l^\infty_N}  
		\leq&\int_0^t  \left\| A_N e^{-A_N^2(t-s)}\left[F\big(U_{\lfloor s\rfloor}^{M, N}\big)-F\left(U_s^{M, N}\right)\right]\right\|_{1,N}  \mathrm{d} s \\
		\leq&C\int_0^t (t-s)^{-\frac{3}{4}} \left\| F\big(U_{\lfloor s\rfloor}^{M, N}\big)-F\left(U_s^{M, N}\right)  \right\|_{l^2_N}  \mathrm{d} s\\ 
		\leq&C\int_0^t (t-s)^{-\frac{3}{4}}
		\left(1+\left\|U_{\lfloor s\rfloor}^{M, N}\right\|_{1,N}^2+ \left\|U_s^{M, N}\right\|_{1,N}^2\right)
		\left\|U_s^{M, N}-U_{\lfloor s\rfloor}^{M, N}\right\|_{l^2_N}  \mathrm{d} s .
	\end{align*} 
	For any $q>\max\left\{ p+2,12/\epsilon \right\},\ t\in[0,T]$ and $x\in \mathcal{O}$, by \eqref{ineq: Ito Isometry}, H{\"o}lder's inequality, \eqref{estimate: e^{-x} 2} and \eqref{estimate: e^{-x} 1} we get  
	\begin{align*}
		\|o^N(t,x)-o^{M,N}(t,x) \|_{L^{q}(\Omega,\mathbb{R})}\leq&   C \bigg(\int_{0}^{t}\bigg(\int_{\mathcal{O}}  \left| G_{t-\lfloor s\rfloor}^N(x, y)-G_{t-s}^N(x, y) \right|^\frac{1}{H_2} \mathrm{d} y\bigg)^\frac{H_2}{H_1}\mathrm{d} s\bigg)^{H_1}   \nonumber  \\
		\leq&   C \bigg(\int_{0}^{t}\bigg(\int_{\mathcal{O}}  \left| G_{t-\lfloor s\rfloor}^N(x, y)-G_{t-s}^N(x, y) \right|^2 \mathrm{d} y\bigg)^\frac{1}{2H_1}\mathrm{d}s\bigg)^{H_1}    \nonumber    \\
		\leq&   C \bigg(\int_{0}^{t}\bigg(\sum_{j=1}^{N-1}e^{-\lambda_{N,j}^2(t-s)}\big(1-e^{-\lambda_{N,j}^2(s-\lfloor s\rfloor)}\big)^2 \bigg)^\frac{1}{2H_1}\mathrm{d} s\bigg)^{H_1}     \nonumber   \\
		\leq&   C \bigg(\int_{0}^{t}(s-\lfloor s\rfloor)^\frac{\hat{\alpha}}{H_1}(t-s)^{-\frac{\alpha}{2H_1}}\bigg(\sum_{j=1}^{N-1} \lambda_{N,j}^{4\hat{\alpha}-2\alpha} \bigg)^\frac{1}{2H_1}\mathrm{d} s\bigg)^{H_1}    \nonumber    \\
		\leq& C\tau^{H_1-\frac{1}{8}-\frac{\epsilon}{8}}, \label{es:O^M,N O^N}
	\end{align*} 
	where $\alpha=2H_1-\frac{\epsilon}{8}$ and $\hat{\alpha}=H_1-\frac{1}{8}-\frac{\epsilon}{8}.$ 
	Similar to \eqref{ineq: futher property of O^N_t}, we can get 
	$$\sup_{t \in[0, T] }\left\|o^{M,N}(t, x )-o^{M,N}(t,z)\right\|_{L^{q}(\Omega,\mathbb{R})}\leq C|x-z|,  $$
	which combined with \eqref{ineq: futher property of O^N_t} yields
	\begin{equation*} 
		\begin{aligned}
			\left\|\left(o^N(t,x)-o^{M,N}(t,x)\right)-\left(o^N(t,z)-o^{M,N}(t,z)\right)\right\|_{L^{q}(\Omega,\mathbb{R})}  
			\leq C|x-z|.
		\end{aligned}
	\end{equation*} 
	Therefore an argument similar to the one used in \eqref{ineq: ||tilde{u^n}-u || 3} shows that
	\begin{equation}     \label{es:O^M,N O^N}
		\begin{aligned}
			\left\|\mathcal{L}_4\right\|_{L^p(\Omega,l^\infty_N)}
			\leq C\left( h^{-\frac{1}{q}} \tau^{H_1-\frac{1}{4}-\frac{\epsilon}{8}}+ h^{1-\epsilon}\right) 
			\leq C\left( \tau^{H_1-\frac{1}{8}-\frac{\epsilon}{2}}+ h^{1-\epsilon}\right) 
		\end{aligned}
	\end{equation} 
	Further, gathering the above estimates together and using \eqref{ineq: regularity of U^{M,N}(t,x)} and \eqref{ineq: continuous of U^{M,N}(t,x) improve} results in
	\begin{equation*} 
		\begin{aligned}  
			\left\|\tilde{U}^{M,N}_t-U^{M,N}_t\right\|_{L^p(\Omega,l^\infty_N)} 
			\leq &\left\|\mathcal{L}_1\right\|_{L^p(\Omega,l^\infty_N)}+\left\|\mathcal{L}_2\right\|_{L^p(\Omega,l^\infty_N)}+\left\|\mathcal{L}_3\right\|_{L^p(\Omega,l^\infty_N)}+\left\|\mathcal{L}_4\right\|_{L^p(\Omega,l^\infty_N)}      \\
			\leq& C\left(h^{1-\epsilon}+ \tau^{H_1-\frac{1}{8}-\frac{\epsilon}{2}} \right).
		\end{aligned}
	\end{equation*}
	This completes the proof of Proposition \ref{prop: ||tilde{u^{M,N}}-u^{M,N} ||}.
\end{proof}
\begin{remark}
	In the proof of Proposition \ref{prop: ||tilde{u^{M,N}}-u^{M,N} ||}  , we employ H{\"o}lder's inequality to expand the spatial integral to study  $o^N(t, x)-o^{M, N}(t, x)$. This allows the utilization of the orthogonality of the basis $\left\{e_j\right\}_{j=0}^{N-1}$, which eliminates the impact of the Hurst parameter $\mathrm{H}_2$ of the noise in space on the strong convergence rate. 
\end{remark}

\begin{proof}[Proof of Propositoin \ref{prop: ||u^N-tilde{u^{M,N}} ||}]
	Define $E^{M,N}_t:=U^N_t-\tilde{U}^{M,N}_t$, which satisfies the following equation:
	\begin{equation} \label{eq:  E^{M,N}_t_matrix}
		\left\{\begin{aligned} 
			&\mathrm{d} E^{M,N}_t+A_N^2E^{M,N}_t \mathrm{d} t=A_N\bigg[ F\left(U_t^{M,N}+E^{M,N}_t+\big(\tilde{U}^{M,N}_t-U_t^{M,N}\big)\right)-\ F\left(U_t^{M,N}\right)\bigg] \mathrm{d} t,\ t \in(0, T], \\
			&E^{M,N}_0=0.
		\end{aligned}\right.
	\end{equation} 
	
	Consequently, according to Theorem \ref{th: U^N regularity}, Theorem \ref{thm: U^{M,N}_regularity}, Lemma \ref{lem: ||tilder U^{M,N} || 1,N}, Proposition \ref{propo: err problem 1} and Proposition \ref{prop: ||tilde{u^{M,N}}-u^{M,N} ||}, for any $q>0$  we deduce 
	\begin{equation} \label{ineq: ||u^m,n-u|| 1} 
		\begin{aligned}
			\mathbb{E} \left[\left\|E_t^{M,N}\right\|_{l_N^2}^{q}\right]  
			\leq&  C  \left(h^{(1-\epsilon)q}+ \tau^{(H_1-\frac{1}{8}-\frac{\epsilon}{2})q} \right).
		\end{aligned}
	\end{equation} 
	Therefore it follows from \eqref{ineq: l_infty to l2}, \eqref{ineq: e^{-A_N^2t}(-A_N)^r to l2}, \eqref{es:O^M,N O^N}, \eqref{ineq: ||u^m,n-u|| 1}, H{\"o}lder's inequality, Theorem \ref{th: U^N regularity} and Theorem \ref{thm: U^{M,N}_regularity} that
	\begin{align*}
		&\left\|E^{M,N}_t\right\|_{L^p(\Omega,l^\infty_N)}\\
		\leq& \left\|E^{M,N}_t+\mathscr{O}_t^{M,N}-\mathscr{O}_t^{N}\right\|_{L^p(\Omega,l^\infty_N)}+
		\left\|\mathscr{O}_t^{M,N}-\mathscr{O}_t^{N}\right\|_{L^p(\Omega,l^\infty_N)}\\
		\leq&  C \left(\left\|(-A_N)^\frac{1}{2} \left(E^{M,N}_t+\mathscr{O}_t^{M,N}-\mathscr{O}_t^{N}\right)\right\|_{L^p(\Omega,l^2_N)}+\left\|E^{M,N}_t\right\|_{L^p(\Omega,l^2_N)}+\left\|\mathscr{O}_t^{M,N}-\mathscr{O}_t^{N}\right\|_{L^p(\Omega,l^\infty_N)}\right)\\
		\leq&C\left\| \int_0^t\left\| (-A_N)^\frac{3}{2}  e^{-A_N^2(t-s)}\left[F\left(U_s^{N}\right)-F\left(U_s^{M,N}\right)\right] 
		\right\|_{l^2_N} \mathrm{d} s  \right\|_{L^{p}(\Omega,\mathbb{R})} + C \left(h^{1-\epsilon}+ \tau^{H_1-\frac{1}{8}-\frac{\epsilon}{2}} \right) \\ 
		\leq&C  \int_0^t(t-s)^{-\frac{3}{4}}  \left\|1+\|U_s^{N}\|_{l_N^{\infty}}^2+\|U_s^{M,N}\|_{l_N^{\infty}}^2\right\|_{L^{2p}(\Omega,\mathbb{R)}} \|U_s^{N}-U_s^{M,N}\|_{L^{2p}(\Omega,l_N^2)}  \mathrm{d} s + C \left(h^{1-\epsilon}+ \tau^{H_1-\frac{1}{8}-\frac{\epsilon}{2}} \right)\\ 
		\leq&C  \int_0^t(t-s)^{-\frac{3}{4}} \left(\|E^{M,N}_s\|_{L^{2p} (\Omega, l_N^2 )}+\|\tilde{U}^{M,N}_s-U_s^{M,N}\|_{L^{2p} (\Omega, l_N^2 )}\right)  \mathrm{d} s +C \left(h^{1-\epsilon}+ \tau^{H_1-\frac{1}{8}-\frac{\epsilon}{2}} \right) \\ 
		\leq&C \left(h^{1-\epsilon}+ \tau^{H_1-\frac{1}{8}-\frac{\epsilon}{2}} \right),
	\end{align*}  
	which completes the proof of Propositoin \ref{prop: ||u^N-tilde{u^{M,N}} ||}.
\end{proof}
\section{Numerical example} \label{section6}
In this section, we present an example to verify the theoretical findings. 
To measure numerical errors in the mean-square sense, we use the Monte-Carlo method over $1000$ independent sample trajectories.

The error is calculated by
$$
\begin{aligned}
	&E(M,N)=\max_{0\leq i\leq M}\left(\frac{1}{1000} \sum_{k=1}^{1000} \max_{0\leq j\leq N}\left|u^{M,N}\left(t_i,x_j,\omega_k\right)-u^{M_{\text{ref }},N_{\text{ref }}}\left(t_i,x_j,\omega_k\right)\right|^2   \right)^{\frac{1}{2}},
\end{aligned}
$$
where $\{t_i,x_j\}_{0\leq i\leq M,0\leq j\leq N}$ represents the discretization of $[0,T] \times [0,\pi]$ with temporal step $T/M$ and spatial step $\pi/N$, and $u^{M,N}\left(t_i,x_j,\omega_k\right)$ is the numerical solution at $k$-th trajectory.  To verify the strong convergence rate, we simulate the reference solution by using the full discretization with $M_{\text {ref }}$ and $N_{\text {ref }}$.
\par 
Given the computational complexity associated with employing the Cholesky method to simulate the fractional Brownian sheet with two parameters, we address the following three cases separately.\\
\textbf{Case 1: }$H_1=0.5,\ H_2=0.75$. \par 
In this case, we set $N_{\text{ref }}=2^{9}$ and $M_{\text{ref }}=2^{13}$. The fractional Brownian sheet possesses temporal increment independence. Consequently,  we generate $M_{\text{ref }}$ independent fractional Brownian motions as $\left\{\frac{1}{\sqrt{\tau_{\text{ref }}}} \left(B(i\tau_{\text{ref }},x) - B((i-1)\tau_{\text{ref }},x)\right)\right\}_{i=1}^{M_{\text{ref }}}.$
These motions are then amalgamated into a fractional Brownian sheet. \\
\textbf{Case 2: }$H_1=0.75,\ H_2=0.5$. \par 
In this case, , we maintain the same $N_{\text{ref }}$ and $M_{\text{ref }}$ as in Case 1. Following a similar methodology, we generate the fractional Brownian sheet.\\
\textbf{Case 3: }$H_1=0.95,\ H_2=0.75$.\par 
In this case, we set $N_{\text{ref }}=2^{9}$ and $M_{\text{ref }}=2^{11}$. The fractional Brownian sheet is simulated using the \textsf{FieldSim} procedure, an R package developed in R and C, as introduced in \cite{Brouste2008}. 
\par 
\begin{example}\label{exm:5-1}
	We consider the SCHE driven by the fractional Brownian sheet as follows 
	\begin{equation} \label{eq: exm 1}
		\left\{
		\begin{aligned} 
			&\frac{\partial u(t,x)}{\partial t}+\Delta^2 u(t,x)=\frac{1}{3}\Delta\left(u^3(t,x)+u^2(t,x)-u(t,x)+1\right)+\frac{1}{3} \frac{\partial^2 B^H(t,x)}{\partial t \partial x}, \quad t\in(0,1],\ x\in(0,\pi), \\ 
			&\frac{\partial u(t,0)}{\partial x}=\frac{\partial u(t,\pi)}{\partial x}=\frac{\partial^3 u(t,0)}{\partial x^3} =\frac{\partial^3 u(t,\pi)}{\partial x^3}=0,\qquad\qquad \qquad \qquad\qquad \quad\ \ t\in(0,1], \\  
			& u(0, x)=\frac{1}{3}+\frac{\sqrt{3}\cos(x)}{3} ,  \quad  \qquad  \qquad\qquad \ \ \qquad\qquad\qquad \qquad\qquad\qquad\quad\ x\in(0,\pi).
		\end{aligned}
		\right.
	\end{equation} 
\end{example}

The errors and strong convergence rates of the fully discrete scheme \eqref{eq: mild_solution_full_discrete} in space and time for different Hurst parameters $H_1$ and $H_2$ are presented in Tables \ref{b1} and \ref{b2}. The results shown in Table
\ref{b1} are consistent with the theoretical results in Theorem \ref{thm: ||u^{M,N}-u||}, which demonstrate the mean-square temporal convergence rates of the full discretization achieve $H_1-\frac{1}{8}-\frac{\epsilon}{2}$. Table \ref{b2} shows the mean-square spatial convergence of the full discretization. When $H_2=\frac{1}{2}$, the spatial convergence rates achieve $(1-\epsilon)$, which are in agreement with theoretical results. When $H_2>\frac{1}{2}$, the convergence rates are slightly higher than the expected rates of $1-\epsilon$. This discrepancy may arise due to our spatial analysis depending on the expression of the mild solution, deviating from conventional finite difference method analysis. Therefore, this approach might not attain the optimal convergence rates.

\begin{table}[htbp]
	\centering
	\caption{Mean-square temporal errors and strong convergence rates with $N=N_{\text{ref }}$.} 
	\label{b1}
	\resizebox{1 \columnwidth}{!}{   \renewcommand\arraystretch{1.5}
		\begin{tabular}{c|cccccc|c} 
			\Xhline{1.5pt}
			$ (H_1,H_2){\big\backslash} M $ &  & $ 8 $ & $16$ & $32$ & $64$ & $128$ & expected rate\\
			\Xhline{0.5pt}
			& $E(M,N)$ & 0.10986 & 0.08345 & 0.06557 & 0.04953   & 0.03684   &    \\
			$ (0.5,0.75) $ & order  &   &  0.39673 &  0.34782  &  0.40458  & 0.42725  & $ 0.375 $ \\
			\Xcline{1-8}{0.5pt}
			& $E(M,N)$ & 0.03244 & 0.02129 & 0.01463 &0.00952 &  0.00618 &    \\
			$ (0.75,0.5) $ & order  &   &  0.60781 & 0.54097  &  0.61960   &0.62411 & $ 0.625 $ \\ 
			\Xcline{1-8}{0.5pt}
			& $E(M,N)$ & 0.03842  &  0.02274  & 0.01330  & 0.00762   &  0.00423   &   \\  
			$ (0.95,0.75) $ & order  &   &0.75694   & 0.77314   & 0.80400  & 0.84960   & $0.825 $ \\ 
			\Xhline{1.5pt}
	\end{tabular} }
\end{table}

\begin{table}[htbp]
	\centering
	\caption{Mean-square  spatial errors and strong convergence rates with $M=M_{\text{ref }}$. } 
	\label{b2}
	\resizebox{1 \columnwidth}{!}{   \renewcommand\arraystretch{1.5}
		\begin{tabular}{c|cccccc|c} 
			\Xhline{1.5pt}
			$ (H_1,H_2){\big\backslash} N $ &  & $ 8 $ & $16$ & $32$ & $64$ & $128$ & expected rate\\
			\Xhline{0.5pt}
			& $E(M,N)$ & 0.03651 & 0.01583 &0.00633 & 0.00261  & 0.00114  &    \\
			$ (0.5,0.75) $ & order  &   &  1.20587 &  1.32277  & 1.27629 &  1.19577  & $ 1 $ \\
			\Xcline{1-8}{0.5pt}
			& $E(M,N)$ & 0.02121 & 0.01067 & 0.00571 &  0.00263 &  0.00129  &    \\
			$ (0.75,0.5) $ & order   &  & 0.99066 & 0.90386 & 1.11485 &  1.02596 &   $1$ \\
			\Xcline{1-8}{0.5pt}
			& $E(M,N)$ & 0.01954  & 0.00833  & 0.00342  & 0.00143 &   0.00060   &   \\  
			$ (0.95,0.75) $ & order  &   &1.23106   & 1.28514 &  1.25394  & 1.26637   & $1$ \\
			\Xhline{1.5pt}
	\end{tabular} }
\end{table} 

\section{Conclusion}
In this study, we have developed and analyzed a numerical scheme for the stochastic Cahn–Hilliard equation (SCHE) driven by a fractional Brownian sheet, utilizing the finite difference method for spatial discretization and the tamed exponential Euler method for temporal discretization. The explicit nature of the tamed exponential Euler scheme, coupled with the finite difference approach, renders the numerical implementation both efficient and straightforward. Our analysis establishes a strong convergence rate of $H_1 - \frac{1}{8} - \frac{\epsilon}{2}$ in time and $1 - \epsilon$ in space, under the $L^\infty$-norm—a more rigorous metric compared to  the $L^2$-norm employed in existing literature. Importantly, our results also encompass the case of white noise, as the Hurst parameters $H_1$ and $H_2$ can each take the value $\frac{1}{2}$.

This work lays the groundwork for several promising research directions. Firstly, our numerical experiments suggest that the established convergence rate may not be optimal when $H_2 > \frac{1}{2}$, indicating a challenging avenue for further refinement. Secondly, while our current analysis is confined to one-dimensional settings, extending the methodology to multidimensional SCHEs could potentially preserve the strong convergence properties by appropriately adapting the finite difference matrix. Lastly, exploring the weak convergence properties of the proposed scheme offers another fruitful area for future investigation. We look forward to addressing these challenges in our subsequent research.


\bibliographystyle{plain}

\section*{Appendix A. Proof of Lemma 4.1}
\begin{proof}
	\textbf{(i)}  
	When $i>j$, from \eqref{eq: (-A_N)w to w_i} and  H{\"o}lder's inequality we can obtain
	\begin{equation*} \label{ineq: l_infty to l2    1}
		\begin{aligned}
			\mu_i^2=&\mu_j^2+\sum_{k=j}^{i-1}\left(	\mu_{k+1}^2-	\mu_k^2\right)\\
			\leq& \mu_j^2+\sum_{k=j}^{i-1}\left|	\mu_k+	\mu_{k+1}\right| \cdot\left|	\mu_{k+1}-	\mu_k\right| \\
			\leq& \mu_j^2+\sum_{k=1}^{N-1}\left|	\mu_k+	\mu_{k+1}\right| \cdot\left|	\mu_{k+1}-	\mu_k\right| \\
			\leq&\mu_j^2+2\|	\mu\|_{l_N^2}\left\|\left(-A_N\right)^{\frac{1}{2}} 	\mu\right\|_{l_N^2} .
		\end{aligned}
	\end{equation*} 
	This result also holds for for $i \leq j$.  Multiplying the above inequality by $h$ and summing from $j=1$ to $N$, we get
	$$
	\pi|\mu_i|^2 \leq \|\mu\|_{l_N^2}^2+2 \pi\|\mu\|_{l_N^2}\left\|\left(-A_N\right)^{\frac{1}{2}} \mu\right\|_{l_N^2}.
	$$
	It easily follows that
	\begin{equation} \label{ineq: l_infty to l2    2}
		\begin{aligned}
			\|\mu\|_{l_N^\infty}^2 \leq C\left( \|\mu\|_{l_N^2}^2+ \|\mu\|_{l_N^2}\left\|\left(-A_N\right)^{\frac{1}{2}} \mu\right\|_{l_N^2}\right),
		\end{aligned}
	\end{equation} 
	which  yields  \eqref{ineq: l_infty to l2}.\\
	
	\textbf{(ii)} When $k=1$, using \eqref{eq: (-A_N)w to w_i} and \eqref{ineq: l_infty to l2} we can obtain
	\begin{equation*} 
		\begin{aligned}
			\left| w\right|_{1,N}^2=&\frac{1}{h} \sum_{j=1}^{N-1}\left|\mu_{j+1}\nu_{j+1}-\mu_j\nu_j\right|^2 \\
			\leq&\frac{2}{h} \sum_{j=1}^{N-1}\left|\nu_{j+1}\right|^2\left|\mu_{j+1}-\mu_j\right|^2+\frac{2}{h} \sum_{j=1}^{N-1}\left|\mu_{j}\right|^2\left|\nu_{j+1}-\nu_j\right|^2 \\
			\leq&2 \|\nu\|_{l_N^\infty}^2\left\|\left(-A_N\right)^{\frac{1}{2}} \mu\right\|_{l_N^2}^2 
			+2 \|\mu\|_{l_N^\infty}^2\left\|\left(-A_N\right)^{\frac{1}{2}} \nu\right\|_{l_N^2}^2 \\
			\leq& C\|\mu\|_{1,N}\|\nu\|_{1,N}.
		\end{aligned}
	\end{equation*} 
	When $k=2$, it follows the definition of $A_N$ and \eqref{ineq: l_infty to l2} that 
	\begin{equation*} 
		\begin{aligned}
			\left|w\right|_{2,N}^2=& \frac{|\mu_2\nu_2-\mu_1\nu_1|^2}{h^3}+\frac{|\mu_N\nu_N-\mu_{N-1}\nu_{N-1}|^2}{h^3}+  
			h\sum_{j=1}^{N-1}\left|\frac{\mu_{j+1}\nu_{j+1}-2\mu_j\nu_j+\mu_{j-1}\nu_{j-1}}{h^2}\right|^2  	\\
			\leq&2h\left(|\mu_2\delta_h^2\nu_1|^2+|\nu_1\delta_h^2\mu_1|^2+|\mu_N\delta_h^2\nu_N|^2+|\nu_{N-1}\delta_h^2\mu_N|^2\right)\\
			&+3h\sum_{j=1}^{N-1}\left( \left| \mu_{j+1}\delta_h^2\nu_j\right|^2+\left| \nu_{j}\delta_h^2\mu_j\right|^2 
			+\left| \frac{(\mu_{j+1}-\mu_{j-1})(\nu_{j}-\nu_{j-1})}{h^2}   \right|^2 \right)   	\\
			\leq&3\Big(\|\mu\|_{l_N^\infty}^2 \left\|\left(-A_N\right) \nu\right\|_{l_N^2}^2+\|\nu\|_{l_N^\infty}^2 \left\|\left(-A_N\right)\mu\right\|_{l_N^2}^2\Big) +12\left\|\left(-A_N\right)^\frac{1}{2} \mu\right\|_{l_N^\infty}^2 \left\|\left(-A_N\right)^\frac{1}{2} \nu\right\|_{l_N^2}^2  	\\
			\leq&C\|\mu\|_{2,N}^2\|\nu\|_{2,N}^2.
		\end{aligned}
	\end{equation*} 
	When $k=3$,  by  the definition of $A_N$ and  \eqref{eq: (-A_N)w to w_i} one can have
	$$ 	\left|w\right|_{3,N}^2= \frac{1}{h} \sum_{j=1}^{N-1}\left|\delta_h^2 w_{j+1}-\delta_h^2w_{j+}\right|^2 =h\sum_{j=1}^{N-1}\left|\delta_h^3 w_{j+\frac{1}{2}}\right|^2, $$
	where the third order difference operator is defined  as follows 
	$$
	\delta_h^3 w_{j+\frac{1}{2}}:=  \frac{1}{h}\left(\delta_h^2 w_{j+1}-\delta_h^2w_j\right),\quad j=1,\ 2,\ \cdots,\ N-1.
	$$  
	It is easy to check that
	\begin{equation*} 
		\begin{aligned}
			& \frac{\mu_{j+2}\nu_{j+2}-3\mu_{j+1}\nu_{i+1}+3\mu_j\nu_j-\mu_{j-1}\nu_{j-1}}{h^3}\\
			=&\nu_{j+1}\delta_h^3\mu_{j+\frac{1}{2}}+\mu_j\delta_h^3v_{j+\frac{1}{2}}
			+\frac{(\nu_{j+2}-\nu_{j+1})(\delta_h^2\mu_{j+1}+\delta_h^2\mu_{j})}{h}
			+\frac{(\mu_{j+1}-\mu_{j-1})(\delta_h^2\nu_{j+1}+\delta_h^2\nu_{j})}{h}\\
			&+\frac{(\nu_j-\nu_{j-1})\delta_h^2\mu_{j}}{h}
			-\frac{(\mu_j-\mu_{j-1})\delta_h^2\nu_{j}}{h} ,\quad j=2,\ 3,\ \cdots,\ N-2.
		\end{aligned}
	\end{equation*} 
	Then it follows from \eqref{ineq: l_infty to l2} that
	\begin{align*}
		\left|w\right|_{3,N}^2=& \frac{|\mu_3\nu_3-3\mu_2\nu_2+2\mu_1\nu_1|^2}{h^5}+\frac{|2\mu_{N}\nu_{N}-3\mu_{N-1}\nu_{N-1}+\mu_{N-2}\nu_{N-2}|^2}{h^5}\\
		&+h\sum_{j=2}^{N-2}\left|\frac{\mu_{j+2}\nu_{j+2}-3\mu_{j+1}\nu_{j+1}+3\mu_j\nu_j-\mu_{j-1}\nu_{j-1}}{h^3}\right|^2  	\\
		\leq&Ch\left(|\mu_2\delta_h^3\nu_1|^2+|\nu_3\delta_h^3\mu_1|^2+|\mu_{N-1}\delta_h^3\nu_N|^2+|\nu_{N-2}\delta_h^3\mu_N|^2 \right)
		+\frac{C|\nu_3-\nu_1|^2|\delta_h^2\mu_1|^2}{h}\\
		&+\frac{C|\nu_{N-2}-\nu_N|^2|\delta_h^2\mu_N|^2}{h}+Ch\sum_{j=2}^{N-2}\left( \left| \nu_{j+1}\delta_h^3\mu_{j+\frac{1}{2}}\right|^2
		+\left| \mu_{j}\delta_h^3\nu_{j+\frac{1}{2}}\right|^2+ \left|\frac{(\nu_{j+2}-\nu_{j+1})(\delta_h^2\mu_{j+1}+\delta_h^2\mu_{j})}{h}\right|^2\right)    \\
		&+Ch\sum_{j=2}^{N-2}\left( \left|\frac{(\mu_{j+1}-\mu_{j-1})(\delta_h^2\nu_{j+1}+\delta_h^2\nu_{j})}{h}\right|^2
		+\left|\frac{(\nu_j-\nu_{j-1})\delta_h^2\mu_{j}}{h}\right|^2
		+	\left|\frac{(\mu_j-\mu_{j-1})\delta_h^2\nu_{j}}{h} \right|^2\right)  \\
		\leq&C\Big(\|u\|_{l_N^\infty}^2 \left\|\left(-A_N\right)^\frac{3}{2} \nu\right\|_{l_N^2}^2+\|\nu\|_{l_N^\infty}^2 \left\|\left(-A_N\right)^\frac{3}{2}  u\right\|_{l_N^2}^2\Big)  	\\
		&+C\Big( \left\|\left(-A_N\right)^\frac{1}{2} u\right\|_{l_N^\infty}^2 \left\|\left(-A_N\right) \nu\right\|_{l_N^2}^2
		+ \left\|\left(-A_N\right)^\frac{1}{2} \nu\right\|_{l_N^\infty}^2 \left\|\left(-A_N\right) u\right\|_{l_N^2}^2\Big) \\   
		\leq&C \|\mu\|_{3,N}^2\|\nu\|_{3,N}^2.
	\end{align*} 
	
	\par   
	\textbf{(iii)} Since
	$$
	\left\|\left(-A_N\right)^{\frac{1}{2}} \mu\right\|_{l_N^2}^2=\left\langle-A_N \mu, \mu\right\rangle_{l_N^2}\leq 
	\left\|A_N \mu\right\|_{l_N^2}\left\| \mu\right\|_{l_N^2},
	$$
	it follows from \eqref{ineq: l_infty to l2    2} that
	\begin{equation*}
		\begin{aligned}
			\|\mu\|_{l_N^6}^6 \leq & h \sum_{i=1}^{N}\left|\mu_i\right|^2\|\mu\|_{l_N^{\infty}}^4   
			\leq  C\left( \|\mu\|_{l_N^2}^6+ \|\mu\|_{l_N^2}^4\left\|\left(-A_N\right)^{\frac{1}{2}} \mu\right\|_{l_N^2}^2\right) 
			\leq  C\left( \|\mu\|_{l_N^2}^6+ \|\mu\|_{l_N^2}^5\left\|A_N \mu\right\|_{l_N^2}\right),
		\end{aligned}
	\end{equation*}
	which shows \eqref{ineq: l_6 to l2}. \\
	\par \textbf{(iv)} 	It follows from the symmetry of $A_N$, \eqref{eq: l2 to Euclidean norm} and \eqref{estimate: e^{-x} 2} with $\alpha=\frac{\gamma}{2}$ that 
	\begin{equation*}  
		\begin{aligned}
			\left\|\left(-A_N\right)^\gamma e^{-A_N^2 t} u\right\|_{l_N^2}^2
			=&h\sum_{j=0}^{N-1} \left\langle \left(-A_N\right)^\gamma e^{-A_N^2 t} u,e_j \right\rangle^2  \\
			=&h\sum_{j=0}^{N-1} \left| \lambda_{N,j}^\gamma e^{-\lambda_{N,j}^2 t}\right|^2\left\langle  u,e_j \right\rangle^2 \\
			\leq&Ct^{-\gamma}h\sum_{j=0}^{N-1} \left\langle  u,e_j \right\rangle^2  \\
			\leq  & C t^{-\gamma}\|u\|_{l_N^2}^2.
		\end{aligned}
	\end{equation*} 
	
	\par \textbf{(v)} By the symmetry of $A_N$, \eqref{eq: l2 to Euclidean norm}, \eqref{estimate: e^{-x} 2} and \eqref{estimate: e^{-x} 1}, we obtain   
	\begin{align*}
		&\left\|(-A_N)^{\gamma_1} e^{-A_N^2t}\left( I- e^{-A_N^2s}\right)\mu\right\|_{l_N^{2}}^2 \\
		=&h\sum_{j=0}^{N-1} \left\langle (-A_N)^{\gamma_1} e^{-A_N^2t}\left( I- e^{-A_N^2s}\right) \mu,e_j \right\rangle^2  \\
		=&h\sum_{j=0}^{N-1} \left| \lambda_{N,j}^{\gamma_1-\gamma_2} e^{-\lambda_{N,j}^2 t}
		\left( 1- e^{-\lambda_{N,j}^2s}\right)
		\right|^2\left\langle (-A_N)^{\gamma_2}\mu,e_j \right\rangle^2 \\
		\leq&Ct^{-2\alpha}s^{2\hat{\alpha}}h\sum_{j=0}^{N-1}\lambda_{N,j}^{2\gamma_1-2\gamma_2-4\alpha+4\hat{\alpha}}\left\langle (-A_N)^{\gamma_2} \mu,e_j \right\rangle^2  \\
		\leq  & C t^{-(\gamma_1-\gamma_2+2\gamma_3)}s^{2\gamma_3}\left\|(-A_N)^{\gamma_2}\mu\right\|_{l_N^2}^2,
	\end{align*}
	where $\alpha=-\frac{\gamma_1-\gamma_2+2\gamma_3}{2}$ and $\hat{\alpha}=\gamma_3.$ The proof is completed.
\end{proof}
\section*{Appendix B. Proof of Lemma 4.4}
\begin{proof}
	Recall $f(x)=a_0x^3+a_1x^2+a_2x+a_3$ with $a_0\in \mathbb{R}^+$ and $\ a_1,\ a_2,\ a_3\in\mathbb{R}$. 
	Let
	$\theta_{i}=(\theta_{i,1},\theta_{i,2},\cdots, \theta_{i,N})^\top$ for $i=1,2,3,$ $4$, where  $\theta_{1,j}=a_0(\mu_j^2+\mu_j\nu_j+\nu_j^2)+a_1(\mu_j+\nu_j)+a_2,\ \theta_{2,j}=\mu_j^2,\ \theta_{3,j}=\mu_j\nu_j$ and $\theta_{4,j}=\nu_j^2$ for $ j=1,2,\cdots,N$.
	Then 
	$$  \left(F_N(\mu)-F_N(\nu)\right)=(\theta_{1,1}(\mu_1-\nu_1),\theta_{1,2}(\mu_2-\nu_2),\cdots, \theta_{1,N}(\mu_N-\nu_N))^\top$$
	It follows from \eqref{ineq: uv to u v} that 
	\begin{equation*} 
		\begin{aligned}
			\left|F_N(\mu)-F_N(\nu)\right|_{k,N} 
			\leq & C \|\theta_1\|_{k,N}\|\mu-\nu\|_{k,N}\\
			\leq & C\Big(a_0( \|\theta_2\|_{k,N}+\|\theta_3\|_{k,N}+\|\theta_4\|_{k,N})+
			|a_1|(\|\mu\|_{k,N}+\|\nu\|_{k,N})+|a_2| \Big)\|\mu-\nu\|_{k,N} \\
			\leq & C\Big(1+	\|\mu\|_{k,N}^2
			+ \|\nu\|_{k,N}^2 \Big)\|\mu-\nu\|_{k,N}.
		\end{aligned}
	\end{equation*} 
	By utilizing the above inequality and letting $\nu=0$, we can obtain \eqref{ineq:  F local Lipschitz   2}.  The proof is completed.
\end{proof}

\section*{Appendix C. Proof of Lemma 4.13}
\begin{proof}  
	Based on Theorem \ref {th:  u_regularity} and Proposition \ref{lem: ||tilde{u^n}-u ||},   it is simple to show that
	\begin{equation} \label{ineq: tilde{U}^N_t l2}
		\begin{aligned}
			\mathbb{E}\left[\left\|  \tilde{U}^N_t \right\|^p_{l^2_N}\right] 
			\leq& C,\qquad \forall\ t\in[0,T].
		\end{aligned}
	\end{equation}
	In the following, we improve the moment boundedness of $ \tilde{U}^N_t$. Let $k$ be the largest integer smaller than $4H_1+H_2-1$. Due to \eqref{eq:eq:discrete_CH_FractionalNoise_SDE_matrix auxiliary process}, $ \tilde{U}^N_t$ can be written as   $$ \tilde{U}^N_t=e^{-A_N^2t}U_0+\int_0^t A_Ne^{-A_N^2(t-s)}F_N(U_s)\ \mathrm{d} s+\mathscr{O}^N_t.$$  
	For any $t\in[0,T]$ and $1\leq i\leq k$, it follow from  \eqref{ineq: e^{-A_N^2t}(-A_N)^r to l2} and Lemma \ref{lem: F local Lipschitz} that 
	\begin{equation}   \label{Appendix   3.1 }
		\begin{aligned}
			\big\| (-A_N)^\frac{i}{2}\tilde{U}^N_t 		\big\|_{l^2_N } 
			\leq&\left\|e^{-A_N^2t}(-A_N)^\frac{i}{2}U_0\right\|_{l^2_N }+\int_{0}^{t}  \left\|(-A_N)^\frac{2+i}{2}e^{-A_N^2(t-s)}  F_N(U_s)\right\|_{l^2_N}\mathrm{d} s+\left\|(-A_N)^\frac{i}{2}\mathscr{O}^N_t \right\|_{l^2_N}\\ 
			\leq& C  \bigg(\left\|(-A_N)^\frac{i}{2}U_0\right\|_{l^2_N}+\int_{0}^{t}(t-s)^{-\frac{3}{4}}  \left\| (-A_N)^\frac{i-1}{2}F_N(U_s)\right\|_{l^2_N}\mathrm{d} s+\left\|(-A_N)^\frac{i}{2}\mathscr{O}^N_t \right\|_{l^2_N}\bigg)\\
			\leq& C  \bigg(\bigg(\left\|(-A_N)^2U_0\right\|_{l^2_N}+\int_{0}^{t}(t-s)^{-\frac{3}{4}}  \left(1+\left\| U_s\right\|_{i-1,N}^3\right)\mathrm{d} s+\left\|(-A_N)^\frac{i}{2}\mathscr{O}^N_t \right\|_{l^2_N}\bigg).
		\end{aligned}
	\end{equation}
	By Theorem \ref{th:  u_regularity},  \eqref{regularity of u(t,x) u_f 2.1} and \eqref{regularity of u(t,x) u_f appendix} one has 
	$$  \sup_{t \in[0, T] } \mathbb{E}\left[ \left\| U_s\right\|_{i-1,N}^q\right]\leq C, \quad  \forall 1\leq i\leq k,\ q\geq  1. $$
	Further, combing this with  \eqref{es:   U_0 4,N}, \eqref{ineq: tilde{U}^N_t l2}, \eqref{Appendix   3.1 } and  Lemma \ref{lem: 4.4 regularity of O^N_t} enables us to obtain   \eqref{ineq: ||tilder U || 1,N}. 
\end{proof}

\end{document}